\documentclass[reqno,11pt]{amsart}
\UseRawInputEncoding
\pdfoutput=1
\usepackage{amsfonts}
\usepackage{graphicx}
\usepackage{amsfonts,amsmath, amssymb}
\usepackage{bbm,mathrsfs,mathscinet}
\usepackage{enumerate}
\usepackage{tikz}
\usepackage{hyperref}
\usepackage{marginnote}
\usepackage{color}
\usepackage{mathrsfs}
\allowdisplaybreaks

\numberwithin{equation}{section}

\newcommand{\R}{\mathbb{R}}

\newtheorem{theorem}{Theorem}[section]
\newtheorem{corollary}[theorem]{Corollary}
\newtheorem{lemma}[theorem]{Lemma}
\newtheorem{proposition}[theorem]{Proposition}
\newtheorem{remark}[theorem]{Remark}

\def\v{\varepsilon}

\def\t{\theta}

\def\b{\beta}

\def\l{\lambda}

\def\r{\rho}

\def\f{\frac}

\begin{document}
	
\title[VPB with  small small entropy]{
	Global well-posedness of Vlasov-Poisson-Boltzmann equations with  neutral initial data and small relative entropy}

\author[Z.H. Jiang]{Zaihong Jiang}
\address[Z.H. Jiang]{ Department of Mathematics, Zhejiang Normal University, Jinhua 321004, P.R.~China
}
\email{jzhong@zjnu.cn}

\author[Y. Wang]{Yong Wang}
\address[Y. Wang]{Institute of Applied Mathematics, AMSS, CAS, Beijing 100190, P.R.~China}
\email{yongwang@amss.ac.cn}

\author[H. Xiong]{Hang Xiong}
\address[H. Xiong]{{School of Mathematics and Computational Science, Xiangtan University, Xiangtan 411105, P.R.~China; Institute of Applied Mathematics, AMSS, CAS, Beijing 100190, P.R.~China}}
\email{hbear0810@amss.ac.cn}

\begin{abstract}
The dynamics of dilute plasma particles such as electrons and ions can be modeled by the fundamental two species Vlasov-Poisson-Boltzmann equations, which describes mutual interactions of plasma particles through collisions in the self-induced electric field. In this paper, we are concerned with global well-posedness of mild solutions to these equations. We establish the global existence and uniqueness of mild solutions to the two species Vlasov-Poisson-Boltzmann equations on the torus for a class of initial data with bounded time-velocity-weighted $L^{\infty}$ norm under a nearly neutral condition, along with smallness conditions on the $L^1_xL^\infty_v$ norm  and defects in mass, energy and entropy. These conditions allow the initial data to exhibit large amplitude oscillations.
Due to the nonlinear effect of  electric field, we  consider the problem in $W^{1, \infty}_{x,v}$ with large amplitude data, new difficulty arises when establishing globally uniform $W^{1, \infty}_{x,v}$ bound, which has been overcome based on nearly neutral condition, time-velocity weight function and a logarithmic estimate.  Moreover,the long-time behavior of solutions in $W^{1, \infty}_{x,v}$ norm, with exponential decay rates of convergence, is also obtained.
\end{abstract}

\date{\today}
\maketitle
	
\setcounter{tocdepth}{2}
\tableofcontents
	
\thispagestyle{empty}

	
\section{Introduction}

\subsection{The Vlasov-Poisson-Boltzmann equations }

The Vlasov-Poisson-Boltzmann (VPB) equations of two species can be used to model the time evolution of dilute charged particles (e.g., electrons and ions) in the absence of an external magnetic field. In general, the VPB in the torus for two species of particles  take the form
\begin{equation}\label{Eq1.1}
	\left\{
	\begin{aligned}
		&\partial_{t} F_{+}+v \cdot \nabla_{x} F_{+}-\frac{e_{+}}{m_{+}}\nabla_{x} \phi \cdot \nabla_{v} F_{+}=Q(F_{+}, F_{+}+F_{-}), \\
		&\partial_{t} F_{-}+v \cdot \nabla_{x} F_{-}+\frac{e_{-}}{m_{-}}\nabla_{x} \phi \cdot \nabla_{v} F_{-}=Q(F_{-}, F_{+}+F_{-}), \\
		&-\Delta_{x} \phi=4\pi\int_{\mathbb{R}^{3}}(e_{+}F_{+}-e_{-}F_{-})dv,\quad \int_{\mathbb{T}^3}\phi(t,x)dx= 0, \\
		&F_{+}(0, x, v)=F_{+, 0}(x, v), \quad F_{-}(0, x, v)=F_{-, 0}(x, v),
	\end{aligned}
	\right.
\end{equation}
where $F_{+}=F_{+}(t,x, v)$ and $F_{-}=F_{-}(t,x, v)$ are the spatially periodic number density functions for the ions (+) and electrons (-), respectively, at time $t\ge 0$, position $x = (x_1, x_2, x_3)\in [-\frac{1}{2},\frac{1}{2}]^3 =\mathbb{T}^3$, velocity $v = (v_1, v_2, v_3)\in \mathbb{R}^3$. $e_{\pm}$ and $m_{\pm}$ denote the magnitude of their charges
and masses and $\phi(t,x)$ denotes the electric potential. 

The charged particles are moving under the electric force and the inter-particle collisions along their trajectories. The differences between the equations of $F_{+}$ and $F_{-}$ are that the directions of the electric force acting on ions and electrons are opposite and the collisions occur not only between the same kind of particles, but also between different types of particles.

The collision between particles is given by the standard Boltzmann collision operator $Q(G_1, G_2)$, where $G_1(v)$, $G_2(v)$ are two number density functions for two types of particles with masses $m_i$ and diameters $\sigma_i\ (i=1,2)$. By \cite[p83 and p89]{Chapman70Book}, $Q(G_1,G_2)(v)$ is defined as
\begin{align}\label{Eq1.2}
	Q(G_1, G_2)&=\frac{1}{4}(\sigma_1+\sigma_2)^2\int_{\mathbb{R}^{3}} \int_{\mathbb{S}^{2}}B(v-u,\t) \left(G_1\left(v^{\prime}\right) G_2\left(u^{\prime}\right)-G_1(v)G_2(u)\right)  \mathrm{d} \omega \mathrm{d} u\nonumber\\
	&:=Q_{\text{gain}}(G_1,G_2)-Q_{\text{loss}}(G_1,G_2),
\end{align}
where the relationship between the post-collision velocity $(v',u')$ of two particles with the pre-collision velocity $(v,u)$ is given by
\begin{equation*}
	u'=u+\frac{2m_1}{m_1+m_2}[(v-u)\cdot\omega]\omega,\quad v'=v-\frac{2m_2}{m_1+m_2}[(v-u)\cdot\omega]\omega,
\end{equation*}
for $\omega\in \mathbb{S}^2$, which can be determined by conservation laws of momentum and energy
\begin{equation*}
	m_1v+m_2u=m_1v'+m_2u',\quad \frac{1}{2}m_1|v|^2+\frac{1}{2}m_2|u|^2=\frac{1}{2}m_1|v'|^2+\frac{1}{2}m_2|u'|^2.
\end{equation*}
The Boltzmann collision kernel $B=B(v-u,\theta)$ in \eqref{Eq1.1} depends only on $|v-u|$ and $\theta$ with $\cos\theta=(v-u)\cdot \omega/|v-u|$.   Throughout this paper,  we  consider  the hard  potentials under the Grad's angular cut-off assumption, for instance,
\begin{equation*}
	B(v-u,\t)=|v-u|^{\gamma}b(\t),
\end{equation*}
with 
\begin{equation}
	0\leq\gamma\leq 1,\quad 0\leq b(\t)\lesssim  |\cos\t|.\notag
\end{equation}


In this paper, we  normalize all physical constants in the VPB  \eqref{Eq1.1} to be one. Then \eqref{Eq1.1} can be written as 
\begin{equation}\label{Eq1.4}
	\left\{
	\begin{aligned}
		&\partial_{t} F_{+}+v \cdot \nabla_{x} F_{+}-\nabla_{x} \phi \cdot \nabla_{v} F_{+}=Q(F_{+}, F_{+}+F_{-}),\\
		&\partial_{t} F_{-}+v \cdot \nabla_{x} F_{-}+\nabla_{x} \phi \cdot \nabla_{v} F_{-}=Q(F_{-}, F_{+}+F_{-}), \\
		&-\Delta_{x} \phi=\rho:=\int_{\mathbb{R}^{3}}(F_{+}-F_{-})\mathrm{d}v,\quad \int_{\mathbb{T}^3}\phi(t,x)\mathrm{d}x= 0, \\
		&F_{+}(0, x, v)=F_{+, 0}(x, v), \quad F_{-}(0, x, v)=F_{-, 0}(x, v),
	\end{aligned}
	\right.
\end{equation}
where
\begin{align}\label{Eq1.5}
	Q(G_1, G_2)&=\int_{\mathbb{R}^{3}} \int_{\mathbb{S}^{2}}B(v-u,\t) \left(G_1\left(v^{\prime}\right) G_2\left(u^{\prime}\right)-G_1(v)G_2(u)\right)  \mathrm{d} \omega \mathrm{d} u\nonumber\\
	&:=Q_{\text{gain}}(G_1,G_2)-Q_{\text{loss}}(G_1,G_2),
\end{align}
and 
\begin{equation*}
	v+u=v'+u',\quad |v|^2+|u|^2=|v'|^2+|u'|^2.
\end{equation*}

\subsection{A brief history of VPB }
The mathematical studies on the VPB originated from the pure Boltzmann equation and the Vlasov-Poisson equations. Extensive literature exists on the well-posedness of the Boltzmann equation and the Vlasov-Poisson equations; see \cite{Cercignani1994, Glassey96Book, Rein07Book, Villani2002}  and references therein. In what follows, we focus only on results that are directly relevant to our work.

For the Boltzmann equation, the local existence and uniqueness of general initial data in $L^\infty$ framework were firstly investigated by Kaniel-Shinbrot \cite{Kaniel-S78} and the global existence was later obtained by Illner-Shinbrot \cite{IIIner-S84} under an additional smallness assumption on velocity-weighted $L^\infty$ norm. It is well known that for general initial data with finite mass, energy and entropy, the global existence of renormalized solutions was proved  by Diperna-Lions \cite{D-Lion89} via the weak compactness method; however, the uniqueness of such solutions is unknown.  On the other hand, the convergence of a class of  large amplitude solutions toward the global Maxwellian with an explicit almost exponential rate in large time was also obtained by   Desvillettes-Villani \cite{Desvillettes-V05} conditionally under  some assumptions on smoothness and polynomial moment bounds of the solutions. The result has been  improved by Gualdani-Mischler-Mouhot \cite{GMM17} to derive a sharp exponential time rate.  In the perturbation framework, due to the extensive study of the linearized operator (Grad \cite{Grad63}, Ellis-Pinsky \cite{EP75}, and Baranger-Mouhot \cite{BM05}, for instance), the well-posedness theory of the Boltzmann equation is indeed well established in different kinds of settings since the pioneering work by Ukai \cite{Uk74}. For instance, the energy method in smooth Sobolev spaces was developed in Guo \cite{Guo-03} and Liu-Yang-Yu \cite{Liu-Yang-Yu04}. Another $L^2\cap L^\infty$ approach was found by Guo \cite{Guo10ARMA,Guo10QAM} even for treating the Boltzmann equation on a general bounded domain. Recently, Duan-Huang-Wang-Yang \cite{Duan17ARMA} developed a  $L^\infty_{x}L^1_{v}\cap L^\infty_{x,v}$ approach to prove the global well-posedness and uniqueness of the Boltzmann equation for a class of initial data ​​with large amplitude. For further related results, see \cite{AMUXY12, Briant16, DLX16ARMA, Duan19ARMA, E-G-K-M13CMP, GS11, Guo17IM, Kim2011CMP,  Kim2014CPDE, Ko2022JDE, MS16JDE, Mouhot04ARMA, SS14, Strain-Guo08} and references therein.

The well-posedness of the Vlasov-Poisson equations is well established. The local existence and uniqueness of general initial data in the framework of classical solutions were established by Kurth \cite{Kurth52}, Horst \cite{Horst81}, and Batt \cite{Batt77JDE}; Bardos-Degond \cite{B-D85} later proved the first global classical solution for small initial data by introducing the free streaming condition.  The development for Vlasov-Poisson system culminated in 1989 when independently and almost simultaneously two different proofs for global existence of classical solutions for general large data were given, one by Pfaffelmoser \cite{Pfaffelmoser92JDE} and one by Lions-Perthame \cite{Lions91IM}. For further results on the Vlasov-Poisson equations, see \cite{Caprino15CPDE, Caprino18JHDE, Chen16CMP}; as well as the monographs \cite{Glassey96Book, Rein07Book} and references therein.

For the VPB, the global existence of renormalized solutions with large initial data was established in  \cite{Lions91VPB}, and this result was later extended to cases with physical boundaries in  \cite{Mischler00}. Within the perturbation framework of the VPB, Guo \cite{Guo01CMP} proved the global smooth small-amplitude solutions near vacuum; it was shown in \cite{Guo02CPAM} that the VPB near Maxwellian admits a unique global classical solution in periodic box; see also \cite{Yang06ARMA,Yang06CMP} in the whole space. Subsequently, a robust energy method was developed in \cite{Guo03IM} to analyze  Vlasov-Maxwell-Boltzmann equations with the self-consistent electric and magnetic fields; see also\cite{Duan09SIAM, Duan12JDE, Duan13M3AS, Yang06ARMA, Yang06CMP, Zhang09JDE}. Notably, Duan-Yang-Zhao developed a time-velocity-weighted energy method to study the VPB  for general collision potentials \cite{Duan12JDE, Duan13M3AS}; see also \cite{Xiao13JDE, Xiao14SCI, Xiao17JFA}. Moreover, the large-time behavior of global solutions is also extensively studied by using different approaches \cite{Duan11ARMA, Guo03IM, Li2016IUMJ, Li2020ARMA,Li2021ARMA, Li08JDE, Wang13JDE, Yang06CMP,Yang11CMP}. We would like to mention $L^2\cap L^\infty$ approach  is also widely used for the  VPB  \cite{Cao19ARMA, Cao19KRM, G-J10CMP,  Li2021JDE} on hydrodynamic limit and bounded domain. For additional relevant contributions, see \cite{Chae06JDE, Deng21CMP, Deng23nonlinearity, Dong19SIAM, Duan13CMP, Wang19JDE} and references therein.

All existing studies on the VPB within the perturbation framework necessitate initial data of small amplitude near the global Maxwellian. This paper aims to establish the well-posedness of the VPB for initial data exhibiting large amplitude yet small relative entropy. Furthermore, we prove that the solutions converge exponentially to the global Maxwellian in $W^{1, \infty}_{x,v}$ norm.

\subsection{Notations}
Throughout this paper, we shall use the following conventions:
\begin{itemize}
\item 
$C$ denotes a generic positive constant and $C_{a}, C_{b}, C_{N} \ldots$ denote the generic positive constants depending on $a, b, N, \ldots$, respectively, which may vary from line to line. And we use $C_{1}$, $C_{2}$, $C_{3}$, etc., denote fixed positive constants.  $A \lesssim B$ means that there exists a constant $C>0$ such that $A \le C B$.  $A\cong B$ means that both $A \lesssim B$ and $B \lesssim A$ hold. 

\item For a real-valued function $\mathrm{g}(t,x,v)$ defined on $\mathbb{R}_{+}\times \mathbb{T}^3\times \mathbb{R}^3$,  the derivative of the i-th component $x_{i}$ of the spatial variable $x=(x_1,x_2,x_3)$ is defined as $\partial_{x_{i}} \mathrm{g}$ or $\partial_{i}\mathrm{g}$. The derivative of the i-th component $v_{i}$ of the spatial variable $v=(v_1,v_2,v_3)$ is defined as $\partial_{v_{i}} \mathrm{g}$ or $\partial^{i}\mathrm{g}$. 
$\nabla_{x}\mathrm{g}=\partial_{x} \mathrm{g}:=(\partial_1\mathrm{g},\partial_2\mathrm{g},\partial_3\mathrm{g})$ and $\nabla_{v}\mathrm{g}=\partial_{v} \mathrm{g}:=(\partial^1\mathrm{g},\partial^2\mathrm{g},\partial^3\mathrm{g})$ denote the gradient of $\mathrm{g}$ with respect to the spatial variable and velocity variable respectively. 

\item We use $[\cdot,\cdot]^{\mathrm{T}}$ to denote a column vector.
For a vector-valued function $\mathbf{g}(t,x,v)=[g_{+}, g_{-}]^{\mathrm{T}}$ defined on $\mathbb{R}_{+}\times \mathbb{T}^3\times \mathbb{R}^3$, we define $\partial_{x,v}\mathbf{g}(t,x,v)=[\partial_{x,v} g_{+}, \partial_{x,v} g_{-}]^{\mathrm{T}}$.

\item  $\|\cdot\|_{L^2}$ denotes either the standard  $L^{2}\left( \mathbb{T}_{x}^{3}\right)$-norm or $L^{2}\left( \mathbb{R}_{v}^{3}\right)$-norm or $L^{2}\left(\mathbb{T}_x^3 \times \mathbb{R}_{v}^{3}\right)$-norm. Similarly, $\|\cdot\|_{L^\infty}$ denotes either the $L^{\infty}\left( \mathbb{R}_{x}^{3}\right)$-norm or $L^{\infty}\left( \mathbb{R}_{v}^{3}\right)$-norm or  $L^{\infty}\left(\mathbb{T}_x^3 \times \mathbb{R}_{v}^{3}\right)$-norm.  We denote $\langle\cdot, \cdot\rangle$ as either the $L^{2}\left(\mathbb{T}_{x}^{3}\right)$ inner product or $L^{2}\left(\mathbb{R}_{v}^{3}\right)$ inner product or $L^{2}\left(\mathbb{T}_x^3 \times \mathbb{R}_{v}^{3}\right)$ inner product. We denote $\|\cdot\|_{\nu}:=\|\sqrt{\nu} \cdot\|_{L^2}$. 

\item For the vector-valued function $\mathbf{g}$,
we denote $\|\mathbf{g}\|_{L^2}:=(\|g_{+}\|_{L^2}^2+\|g_{-}\|_{L^2}^2)^{\f{1}{2}}$,  $\|\mathbf{g}\|_{L^\infty}:=\|g_{+}\|_{L^\infty}+\|g_{-}\|_{L^\infty}$,  $\|\mathbf{g}\|_{L^1_x L_v^{\infty}}:=\|g_{+}\|_{L^1_x L_v^{\infty}}+\|g_{-}\|_{L^1_x L_v^{\infty}}$
 and  $\|\mathbf{g}\|_{\nu}:=(\|g_{+}\|_{\nu}^2+\|g_{-}\|_{\nu}^2)^{\f{1}{2}}$.  
Moreover, we can also define $\|\partial_{x,v}\mathbf{g}\|_{L^\infty}:=\|\partial_{x,v}g_{+}\|_{L^\infty}+\|\partial_{x,v}g_{-}\|_{L^\infty}$, where $\|\partial_{x,v}g_{\pm}\|_{L^\infty}:=\sum\limits_{i=1}^{3}(\|\partial_{x_i}g_{\pm}\|_{L^\infty}+\|\partial_{v_i}g_{\pm}\|_{L^\infty})$. Similarly, for any vector-valued function $\mathbf{b}(x):=[b_1,b_1,b_3]^{\mathrm{T}}$, we can define the norm $\|\mathbf{b}\|_{L^2}:=\sum\limits_{i=1}^{3}\|b_{i}\|_{L^2}$  where $b_i(x)$ ($i=1,2,3$) is a real-valued function. 
For simplicity, we denote $\|(\varXi, \varPi)\|_{L^2}:=\|\varXi\|_{L^2}+\|\varPi\|_{L^2}$, where $\varXi$, $\varPi$ are real/vector-valued functions. 
\end{itemize}

\subsection{Main results}
In this paper, we study the classical solutions of \eqref{Eq1.4} around a normalized global Maxwellian:
\begin{equation*}
	\mu(v)= (2\pi)^{-\frac32}e^{-\frac{|v|^2}{2}}.
\end{equation*}
We define 
\begin{equation}\label{Eq1.6}
	f_+(t,x,v)=\frac{F_+(t,x,v)-\mu}{\sqrt{\mu}},\quad f_-(t,x,v)=\frac{F_-(t,x,v)-\mu}{\sqrt{\mu}}.
\end{equation}
Throughout the paper,  let's define
\begin{align*}
	\mathbf {F}(t,x,v):=[F_{+}(t,x,v),F_{-}(t,x,v)]^{\mathrm{T}},\quad \mathbf {f}(t,x,v):=[f_{+}(t,x,v),f_{-}(t,x,v)]^{\mathrm{T}}.
\end{align*}
For any given $\mathbf{g}=[g_{+},g_{-}]^{\mathrm{T}}$, the linearized collision operator is defined as 
\begin{align}\label{Eq1.7}
	\mathbf{Lg}:=[\mathrm{L}^{+}\mathbf{g}, \mathrm{L}^{-}\mathbf{g}]^{\mathrm{T}},
\end{align}
where
\begin{equation*}
	\mathrm{L}^{\pm}\mathbf{g}:=-\frac{1}{\sqrt{\mu}}\{2Q(\sqrt{\mu} g_{\pm}, \mu)+Q(\mu,\sqrt{\mu} g_{+}+\sqrt{\mu} g_{-})\}.
\end{equation*}
We split $\mathbf{L}$ in a standard way: $\mathbf{Lg}=\nu(v)\mathbf{g}-\mathbf{Kg}$. The collision frequency is
\begin{align}\label{Eq1.10}
	\nu(v):= 2\int_{\mathbb{R}^{3}} \int_{\mathbb{S}^{2}} B(v-u,\t) \mu(u) \mathrm{d} \omega \mathrm{d} u \cong (1+|v|)^{\gamma},
\end{align}
and $\mathbf{K}:=[\mathrm{K}^{+}, \mathrm{K}^{-}]^{\mathrm{T}}$ with $\mathrm{K}^{\pm}:=\mathrm{K}_{2}^{\pm}-\mathrm{K}_{1}^{\pm}$ which are defined  as
\begin{align}
	\left(\mathrm{K}_{1}^{+} \mathbf{g}\right)(v)&=\left(\mathrm{K}_{1}^{-} \mathbf{g}\right)(v):= \frac{1}{\sqrt{\mu}}Q_{\text{loss}}(\mu,\sqrt{\mu} g_{+}+\sqrt{\mu} g_{-})\nonumber\\
	&=\int_{\mathbb{R}^{3}} \int_{\mathbb{S}^{2}} B(v-u,\t)\sqrt{\mu(v) \mu(u)} (g_{+}(u)+g_{-}(u)) \mathrm{d} \omega \mathrm{d} u,\label{Eq1.11}\\
	\left(\mathrm{K}_{2}^{+} \mathbf{g}\right)(v)&:= \frac{2}{\sqrt{\mu}}Q_{\text{gain}}(\sqrt{\mu}       g_{+},\mu)+\frac{1}{\sqrt{\mu}}Q_{\text{gain}}(\mu,\sqrt{\mu} g_{+}+\sqrt{\mu}  g_{-})\nonumber\\
	&=2\int_{\mathbb{R}^{3}} \int_{\mathbb{S}^{2}} B(v-u,\t) \sqrt{\mu(u) \mu(u')}g_{+}(v') \mathrm{d} \omega \mathrm{d} u\nonumber\\
	&\quad +\int_{\mathbb{R}^{3}} \int_{\mathbb{S}^{2}} B(v-u,\t)\sqrt{\mu(u) \mu(v')}(g_{+}(u')+g_{-}(u')) \mathrm{d} \omega \mathrm{d} u,\label{Eq1.12}\\
	\left(\mathrm{K}_{2}^{-} \mathbf{g}\right)(v)&:= \frac{2}{\sqrt{\mu}}Q_{\text{gain}}(\sqrt{\mu} g_{-},\mu)+\frac{1}{\sqrt{\mu}}Q_{\text{gain}}(\mu,\sqrt{\mu} g_{+}+\sqrt{\mu} g_{-})\nonumber\\
	&=2\int_{\mathbb{R}^{3}} \int_{\mathbb{S}^{2}} B(v-u,\t) \sqrt{\mu(u) \mu(u')}g_{-}(v') \mathrm{d} \omega \mathrm{d} u\nonumber\\
	&\quad +\int_{\mathbb{R}^{3}} \int_{\mathbb{S}^{2}}B(v-u,\t) \sqrt{\mu(u) \mu(v')}(g_{+}(u')+g_{-}(u')) \mathrm{d} \omega \mathrm{d} u.\label{Eq1.13}
\end{align}
For $\mathbf{g}=[g_{+}, g_{-}]^{\mathrm{T}}$ and $\mathbf{f} = [f_{+}, f_{-}]^{\mathrm{T}}$, the nonlinear collision operator is defined as
\begin{align}\label{Eq1.14}
	\mathbf{\Gamma}(\mathbf{g},\mathbf{f}):=[\mathrm{\Gamma}^{+}(\mathbf{g},\mathbf{f}),\mathrm{\Gamma}^{-}(\mathbf{g},\mathbf{f})]^{\mathrm{T}}
\end{align}
with
\begin{equation}\label{Eq1.15}
	\mathrm{\Gamma}^{\pm}(\mathbf{g}, \mathbf{f}) :=\frac{1}{\sqrt{\mu}}\left\{Q(\sqrt{\mu} g_{\pm}, \sqrt{\mu} (f_{+}+f_{-}))\right\}
\end{equation}
Thus the VPB \eqref{Eq1.4} can be rewritten as
\begin{equation}\label{Eq1.17}
	\left\{
	\begin{aligned}
		&(\partial_t+v \cdot \nabla _x-\nabla_x \phi \cdot \nabla_v) f_{+}+\mathrm{L}^{+}\mathbf{f}= \mathrm{\Gamma}^{+}(\mathbf{f},\mathbf{f})-\frac{1}{2}\nabla_x \phi \cdot v f_{+}-\nabla_x \phi \cdot  v \sqrt{\mu}, \\
		&(\partial_t+v \cdot \nabla _x+\nabla_x \phi \cdot \nabla_v) f_{-}+\mathrm{L}^{-}\mathbf{f}= \mathrm{\Gamma}^{-}(\mathbf{f},\mathbf{f})+\frac{1}{2}\nabla_x \phi \cdot v f_{-}+\nabla_x \phi \cdot  v \sqrt{\mu}, \\
		&-\Delta_{x} \phi= \rho:=\int_{\mathbb{R}^{3}}\sqrt{\mu}(f_{+}-f_{-})\mathrm{d}v,  \quad \int_{\mathbb{T}^3}\phi(t,x)\mathrm{d}x= 0,   \\
		&f_{+}(0,x,v)=f_{+, 0}(x,v),\quad f_{-}(0,x,v)=f_{-, 0}(x,v). 
	\end{aligned}
	\right.
\end{equation}

For later use, $\eqref{Eq1.17}_{1,2}$ can be rewritten as
\begin{equation}\label{Eq1.18}
	\partial_t \mathbf{f} + v \cdot \nabla_x \mathbf{f} - \mathbf{q} (\nabla_x \phi \cdot \nabla_v) \mathbf{f} + \mathbf{q} (\frac{v}{2} \cdot \nabla_x \phi) \mathbf{f}+ \mathbf{L}\mathbf{f} = \mathbf{\Gamma}(\mathbf{f},\mathbf{f}) - \mathbf{q_1}v \cdot \nabla_x \phi \sqrt \mu,
\end{equation}
where $\mathbf{q} = \begin{bmatrix} 1 & 0 \\ 0 & -1  \end{bmatrix} $, and $\mathbf{q_1} = \begin{bmatrix} 1  \\ -1  \end{bmatrix} $.

It is well known that the the linearized collision operator $\mathbf{L}$ is non-negative.  The null space $\mathbf{\mathcal{N}}$ of $\mathbf{L}$ is the six dimensional space  (Lemma 1)
\begin{equation*}
	\mathbf{\mathcal{N}}=\operatorname{span}\left\{ \begin{bmatrix} \sqrt \mu \\ 0  \end{bmatrix}, \begin{bmatrix} 0  \\ \sqrt \mu \end{bmatrix}, \begin{bmatrix} \frac{v_i}{\sqrt 2 } \sqrt \mu \\ \frac{v_i}{\sqrt 2 } \sqrt \mu   \end{bmatrix}, \begin{bmatrix} \frac{|v|^2 - 3}{2\sqrt 2} \sqrt \mu \\ \frac{|v|^2 - 3}{2\sqrt 2} \sqrt \mu   \end{bmatrix}
	\right\}, \ (i=1,2,3).
\end{equation*}
For any fixed $(t, x)$, and any function
$$
\mathbf{g}(t, x, v)= \begin{bmatrix} g_+(t, x, v) \\ g_- (t, x, v)\end{bmatrix} ,
$$
we define $\mathbf P$ as its $v-$ projection onto $L_{v}^{2}\left(\mathbb{R}^{3}\right)$, to the null space $\bf{\mathcal{N}}$. We then decompose $\mathbf{g}(t, x, v)$ uniquely as
$$
\mathbf{g}(t, x, v)=\{\mathbf{ P g}\}(t, x, v)+\{\mathbf I-\mathbf P\} \mathbf{g}(t, x, v) ,
$$
where $\mathbf {P} \mathbf{g}$ is denoted by
\begin{equation*}
	\mathbf P \mathbf{g}(t, x, v) := \left\{ a_+(t, x) \begin{bmatrix} \sqrt \mu \\ 0  \end{bmatrix} + a_-(t, x) \begin{bmatrix} 0  \\ \sqrt \mu \end{bmatrix} + \mathbf{b}(t, x)  \cdot \frac{v}{\sqrt 2 } \begin{bmatrix} \sqrt \mu \\ \sqrt \mu  \end{bmatrix} + c(t, x)  \frac{|v|^2 - 3}{2\sqrt 2}\begin{bmatrix} \sqrt \mu \\ \sqrt \mu  \end{bmatrix}
	\right\},
\end{equation*}
where $\mathbf{b}(t,x)=[b_{1}(t,x),b_{2}(t,x),b_{3}(t,x)]^{\mathrm{T}}$.
Usually, we call $\mathbf{P} \mathbf{g}$  the hydrodynamic part of $\mathbf{g}$, and $\{\mathbf I-\mathbf P\} \mathbf{g}$  the microscopic part.

A direct calculation shows that the classical solutions $F_{\pm}(t, x, v)$ for VPB  \eqref{Eq1.4} satisfies the conservation laws of defect mass, momentum, and total energy:
\begin{align}
	&\int_{\mathbb{T}^3}\int_{\mathbb{R}^3}\{F_{\pm}(t,x,v)-\mu(v)\}\mathrm{d}v\mathrm{d}x=\int_{\mathbb{T}^3}\int_{\mathbb{R}^3}\{F_{\pm,0}(t,x,v)-\mu(v)\}\mathrm{d}v\mathrm{d}x:=M_{\pm,0}, \label{Eq2.4}  \\
	&\int_{\mathbb{T}^3}\int_{\mathbb{R}^3}v\{(F_{+}(t,x,v)-\mu(v))+(F_{-}(t,x,v)-\mu(v))\}\mathrm{d}v\mathrm{d}x\nonumber\\	=&\int_{\mathbb{T}^3}\int_{\mathbb{R}^3}v\{(F_{+,0}(t,x,v)-\mu(v))+(F_{-,0}(t,x,v)-\mu(v))\}\mathrm{d}v\mathrm{d}x:=\mathbf{J}_{0},\label{Eq2.6}\\
	&\int_{\mathbb{T}^3}\int_{\mathbb{R}^3}|v|^2\{F_{+}(t,x,v)+F_{-}(t,x,v)-2\mu\} \mathrm{d}v\mathrm{d}x+\int_{\mathbb{T}^3}|\nabla_x \phi(t,x)|^2\mathrm{d}x\nonumber\\
	=&\int_{\mathbb{T}^3}\int_{\mathbb{R}^3}|v|^2\{F_{+,0}(x,v)+F_{-,0}(x,v)-2\mu\} \mathrm{d}v\mathrm{d}x+\int_{\mathbb{T}^3}|\nabla_x \phi_0|^2\mathrm{d}x:=E_0, \label{Eq2.7}
\end{align}
as well as the inequality of defect entropy
\begin{align}\label{Eq2.8}
		&\int_{\mathbb{T}^3}\int_{\mathbb{R}^3}\{(F_{+}\ln F_{+}+F_{-}\ln F_{-})(t,x,v)-2\mu(v)\ln{\mu(v)}\}\mathrm{d}v\mathrm{d}x\nonumber\\
	&\quad\le \int_{\mathbb{T}^3}\int_{\mathbb{R}^3}\{(F_{+,0}\ln F_{+,0}+F_{-,0}\ln F_{-,0})(x,v)-2\mu(v)\ln{\mu(v)}\}\mathrm{d}v\mathrm{d}x.
\end{align}
For later use, we define the relative entropy 
\begin{align}\label{Eq2.9}
	\mathcal{E}(\mathbf{F}(t)):=&\int_{\mathbb{T}^3}\int_{\mathbb{R}^3}\left\{F_{+}(t)\ln F_{+}(t)+F_{-}(t)\ln F_{-}(t)-2\mu\ln \mu\right\}\mathrm{d}v\mathrm{d}x \nonumber\\
	-&\int_{\mathbb{T}^3}\int_{\mathbb{R}^3}(1+\ln \mu)(F_{+}(t)+F_{-}(t)-2\mu)\mathrm{d}v\mathrm{d}x+\frac{1}{2}\int_{\mathbb{T}^3}|\nabla_x \phi(t)|^2\mathrm{d}x.
\end{align}

Motivated by \cite{Duan12JDE}, we introduce a  time-velocity weight function
\begin{equation}\label{Eq1.22}
	w_\b(t, v):=
	(1+|v|^2)^{\frac{\beta}{2}}e^{\frac{ \sigma_0 }{1+t}|v|^2},
\end{equation}
for some $\b\geq 4$, $0<\sigma_0\leq \frac{1}{16}$.  Throughout the paper, we denote
\begin{align}\label{Eq1.23}
	\tilde{h}_{\pm}(t,x,v):=w_{\b_1}(t, v)f_{\pm}(t,x,v),\quad h_{\pm}(t,x,v):=w_{\b}(t, v)f_{\pm}(t,x,v).
\end{align}
Similarly, we can define $\mathbf{h}:=[h_{+},h_{-}]^{\mathrm{T}}$ and $\mathbf{\tilde{h}}:=[\tilde{h}_{+},\tilde{h}_{-}]^{\mathrm{T}}$, then define $\partial_{x,v}\mathbf{\tilde{h}}:=[\partial_{x,v}\tilde{h}_{+},\partial_{x,v}\tilde{h}_{-}]^{\mathrm{T}}$.

The main result of this paper is stated as follows:

\begin{theorem}\label{Thm1.1}
	Assume $0\leq\gamma\leq 1$, $4\leq\b_1<\b -4$ and $(M_{\pm, 0}, \mathbf{J}_{0}, E_0)=(0, \mathbf{0}, 0)$.  For given $M_0\geq 1$, suppose the initial data  $F_{\pm, 0}(x,v)=\mu(v)+\sqrt{\mu(v)}f_{\pm, 0}(x,v)\ge 0$ satisfy $\|h_{\pm, 0}\|_{L^\infty}\le M_0$, $\|\partial_{x, v}\tilde{h}_{\pm, 0}\|_{L^\infty}< +\infty$. Then  there exist small constants $\v_1$ (depending on $\gamma, \b, \b_1, M_0$) and $\v_0$ (depending on $\|\partial_{x, v}\tilde{h}_{\pm, 0}\|_{L^\infty}$) such that if
	\begin{align}\label{Eq1.24}
		\|f_{+, 0}-f_{-, 0}\|_{L^\infty}\leq \v_0, \quad	\mathcal{E}(\mathbf{F}_0)+\|\tilde{h}_{\pm, 0}\|_{L^1_xL^{\infty}_v}\le \varepsilon_1, 
	\end{align}
	the VPB \eqref{Eq1.4} admits a unique global mild solution $F_{\pm}(t,x,v)=\mu(v)+\sqrt{\mu(v)}f_{\pm}(t,x,v)\ge 0$ satisfying the conservation laws of defect mass, momentum, energy \eqref{Eq2.4}-\eqref{Eq2.7} as well as the additional defect entropy inequality \eqref{Eq2.8}, and
	\begin{align}
		&\|h_{\pm}(t)\|_{L^\infty}\le C\exp\{-\lambda_1 t\}, \label{Eq1.26} \\
		&\|f_{+}(t)-f_{-}(t)\|_{L^\infty}\le C\min\left\{\v_0 \exp\{C(1+\|\partial_{x,v}\mathbf{\tilde{h}}_0\|_{L^\infty})^2t\}, e^{-\lambda_1 t}\right\}, \label{Eq1.25} \\
		&\|\partial_{x, v}\tilde{h}_{\pm}(t)\|_{L^\infty}\le C(1+\|\partial_{x,v}\mathbf{\tilde{h}}_0\|_{L^\infty})^2\min\Big\{1, \Big(1+\ln (1+\|\partial_{x,v}\mathbf{\tilde{h}}_0\|_{L^\infty})\Big)e^{-\lambda_2t}\Big\}, \label{Eq1.27}
	\end{align}	
	for constants $\lambda_1>\lambda_2>0$,	where $C$ depends only on $\gamma, M_0, \beta_1, \beta$.	Moreover, if initial data $f_{\pm,0}$ are continuous, then  $f_{\pm}(t,x,v)$ are continuous in $[0,\infty)\times\mathbb{T}^3\times\mathbb{R}^3$.
\end{theorem}

\begin{remark}
	We point out that the initial data satisfying \eqref{Eq1.24} are  allowed to have  large amplitude oscillations in spatial variable. For instance, one may take
	\begin{align*}
		F_{+, 0}(x,v)=\rho_0(x)\mu, \quad F_{-, 0}(x,v)=(\rho_0(x)-\v_0)\mu,
	\end{align*}
	where $\r_0(x)\geq0$ is bounded and $\v_0$ sufficiently small. 
	It is straightforward to verify  
	$$	\mathcal{E}(\mathbf{F}_0)+\|\tilde{h}_{\pm, 0}\|_{L^1_xL^{\infty}_v}\lesssim \|\r_0\ln\r_0-\r_0+1\|_{L^1_x}+\|\r_0-1\|_{L^1_x}.$$ Even though 
	$ \|\r_0\ln\r_0-\r_0+1\|_{L^1_x}+\|\r_0-1\|_{L^1_x}$ is required to be small,  initial data are still allowed to have large amplitude oscillations.   
%
\end{remark}

\begin{remark}\label{Remark 1.3}
	By the Gagliardo-Nirenberg's inequality, one has
	\begin{align*}
		\|f_{\pm, 0}\|_{L_{x,v}^{\infty}}\lesssim \|f_{\pm,0}\|_{L_{x,v}^{2}}^{\frac{1}{4}}\|\nabla_{x,v}f_{\pm, 0}\|_{L_{x,v}^{\infty}}^{\frac{3}{4}}.
	\end{align*}
	To ensure this quantity is of order unity (i.e., $\|f_{\pm,0}\|_{L_{x,v}^{2}}^{\frac{1}{4}}\|\nabla_{x,v}f_{\pm,0}\|_{L_{x,v}^{\infty}}^{\frac{3}{4}}\approx 1$), it follows that
	\begin{align*}
		\|\nabla_{x,v}f_{\pm,0}\|_{L_{x,v}^{\infty}}\gtrsim \|f_{\pm,0}\|_{L_{x,v}^{2}}^{-\frac{1}{3}}.
	\end{align*}
On the other hand, the smallness of $\mathcal{E}(\mathbf{F}_0)$ guarantees that $\|f_{\pm,0}\|_{L_{x,v}^{2}}$is sufficiently small, as detailed in \eqref{Eq3.26} and \eqref{Eq3.27}. Consequently, $\|\nabla_{x,v}f_{\pm,0}\|_{L_{x,v}^{\infty}}$ necessarily become extremely large.
\end{remark}

%

\bigskip
\subsection{Main difficulties and strategy of the proof}
We make some comments on the main ideas of the proof and explain the main difficulties and techniques of the present paper.
For the global well-posedness of VPB with small relative entropy but large amplitude initial data, the difficulties arise from the nonlinearity of the characteristic flow of particles and nonlinearity of the collision operator.
 The characteristic flow  is the solution of a Hamiltonian system, which leads us to solve VPB  in $W^{1, \infty}_{x,v}$. The nonlinear effect of  collision operator brings great challenges to the large initial data problem of Boltzmann equation. As explained previously, Diperna-Lions \cite{D-Lion89} established the global existence of renormalized solutions for general large initial data, but the uniqueness of the solution is still open. Recently, Duan-Huang-Wang-Yang \cite{Duan17ARMA}  established the first unique global solution of the Boltzmann equation under the condition of small relative entropy. However, under the condition of small relative entropy, to include a large amplitude for the solution of  VPB, the derivative  must be sufficiently large, as explained in Remark \ref{Remark 1.3}, which makes the estimation of  characteristics very difficult. In the present paper, we shall consider the case of VPB with small relative entropy under nearly neutral condition.

In the following, we briefly explain the key  strategies of  proof of present paper.

 As far as we know, due to the appearance of the electric field $\nabla_{x} \phi (t,x)$,  there is no existence result even for the local smooth solution of VPB with large initial data. Specifically, it is difficult to control  $\nabla_x \phi \cdot v f_{\pm}$ in \eqref{Eq1.17} because the sign is not clear. Inspired by \cite{Duan12JDE, Duan13M3AS}, we introduce the time-velocity weight $w_\b(t, v)$, see \eqref{Eq1.22}, which helps us to control the electric field. 
 As explained in Remark \ref{Remark 1.3}, $\|\partial_{x, v}\mathbf{\tilde{h}}_{0}\|_{L^\infty}$ may be very large in our setting, then we have to establish the local existence theorem with the lifespan  depending only on $\|\mathbf{h}_0\|_{L^\infty}$ but  independent of  $\|\partial_{x, v}\mathbf{\tilde{h}}_{0}\|_{L^\infty}$, which is important  for us to derive the uniform estimates.
 Fortunately, using the classical potential analysis theory, we can obtain a logarithmic estimate on the derivative of  electric field (see Lemma \ref{Lem2.5}), which helps us to establish uniform estimates for $\|\mathbf{h}(t)\|_{L^\infty}$ and $\|\partial_{x,v}\mathbf{\tilde{h}}(t)\|_{L^\infty}$  during the time interval $[0,t_1]$ with $t_1$ depending only on $\|\mathbf{h}_0\|_{L^\infty}$,  see Proposition \ref{prop2.1} for details.

In order to extend the local solution into a global one, we need to establish some uniform {\it a priori} estimates. 
Under the {\it a priori} assumption \eqref{Eq2.21}, we can perform the $L_{x}^{\infty}L_{v}^{1}\cap L_{x,v}^{\infty}$ method  to obtain global uniform estimate on $\|\mathbf{h}\|_{L^\infty}$, see Proposition \ref{prop4.3}. Specifically, one  can first show that 
\begin{align*}
	\sup_{0\leq s\leq t}\|\mathbf{h}(s)\|_{L^\infty}
	\lesssim  \{\|\mathbf{h}_{0}\|_{L^{\infty}}+\|\mathbf{h}_{0}\|_{L^{\infty}}^2+\sqrt{\mathcal{E}(\mathbf{F}_0)}\}
	+\sup_{\substack{t_1\le s\le t\\ y \in \mathbb{T}^3}}\Big\{\|\mathbf{h}(s)\|^{\frac{3}{2}}_{L^{\infty}} \Big(\int_{\mathbb{R}_{v}^3}|\mathbf{\tilde{h}}(s,y,v)|\mathrm{d}v\Big)^{\frac{1}{2}}\Big\},
\end{align*}
see Lemma \ref{Lem3.1}. Motivated by \cite{Duan17ARMA}, we can also prove
\begin{align*}
	\int_{\mathbb{R}^3}|\mathbf{\tilde{h}}(t,x,v)|\mathrm{d}v
	&\lesssim t_{1}^{-3}\|\mathbf{\tilde{h}}_0\|_{L^1_x L_v^{\infty}}
	+\sqrt{\mathcal{E}(\mathbf{F}_0)}+\text{smallness}, \quad \forall t\geq t_1,
\end{align*}
see Lemma \ref{Lem3.2}, which yields that $\sup\limits_{ x\in\mathbb{T}^3, t\geq t_1}\int_{\mathbb{R}_{v}^3}|\mathbf{\tilde{h}}(t,x,v)|\mathrm{d}v$ can be sufficiently small, see Remark \ref{Rem4.4}.

Then, in order to close the  {\it a priori} assumption \eqref{Eq2.21}, a crucial step is to prove that the hydrodynamic part $\mathbf{P}\mathbf{f}$ can be controlled by the microscopic part $\{\mathbf{I}-\mathbf P\} \mathbf{f}$, see Lemma \ref{Lem4.1} below.
Motivated by \cite{E-G-K-M13CMP, Cao19ARMA, Cao19KRM}, the key is to construct some test function.  
However, the periodic boundary conditions  naturally require $\int_{\mathbb{T}^3} \bar{c}(t,x)\mathrm{d}x=0$ (the definition of $\bar{c}$ is shown in \eqref{Eq4.3}), so we cannot handle the estimate of $\bar{c}(t,x)$ as in  \cite{E-G-K-M13CMP, Cao19ARMA, Cao19KRM}, because
$$-\Delta\varphi_{c}(t,x)=\bar{c}(t,x), \quad \int_{\mathbb{T}^3} \varphi_{c}(t,x)\mathrm{d}x=0$$
 is ill-posed. To overcome the difficulty, by noting the conservation laws of defect total energy \eqref{Eq2.7}, we define a new function 
	$$\tilde{c}(t,x):=\bar{c}(t,x)+\frac{\sqrt{2}}{6}e^{-\lambda t}|\nabla_{x}\bar{\phi}(t,x)|^2,$$ 
	which satisfies $\int_{\mathbb{T}^3}\tilde{c}(t,x)\mathrm{d} x=0$  for all $t \geq 0$. Then we choose the test function for $\bar{c}:$
$\Psi_{c}(t,x)=(|v|^2-\beta_{c})\sqrt{\mu}v\cdot\nabla_x\varphi_{c}(t,x)$ with 
$$
-\Delta\varphi_{c}(t,x)=\tilde{c}(t,x), \quad \int_{\mathbb{T}^3} \varphi_{c}(t,x)\mathrm{d}x=0.
$$ 
Thus we  establish Lemma \ref{Lem4.1}, and  then obtain the $L^2$-$L^\infty$ exponential decay of $\mathbf{f}$, see Proposition \ref{prop5.2} and \ref{Decay}.

Next, since $\|\mathbf{h}(t)\|_{L^\infty}$ may not be small in local time, to close  {\it a priori} assumption \eqref{Eq2.21}, 
then we impose the nearly neutral condition to ensure the smallness of $\|\nabla \phi(t)\|_{L^{\infty}}$ and $\|\nabla^2 \phi(t)\|_{L^{\infty}}$,
which further yields that we have to establish  estimates for $\|\partial_{x,v}\mathbf{\tilde{h}}(t)\|_{L^\infty}$. 
Then difficulty arises from the nonlinear term 
\begin{equation}\label{Intro1}
\int_0^te^{-\nu(v)(t-s)}\left|\left(\partial_{x,v}(w_{\b_1}\mathrm{\Gamma}^{\pm}(\mathbf{f},\mathbf{f}))\right)(s,X_{\pm}(s),V_{\pm}(s))\right| \mathrm{d}s.
\end{equation}
The general view is that this term is controlled by the integral 
$$
\int_0^te^{-\nu(v)(t-s)}\nu(v)\|\mathbf{h}(s)\|_{L^\infty}\|\partial_{x,v}\mathbf{\tilde{h}}(s)\|_{L^\infty}\mathrm{d}s,$$ 
but it is hard to get global uniform estimate by using the Gronwall's inequality  due to the factor $\nu(v)$. 
 In the present paper, we control \eqref{Intro1} by
$$
\int_0^t\|\mathbf{h}(s)\|_{L^\infty}\|\partial_{x,v}\mathbf{\tilde{h}}(s)\|_{L^\infty}+e^{-\f{\nu(v)}{2}(t-s)}\nu(v) \|\partial_x\mathbf{\tilde{h}}(s)\|_{L^\infty}\|\mathbf{h}(s)\|_{L^\infty}^{\f{1}{2}} \cdot\sup\limits_{x\in\mathbb{T}^3} \Big\{\int_{\mathbb{R}_{v}^3} |\mathbf{\tilde{h}}(s,x,v)|\mathrm{d}v\Big\}^{\f{1}{2}}    \mathrm{d}s,
$$ 
which, together with \eqref{E2.2}, \eqref{E2.3}, the exponential decay of $\|\mathbf{h}(t)\|_{L^\infty}$ in Proposition \ref{Decay}, the  smallness property of $\displaystyle\sup_{x\in\mathbb{T}^3}\int_{\mathbb{R}_{v}^3} |\mathbf{\tilde{h}}(s,x,v)|\mathrm{d}v$ in Remark \ref{Rem4.4}, Gronwall's inequality, yields global uniform estimate (see \eqref{new6.12} and \eqref{new6.19}).
 Another difficult term is
\begin{align*}
&\int_0^t \int_0^{s-\f{1}{N}} e^{-\int_s^t\tilde{\nu}_{+,1}(\tau)\mathrm{d}\tau}e^{-\int_{s_1}^{s}\hat{\nu}_{+,1}(\tau_1)\mathrm{d}\tau_1}\mathrm{d}s_1ds   \nonumber\\
&\quad\times \iint_{B}\mathtt{k}_{w_{\b_1}}^{(2)}(V_{+}(s),u)\cdot\mathtt{k}_{w_{\b_1}}^{(2)}(\hat{V}_{+}(s_1),u_1)\cdot \partial_x \tilde{h}_{+}(s_1,\hat{X}_{+}(s_1),u_1)du_1du.
\end{align*}
Motivated by \cite{G-J10CMP}, we shall apply integration by parts, which reduces the control of above term as $\mathcal{E}(\mathbf{F}_0)$. Then we have to analyze the complex characteristics.
In fact, Guo-Jang \cite{G-J10CMP}  deal with the characteristics  in local time, while we need to control the characteristic flow  global in time, which is  more delicate, see Lemma \ref{Lem2.7}, Corollary \ref{Coro2.8} and Lemma \ref{Lem2.9}. With above analysis, we can obtain a globally uniform estimate of $\|\partial_{x,v}\mathbf{\tilde{h}}\|_{L^\infty}$, see section \ref{section 6} for more details.

Finally, with above preparations, we can derive that
\begin{align*}
		&\|\nabla \phi(t)\|_{L^{\infty}}\lesssim\min\Big\{\v_0 \exp\{C(1+\|\partial_{x,v}\mathbf{\tilde{h}}_0\|_{L^\infty})^2t\}, e^{-\lambda_1 t}\Big\}\leq \delta^2(1+t)^{-2},  \\
		&\|\nabla^2 \phi(t)\|_{L^{\infty}}\lesssim\min\Big\{\v_0 \mathfrak{H}(t), e^{-\lambda_2t}\Big(1+2\ln \Big(\f{1}{\delta}\Big)+\ln (1+\|\partial_{x,v}\mathbf{\tilde{h}}_0\|_{L^\infty})\Big)\Big\}\leq \delta^2(1+t)^{-\f{5}{2}},
\end{align*}
for all $t\geq 0$. Indeed, for the case of finite time, we ensure \eqref{Eq2.21} by choosing the neutral condition $\|f_{+, 0}-f_{-, 0}\|_{L^\infty}\leq \v_0$ to be small enough. For large time, \eqref{Eq2.21} is proved based on the exponential decay of $	\|\mathbf{h}(t)\|_{L^\infty}$, see Proposition \ref{Decay}. Thus we concludes  the  {\it a priori} assumption \eqref{Eq2.21}, see section \ref{section 7} for more details.


\subsection{Organization of the paper}
The paper is organized as follows. In Section 2, we present a number of useful lemmas and estimates for later use, and establish a new logarithmic estimate for the electric field. In Section 3, we establish the local existence theorem for  VPB so that the lifespan of  $L^\infty$ solution depends only on $\|\mathbf{h}_0\|_{L^\infty}$ and is independent on $\|\partial_{x, v}\mathbf{\tilde{h}}_{0}\|_{L^\infty}$. In Section 4, based  $L_{x}^{\infty}L_{v}^{1}\cap L_{x,v}^{\infty}$ method with the time-velocity-weighted function, we obtain global uniform estimate of $\|\mathbf{h}\|_{L^\infty}$. In Section 5, by estimating the hydrodynamic part via the microscopic part and applying the standard $L^2-L^\infty$ theory, we obtain the exponential decay rate of $\|\mathbf{h}\|_{L^\infty}$. In Section 6, we establish global uniform estimate of $\|\partial_{x,v}\mathbf{\tilde{h}}\|_{L^\infty} $. Finally, in Section 7, we prove the Theorem \ref{Thm1.1}.


\section{Preliminaries}
We  first present a useful result on the linear collision operator $\mathbf{L}$.  We refer to \cite[Lemma 1]{Guo03IM} for its proof.
\begin{lemma}[\cite{Guo03IM}] \label{Lem2.1}
	It holds that
	$$\langle \mathbf{L} \mathbf{g}, \mathbf{f}\rangle=\langle \mathbf{L} \mathbf{f}, \mathbf{g}\rangle, \quad \text{and} \quad  \langle \mathbf{L} \mathbf{g}, \mathbf{g}\rangle \geq 0.$$ 
	And $\mathbf{L g}=0$ if and only if
	$$ \mathbf{g}=\mathbf{P g} .$$
	Moreover, there is $\lambda_0>0$ such that
	\begin{align*}
		\langle \mathbf{L g}, \mathbf{g}\rangle \geq \lambda_0\|\{\mathbf{I}-\mathbf{P}\} \mathbf{g}\|_{\nu}^{2} .
	\end{align*}
\end{lemma}

By employing the conservation laws of defect mass \eqref{Eq2.4} and energy \eqref{Eq2.7}, combined with the defect entropy inequality \eqref{Eq2.8}, we establish the following lemma.
\begin{lemma}\label{Lem2.3}
	Let $F_{\pm}(t, x, v)$ satisfy \eqref{Eq2.4}, \eqref{Eq2.7} and \eqref{Eq2.8}, then it holds that
	\begin{align*}
		&\int_{\mathbb{T}^3} \int_{\mathbb{R}^{3}} \frac{|F_{+}(t, x, v)-\mu(v)|^{2}}{4 \mu(v)} I_{\{|F_{+}(t, x, v)-\mu(v)| \leq \mu(v)\}} \mathrm{d} v \mathrm{~d} x \nonumber\\
		&\quad+\int_{\mathbb{T}^3} \int_{\mathbb{R}^{3}} \frac{1}{4}|F_{+}(t, x, v)-\mu(v)| I_{\{|F_{+}(t, x, v)-\mu(v)| \geq \mu(v)\}} \mathrm{d} v \mathrm{~d} x \nonumber\\
		&\quad+\int_{\mathbb{T}^3} \int_{\mathbb{R}^{3}} \frac{|F_{-}(t, x, v)-\mu(v)|^{2}}{4 \mu(v)} I_{\{|F_{-}(t, x, v)-\mu(v)| \leq \mu(v)\}} \mathrm{d} v \mathrm{~d} x \nonumber\\
		&\quad+\int_{\mathbb{T}^3} \int_{\mathbb{R}^{3}} \frac{1}{4}|F_{-}(t, x, v)-\mu(v)| I_{\{|F_{-}(t, x, v)-\mu(v)| \geq \mu(v)\}} \mathrm{d} v \mathrm{~d} x \nonumber\\
		&\quad +\frac{1}{2}\int_{\mathbb{T}^3}|\nabla_x \phi(t)|^2dx
		\leq \mathcal{E}\left(\mathbf{F}_{0}\right).
	\end{align*}
\end{lemma}
\begin{proof}
	By Taylor expansion, we have
	\begin{align}\label{Eq2.10}
		F_{\pm}(t) \ln F_{\pm}(t)-\mu \ln \mu=(1+\ln \mu)[F_{\pm}(t)-\mu]+\frac{1}{2 \tilde{F}_{\pm}}|F_{\pm}(t)-\mu|^{2},
	\end{align}
	where $\tilde{F}_{\pm}$ is between $F_{\pm}(t)$ and $\mu$. It follows from \eqref{Eq2.9} and \eqref{Eq2.10} that $\mathcal{E}(F(t))\ge 0$ for any $t\ge 0$.  Noting $1+\ln \mu=-\left(\frac{3}{2} \ln (2 \pi)-1\right)-\frac{1}{2}|v|^{2}$, we have from \eqref{Eq2.4} and \eqref{Eq2.7} $-$ \eqref{Eq2.9} that
	\begin{align}\label{Eq2.11}
		&\int_{\mathbb{T}^3}\int_{\mathbb{R}^{3}} \frac{1}{2 \tilde{F_{+}}}|F_{+}(t)-\mu|^{2} \mathrm{d}v \mathrm{d}x+\int_{\mathbb{T}^3}\int_{\mathbb{R}^{3}} \frac{1}{2 \tilde{F_{-}}}|F_{-}(t)-\mu|^{2} \mathrm{d}v \mathrm{d}x+\frac{1}{2}\int_{\mathbb{T}^3}|\nabla_x \phi(t)|^2\mathrm{d}x  \nonumber\\
		=& \int_{\mathbb{T}^3} \int_{\mathbb{R}^{3}}(F_{+}(t) \ln F_{+}(t)+F_{-}(t) \ln F_{-}(t)-2\mu \ln \mu) \mathrm{d}v\mathrm{d}x\nonumber\\
		&+ \Big(\frac{3}{2} \ln (2 \pi)-1\Big) \int_{\mathbb{T}^3} \int_{\mathbb{R}^{3}}(F_{+}(t)+F_{-}(t)-2\mu)\mathrm{d} v \mathrm{d} x\nonumber\\
		&+ \frac{1}{2} \Big\{\int_{\mathbb{T}^3} \int_{\mathbb{R}^{3}}|v|^{2}(F_{+}(t)+F_{-}(t)-2\mu) \mathrm{d} v \mathrm{d} x+\int_{\mathbb{T}^3}|\nabla_x \phi(t)|^2\mathrm{d}x\Big\}  \nonumber\\
		\leq & \int_{\mathbb{T}^3} \int_{\mathbb{R}^{3}}(F_{+, 0} \ln F_{+, 0}+F_{-, 0} \ln F_{-, 0}-2\mu \ln \mu) \mathrm{d} v \mathrm{d} x
		+\Big(\frac{3}{2} \ln (2 \pi)-1\Big) (M_{+, 0}+M_{-, 0})+\frac{1}{2} E_{0}\nonumber\\
		=& \mathcal{E}(\mathbf{F}_0).
	\end{align}
	By similar arguments as in \cite[Lemma 2.7]{Duan17ARMA}, we have
	\begin{align}\label{Eq2.12}
		\frac{|F_{\pm}-\mu|}{\tilde{F}_{\pm}} \geq \frac{1}{2} \ \text{for}\ |F_{\pm}-\mu|\ge \mu,
	\end{align}
	and 
	\begin{align}\label{Eq2.13}
		\frac{1}{\tilde{F}_{\pm}}\ge \frac{1}{2\mu} \ \text{for}\ |F_{\pm}-\mu|\le \mu.
	\end{align}
	Then Lemma \ref{Lem2.3} follows from \eqref{Eq2.11}--\eqref{Eq2.13}. The proof is completed.
\end{proof}

Multiplying the first equation of $\eqref{Eq1.17}_3$ by $\phi$ and using Poincar\'{e}'s inequality, one has
\begin{align*}
	\|\nabla \phi\|_{L^2}^2\le C\|\rho\|_{L^2} \|\phi\|_{L^2}\le C\|\rho\|_{L^2} \|\nabla \phi\|_{L^2},
\end{align*}
which yields that 
\begin{align}\label{Eq2.14}
	\|\nabla \phi\|_{L^2}\le C\|\rho\|_{L^2}\le C\|f_{+}-f_{-}\|_{L^2}.
\end{align}
It follows from $\eqref{Eq1.17}_3$ and \eqref{Eq2.14} that
\begin{align*}
	\|\nabla \phi_0\|_{L^2}\le C\|\rho_0\|_{L^2}\le C
	\Big\|\int_{\mathbb{R}^{3}}\sqrt{\mu}(f_{+, 0}-f_{-, 0})\mathrm{d}v\Big\|_{L^{\infty}}\ll 1.
\end{align*}
By the standard $L^p$-theory for elliptic equations, for $p\in [2,6]$, we have 
\begin{align}\label{Eq2.15}
	\|\phi\|_{W^{2,p}}&\lesssim \|\phi\|_{L^p}+\|\rho\|_{L^p}\lesssim \|\phi\|_{H^1}+\|f_{+}-f_{-}\|_{L^{\infty}}\nonumber\\
	&\lesssim \|\nabla \phi\|_{L^2}+\|\mathbf{h}\|_{L^{\infty}}.
\end{align}
By Gagliardo-Nirenberg's inequality, for $3<p\le 6$, it follows from \eqref{Eq2.14} and \eqref{Eq2.15} that
\begin{align}\label{Eq2.16}
	\|\nabla \phi\|_{L^{\infty}}&\lesssim \|\nabla \phi\|^{\theta}_{L^2}\cdot \|\nabla^2 \phi\|^{1-\theta}_{L^p}\lesssim \|\nabla \phi\|^{\theta}_{L^2}(\|\nabla \phi\|_{L^2}+\|\mathbf{h}\|_{L^{\infty}})^{1-\theta}\nonumber\\
	&\lesssim \v \|\mathbf{h}\|_{L^{\infty}}+\frac{C}{\v}\|\nabla \phi\|_{L^{2}}.
\end{align}

\subsection{The Poisson equation in the torus  $\mathbb{T}^3$}

In this subsection, we list some useful results of the Poisson equation in  $\mathbb{T}^3$.

We first introduce some precision concerning the expression of the field. 
\begin{lemma}[\cite{B-R91}, \cite{Pallard12CPDE}] \label{Lem2.4}
	There exists a unique $\mathbb{T}^3$-periodic function $\varphi\in C^\infty(\mathbb{R}^3)$ such that\\
	(1) For every periodic $\rho\in C^1(\mathbb{T}^3)$ with $\int_{\mathbb{T}^3}\rho(x)dx= 0,$
	$$
	\Phi(x):=\int_{\mathbb{T}^3} \varphi(x-y)\rho(y)dy=\int_{\mathbb{T}^3+x'} \varphi(x-y)\rho(y)dy, \  \  x, x'\in\mathbb{R}^3.
	$$
	is the unique-periodic solution in $C^2(\mathbb{T}^3)$ of the problem:
	\begin{align}\label{Eq2.17}
		\Delta \Phi=\rho, \  \  \int_{\mathbb{T}^3} \Phi(x)dx=0.
	\end{align}
	(2) There exists a function $\varphi_0\in C^\infty(\mathbb{R}^3/\mathbb{Z}^3\cup \{(0,0,0)\}) $ such that
	\begin{align}\label{Eq2.18}
		\varphi(x)=-\frac{1}{4\pi|x|}+\varphi_0(x), \ x\in \mathbb{R}^3/\mathbb{Z}^3 .
	\end{align}
\end{lemma}

Based on Lemma \ref{Lem2.4}, we derive a logarithmic estimate for the electric field, which plays a pivotal role in controlling its nonlinear effect.

\begin{lemma}\label{Lem2.5}
Let	$\Phi$  satisfy \eqref{Eq2.17}. For any $0<d<R<\f12$, we have
	\begin{align}
	&\|\nabla_{x}\Phi\|_{L^{\infty}}\leq C\|\rho\|_{L^\infty}.  \label{Eq2.19} \\
	&\|\nabla^2_x\Phi\|_{L^{\infty}}\leq C\Big(\|\rho\|_{L^{\infty}}(1+\ln \frac{R}{d}+R^{-3})+d\|\nabla\rho\|_{L^{\infty}}\Big). \label{Eq2.20}
	\end{align}	
\end{lemma}
\begin{proof}
According to the periodic condition, it is easy to have that
\begin{equation}\label{A1}
	\Phi(x)=\int_{\mathbb{T}^3} \varphi(x-y)\rho(y)dy=\int_{\mathbb{T}^3} \varphi(y)\rho(x-y)dy.
\end{equation}
which implies that
\begin{equation}\label{A2}
	\partial_{x_i}\Phi(x)=\int_{\mathbb{T}^3} \partial_{x_i}\varphi(x-y)\rho(y)dy=\int_{\mathbb{T}^3} \partial_{y_i}\varphi(y)\rho(x-y)dy.
\end{equation}
Then \eqref{Eq2.19} follows from \eqref{Eq2.18} and \eqref{A2}.

For  $\nabla^2_x\Phi(x)$,  we have from \eqref{A2} that
\begin{align}
	\partial_{x_i}\partial_{x_j}\Phi(x)
	&=\int_{\mathbb{T}^3} \partial_{y_i}\varphi(y)\partial_{x_j}\rho(x-y)dy   
	=-\int_{\mathbb{T}^3} \partial_{y_i}\varphi(y)\partial_{y_j}\rho(x-y)dy    \nonumber\\
	&=-\Big\{\int_{\mathbb{T}^3\cap\{|y|\leq \varepsilon\}}+\int_{\mathbb{T}^3\cap\{|y|> \varepsilon\}}\Big\} \partial_{y_i}\varphi(y)\partial_{y_j}\rho(x-y)dy   \nonumber\\
	&=-\int_{\mathbb{T}^3\cap\{|y|> \varepsilon\}} \partial_{y_i}\varphi(y)\partial_{y_j}\rho(x-y)dy +O(\varepsilon) \|\nabla\rho\|_{L^{\infty}} \nonumber\\
	&=\int_{\mathbb{T}^3\cap\{|y|> \varepsilon\}} \partial_{y_j}\partial_{y_i}\varphi(y)\rho(x-y)dy+I(\rho) +O(\varepsilon) \|\nabla\rho\|_{L^{\infty}},  \label{A3}
\end{align}
where 
\begin{equation*}
	I(\rho):=\lim\limits_{\varepsilon\rightarrow0}\int_{\mathbb{T}^3\cap\{|y|=\varepsilon\}} \partial_{y_i}\varphi(y)\rho(x-y)\overrightarrow{n_j}\mathrm{d}s_y\cong \|\rho\|_{L^\infty}.
\end{equation*}
Recalling the definition of $\varphi$ in \eqref{Eq2.18}, we know that
\begin{equation*}
	\partial_{y_j}\partial_{y_i}\varphi(y)=\Big(\frac{3y_iy_j}{|y|^5}-\frac{\delta_{ij}}{|y|^3}\Big)
	+\partial_{y_j}\partial_{y_i}\varphi_0(y).
\end{equation*}
which, together with \eqref{A3}, yields that
\begin{align}
	\partial_{x_i}\partial_{x_j}\Phi(x)=&I(\rho)+\int_{\mathbb{T}^3\cap\{|y|> \varepsilon\}} \partial_{y_j}\partial_{y_i}\varphi_0(y)\rho(x-y)\mathrm{d}y +O(\varepsilon) \|\nabla\rho\|_{L^{\infty}}  \nonumber\\
	&+\int_{\mathbb{T}^3\cap\{|y|> \varepsilon\}} \Big(\frac{3y_iy_j}{|y|^5}-\frac{\delta_{ij}}{|y|^3}\Big)\rho(x-y)dy:=I .  \label{A4}
\end{align}
A direct calculation shows that
\begin{align}
	I&=\Big\{\int_{\mathbb{T}^3\cap\{ \varepsilon<|y|\leq d\}}+\int_{\mathbb{T}^3\cap\{|y|> d\}} \Big\} \Big(\frac{3y_iy_j}{|y|^5}-\frac{\delta_{ij}}{|y|^3}\Big)\rho(x-y)\mathrm{d}y\nonumber\\
	&=\int_{\mathbb{T}^3\cap\{\varepsilon<|y|\leq d\}} \Big(\frac{3y_iy_j}{|y|^5}-\frac{\delta_{ij}}{|y|^3}\Big)(\rho(x-y)-\rho(x))\mathrm{d}y\nonumber\\
	&\quad+\Big\{\int_{\mathbb{T}^3\cap\{d<|y|\leq R\}}+\int_{\mathbb{T}^3\cap\{|y|\geq R\}}\Big\} \Big(\frac{3y_iy_j}{|y|^5}-\frac{\delta_{ij}}{|y|^3}\Big)\rho(x-y)dy. \label{A5}
\end{align}
It follows from \eqref{A4} and \eqref{A5} that
\begin{align}
	\partial_{y_j}\partial_{y_i}\Phi(y)
	=&\int_{\mathbb{T}^3\cap\{\varepsilon<|y|\leq d\}} \Big(\frac{3y_iy_j}{|y|^5}-\frac{\delta_{ij}}{|y|^3}\Big)(\rho(x-y)-\rho(x))\mathrm{d}y\nonumber\\
	&+\Big\{\int_{\mathbb{T}^3\cap\{d<|y|\leq R\}}+\int_{\mathbb{T}^3\cap\{|y|\geq R\}}\Big\} \Big(\frac{3y_iy_j}{|y|^5}-\frac{\delta_{ij}}{|y|^3}\Big)\rho(x-y)dy \nonumber\\
	&+\int_{\mathbb{T}^3} \partial_{y_i}\partial_{y_j}\varphi_0(y)\rho(x-y)dy+I(\rho)+O(\varepsilon) \|\nabla\rho\|_{L^{\infty}}.  \label{A6}
\end{align}
Taking $\varepsilon$ sufficiently small, then \eqref{Eq2.20} follows from \eqref{A6}. Therefore  the proof of Lemma \ref{Lem2.5} is completed.
\end{proof}

\subsection{Characteristics}
As in \cite{B-D85,Rein07Book}, we study the characteristics under the following free streaming condition,
\begin{align}\label{Eq2.21}
	\|\nabla_x\phi(t)\|_{L^{\infty}}\le \delta (1+t)^{-2},\quad \|\nabla^2_x\phi(t)\|_{L^{\infty}}\le \delta (1+t)^{-\frac{5}{2}},
\end{align}
where $\delta$ is a sufficiently small positive constant  determined later.

Under the condition of \eqref{Eq2.21}, we can define the following backward characteristics of the ions (+) and electrons (-):
\begin{equation}\label{Eq2.22}
	\left\{
	\begin{aligned}
		&\frac{dX_{\pm}(\tau)}{\mathrm{d}\tau}=V_{\pm}(\tau), \\
		&\frac{dV_{\pm}(\tau)}{\mathrm{d}\tau}=\mp\nabla_x\phi(\tau,X_{\pm}(\tau)),     \quad \tau\in [0,t],\\
		&X_{\pm}(t;t,x,v)=x,\ V_{\pm}(t;t,x,v)=v,
	\end{aligned}
	\right.
\end{equation}
where we have used the simplified notations
\begin{align*}
	X_{\pm}(\tau)=X_{\pm}(\tau;t,x,v),\quad V_{\pm}(\tau)=V_{\pm}(\tau;t,x,v).
\end{align*}

\begin{lemma}\label{Lem2.6}
	Assume \eqref{Eq2.21}. For any $(t,x,v)$ and $0\le s_1\le s_2\le t<\infty$, it holds
	\begin{equation*}
			|V_{\pm}(s_2)-V_{\pm}(s_1)|\le C\delta.
	\end{equation*}
\end{lemma}
\begin{proof}
	Integrating $\eqref{Eq2.22}_2$ over $[s_1,s_2]$, we have from \eqref{Eq2.21} that
	\begin{align*}
		|V_{\pm}(s_2)-V_{\pm}(s_1)|=\Big|\int_{s_1}^{s_2}\nabla_x\phi(\tau,X_{+}(\tau))\mathrm{d}\tau\Big|\le C\delta \int_{0}^{\infty}(1+\tau)^{-2}\mathrm{d}\tau\le C\delta.
	\end{align*}
	Therefore the proof of Lemma \ref{Lem2.6} is completed.	
\end{proof}

\begin{lemma}[\cite{B-D85}]\label{Lem2.7}
	With the assumption of \eqref{Eq2.21}, for any  $0\le s\le t<\infty$, it holds that
	\begin{align}
		&\Big|\frac{\partial X_{\pm}(s; t, x, v)}{\partial v}+(t-s)Id\Big|+\Big|\frac{\partial V_{\pm}(s; t, x, v)}{\partial v}-Id\Big|\le C\delta (t-s),    \label{Eq2.23}\\
		&\Big|\frac{\partial X_{\pm}(s; t, x, v)}{\partial x}-Id\Big|+\Big|\frac{\partial V_{\pm}(s; t, x, v)}{\partial x}\Big|\le C\delta,    \label{Eq2.25}
	\end{align}
	for some positive constant $C>0$.
\end{lemma}

\begin{corollary}\label{Coro2.8}
	For $\delta$ suitably small and $0\le s\le t$, it holds that
	\begin{align*}
		\Big|\operatorname{det}\Big(\frac{\partial X_{\pm}(s)}{\partial v}\Big)\Big|\ge \frac{1}{2}(t-s)^3.
	\end{align*}
\end{corollary}

\begin{lemma}\label{Lem2.9}
	With the assumption of \eqref{Eq2.21}, for any  $0\le s\le t<\infty$, it holds that
	\begin{align}\label{Eq2.35}
		\sup_{0\leq s \leq t,x\in \mathbb{T}^{3}}\Big\{
		\int_{|v|\leq CN}|\partial
		_{vv}X_{\pm}(s ;t,x,v)|^{2}\mathrm{d}v\Big\} ^{1/2}  \leq C_{N}(t-s)^\f{5}{2}\sup_{0\le s\le t}\|\partial_x \mathbf{\tilde{h}}(s)\|_{L^{\infty}}.
	\end{align}
\end{lemma}

\begin{proof}
	From \eqref{Eq2.22} and a direct computation, for  $0\leq s\leq t$, one has 
	\begin{align}\label{Eq2.36}
		\partial^2_{vv}\ddot{X}_{\pm}(\tau)=\pm\nabla_x^3\phi(\tau, X_{\pm}(\tau))( \partial_vX_{\pm}(\tau))^2 
		\pm \nabla_x^2\phi(\tau, X_{\pm}(\tau))(\partial^2_{vv}X_{\pm}(\tau)),
	\end{align}
	with $\partial^2_{vv}X_{\pm}(t)=	\partial^2_{vv}\dot{X}_{\pm}(t)=0$. Integrating \eqref{Eq2.36} on $\tau$ twice, then switching the order of integration, we obtain
	\begin{align}\label{Eq2.37}
		\partial^2_{vv}X_{\pm}(s)=	\int_s^t\int_\tau^t\partial^2_{vv}\ddot{X}_{\pm}(\tau_1) \mathrm{d} \tau_1 \mathrm{d}\tau=\int_s^t\int_s^{\tau_1}\partial^2_{vv}\ddot{X}_{\pm}(\tau_1) \mathrm{d} \tau \mathrm{d}\tau_1. 
	\end{align}
	Recalling \eqref{Eq2.21}, it follows from \eqref{Eq2.23}, \eqref{Eq2.36} and \eqref{Eq2.37} that
	\begin{align*}
		|\partial^2_{vv}X_{\pm}(s)|&\leq	
		\int_s^t\int_s^{\tau_1} |\nabla_x^3\phi(\tau_1, X_{\pm}(\tau_1))||\partial_vX_{\pm}(\tau_1)|^2  \mathrm{d} \tau \mathrm{d}\tau_1
		+\delta\int_s^t(1+\tau_1)^{-\frac{3}{2}}|\partial^2_{vv}X_{\pm}(\tau_1)|  \mathrm{d}\tau_1  \nonumber\\
		&\leq C\int_s^t |\nabla_x^3\phi(\tau_1, X_{\pm}(\tau_1))|(\tau_1-s)(t-\tau_1)^2\mathrm{d}\tau_1
		+\delta\int_s^t(1+\tau_1)^{-\frac{3}{2}}|\partial^2_{vv}X_{\pm}(\tau_1)|  \mathrm{d}\tau_1    \nonumber\\
		&\leq C(t-s)\int_s^t|\nabla_x^3\phi(\tau_1, X_{\pm}(\tau_1))|(t-\tau_1)^2 \mathrm{d}\tau_1
		+\delta\int_s^t(1+\tau_1)^{-\frac{3}{2}}|\partial^2_{vv}X_{\pm}(\tau_1)|  \mathrm{d}\tau_1 .
	\end{align*}
	Gronwall's inequality implies that	for  $0\leq s\leq t$
	\begin{align*}
		|\partial^2_{vv}X_{\pm}(s; t, x, v)|\leq C(t-s)\int_s^t|\nabla_x^3\phi(\tau, X_{\pm}(\tau; t, x, v))| (t-\tau_1)^2\mathrm{d}\tau.
	\end{align*}
	We thus conclude that
	\begin{align}\label{Eq2.38}
		&\Big(\int_{|v|\leq CN}|\partial_{vv}X_{\pm}(s ;t,x,v)|^{2}\mathrm{d}v\Big) ^{1/2}   \nonumber\\
		\leq& C(t-s)\Big\{\int_{|v|\leq CN}\Big(\int_s^t|\nabla_x^3\phi(\tau, X_{\pm}(\tau; t, x, v))|(t-\tau_1)^2 \mathrm{d}\tau\Big)^2\mathrm{d}v\Big\}^{1/2}  \nonumber\\
		\leq& C(t-s) \int_s^t(t-\tau_1)^2\Big(\int_{|v|\leq CN}|\nabla_x^3\phi(\tau, X_{\pm}(\tau; t, x, v))|^2\mathrm{d}v\Big)^{1/2} \mathrm{d}\tau.
	\end{align}
	
	On the other hand, making the change of variable $y=X_{\pm}(\tau; t, x, v)$, one has from Corollary \ref{Coro2.8} that
	\begin{align}\label{Eq2.39}
		\Big(\int_{|v|\leq CN}|\nabla_x^3\phi(\tau, X_{\pm}(\tau; t, x, v))|^2\mathrm{d}v\Big)^{\f{1}{2}}  
		\leq& \Big(\int_{|y-x|\leq CN(t-\tau)}|\nabla_x^3\phi(\tau, y)|^2 |\det(\frac{\partial v}{\partial y})| \mathrm{d}y\Big)^{\f{1}{2}} \nonumber\\
		\leq& C(t-\tau)^{-\f{3}{2}} \Big(\int_{|y-x|\leq CN(t-\tau)}|\nabla_x^3\phi(\tau, y)|^2 \mathrm{d}y\Big)^{\f{1}{2}}\nonumber\\
		\leq& C_N \f{\left((t-\tau)^3+1\right)^{\f{1}{2}}}{(t-\tau)^{\f{3}{2}}}\|\nabla_x^3 \phi(\tau)\|_{L^2}\nonumber\\
		\leq& C_N (t-\tau)^{-\frac{3}{2}}\|\partial_x \mathbf{\tilde{h}}(\tau)\|_{L^{\infty}},
	\end{align}
	where we have used the standard elliptic estimate $\|\nabla_x^3 \phi(\tau)\|_{L^2}\leq C\|\nabla_x \rho(\tau)\|_{L^2}\leq C\|\partial_x \mathbf{\tilde{h}}(\tau)\|_{L^{\infty}}$ in the last inequality. Substituting \eqref{Eq2.39} into \eqref{Eq2.38}, we complete the proof of Lemma \ref{Lem2.9} .
\end{proof}

\subsection{Some useful estimates}
Lastly, for convenient use later, we will provide some very important estimates regarding the collision operators.

For any scalar function $\mathrm{g}(v)$, we define the operator $\mathrm{K}$ in scalar form:
\begin{align*}
\mathrm{K} \mathrm{g}:=\frac{1}{\sqrt{\mu}}\left\{Q_{\text{gain}}(\sqrt{\mu} \mathrm{g}, \mu)+Q(\mu,\sqrt{\mu} \mathrm{g})\right\}.
\end{align*}
The $\mathrm{K} \mathrm{g}$  can be further split into $\mathrm{K} \mathrm{g}=\mathrm{K}_2 \mathrm{g} -\mathrm{K}_1 \mathrm{g}$ with
\begin{align}
	\left(\mathrm{K}_{1} \mathrm{g}\right)(v)
	:=&\int_{\mathbb{R}^{3}} \int_{\mathbb{S}^{2}} B(v-u,\t)\sqrt{\mu(v) \mu(u)} \mathrm{g}(u) \mathrm{d} \omega \mathrm{d} u
	=\int_{\mathbb{R}^3}\mathrm{k}_1(v,\eta)\mathrm{g}(\eta)\,\mathrm{d}\eta,    \label{newK1}\\
	\left(\mathrm{K}_{2} \mathrm{g}\right)(v)
	:=&\int_{\mathbb{R}^{3}} \int_{\mathbb{S}^{2}} B(v-u,\t) \sqrt{\mu(u) \mu(u')}\mathrm{g}(v') \mathrm{d} \omega \mathrm{d} u\nonumber\\
	&\quad +\int_{\mathbb{R}^{3}} \int_{\mathbb{S}^{2}} B(v-u,\t)\sqrt{\mu(u) \mu(v')}\mathrm{g}(u') \mathrm{d} \omega \mathrm{d} u 
	=\int_{\mathbb{R}^3}\mathrm{k}_2(v,\eta)\mathrm{g}(\eta)\,\mathrm{d}\eta,    \label{newK2}
\end{align}
where $\mathrm{k}_i(v,\eta)=\mathrm{k}_i(\eta,v)$ ($i=1,2$) satisfies the following Grad's estimates \cite{Grad63}:
\begin{align}
	0\leq &\mathrm{k}_1(v,\eta)\leq C|v-\eta|^{\gamma}e^{-\f{1}{4}(|v|^2+|u|^2)},  \label{newK3}\\
	0\leq &\mathrm{k}_2(v,\eta)\leq C\f{1}{|v-\eta|}e^{-\f{|v-\eta|^2}{8}}e^{-\f{||v|^2-|\eta|^2|^2}{8|v-\eta|^2}}.     \label{newK4}
\end{align}

By comparing \eqref{Eq1.11}$-$\eqref{Eq1.13} with \eqref{newK1}$-$\eqref{newK2}, for any $\mathbf{g}=[g_{+},g_{-}]^{\mathrm{T}}$, we see
 \begin{align}\label{newK5}
 	\left(\mathrm{K}^{\pm} \mathbf{g}\right)(v)
 	&=\int_{\mathbb{R}^3}\Big(\f{3}{2}\mathrm{k}_2(v,\eta)-\mathrm{k}_1(v,\eta)\Big)g_{\pm}(\eta)\,\mathrm{d}\eta+\int_{\mathbb{R}^3}\Big(\f{1}{2}\mathrm{k}_2(v,\eta)-\mathrm{k}_1(v,\eta)\Big)g_{\mp}(\eta)\,\mathrm{d}\eta \nonumber\\
 	&=:\int_{\mathbb{R}^3}\mathtt{k}^{(2)}(v,\eta)g_{\pm}(\eta)\,\mathrm{d}\eta+\int_{\mathbb{R}^3}\mathtt{k}^{(1)}(v,\eta)g_{\mp}(\eta)\,\mathrm{d}\eta.
\end{align}

According to \eqref{newK3} $-$ \eqref{newK5}, one has
\begin{lemma}\label{Lem2.10}
	There is $C>0$ such that 
	\begin{align}\label{Eq2.40}
		0\leq |\mathtt{k}^{(i)}(v,\eta)|\leq C\Big(\f{1}{|v-\eta|}+|v-\eta|^{\gamma}\Big)e^{-\f{|v-\eta|^2}{8}}e^{-\f{||v|^2-|\eta|^2|^2}{8|v-\eta|^2}}, \quad i=1,2,
	\end{align}
for any $v, \eta\in \R^3$ with $v\neq \eta$. 
\end{lemma}
It is easy to show that
\begin{align}\label{Eq2.47}
	\partial_{x} (\mathrm{K}^{\pm}\mathbf{g})=\int_{\mathbb{R}^3}\mathtt{k}^{(2)}(v,\eta)	\partial_{x}g_{\pm}(\eta)\,\mathrm{d}\eta+\int_{\mathbb{R}^3}\mathtt{k}^{(1)}(v,\eta)	\partial_{x}g_{\mp}(\eta)\,\mathrm{d}\eta,
\end{align}
\begin{align}\label{Eq2.50}
	\partial_{v}(\mathrm{K}^{\pm}\mathbf{g})=&\int_{\mathbb{R}^3}\tilde{\mathtt{k}}^{(2)}(v,\eta)g_{\pm}(\eta)\,\mathrm{d}\eta+\int_{\mathbb{R}^3}\mathtt{k}^{(2)}(v,\eta)\partial_{v}g_{\pm}(\eta)\,\mathrm{d}\eta \nonumber\\
	&+\int_{\mathbb{R}^3}\tilde{\mathtt{k}}^{(1)}(v,\eta)g_{\mp}(\eta)\,\mathrm{d}\eta
	+\int_{\mathbb{R}^3}\mathtt{k}^{(1)}(v,\eta)\partial_{v}g_{\mp}(\eta)\,\mathrm{d}\eta,
\end{align}
where $\tilde{\mathtt{k}}^{(i)}(v,\eta)$ ($i=1,2$) satisfies for all $0<\v<1$
\begin{align}\label{Eq2.51}
	0\leq |\tilde{\mathtt{k}}^{(i)}(v,\eta)|\leq C\Big(\f{1}{|v-\eta|}+|v-\eta|\Big)e^{-\f{(1-\v)|v-\eta|^2}{8}}e^{-\f{(1-\v)||v|^2-|\eta|^2|^2}{8|v-\eta|^2}}, \quad i=1,2.
\end{align}
Moreover, using an argument similar to the one in  \cite{Guo10ARMA},  it holds for $\b \geq 0$ and $\sigma_0<\frac{1}{16}$ that
\begin{align}\label{Eq2.41}
	\int_{\mathbb{R}^3}\Big|\mathtt{k}^{(i)}(v,\eta)\f{w_{\b}(t, v)}{w_{\b}(t, \eta)}\Big|\mathrm{d}\eta+\int_{\mathbb{R}^3}\Big|\tilde{\mathtt{k}}^{(i)}(v,\eta)\f{w_{\b}(t, v)}{w_{\b}(t, \eta)}\Big| \mathrm{d}\eta\leq  C(1+|v|)^{-1} \quad i=1,2.
\end{align}

For any scalar functions $\mathrm{g}(v)$ and $\mathrm{f}(v)$, we define the nonlinear operator $\mathrm{\Gamma}$:
\begin{align*}
	\mathrm{\Gamma} (\mathrm{f},\mathrm{g}):=\frac{1}{\sqrt{\mu}}Q(\sqrt{\mu}\mathrm{f}, \sqrt{\mu} \mathrm{g})&=\frac{1}{\sqrt{\mu}}\left\{Q_{\text{gain}}(\sqrt{\mu} \mathrm{f},\sqrt{\mu} \mathrm{g})-Q_{\text{loss}}(\sqrt{\mu} \mathrm{f},\sqrt{\mu} \mathrm{g})\right\} \nonumber\\
	&=:\mathrm{\Gamma}_{\text{gain}} (\mathrm{f},\mathrm{g})-\mathrm{\Gamma}_{\text{loss}} (\mathrm{f},\mathrm{g}).
\end{align*}

The next lemma is very powerful and important.  The reader can refer to \cite{Duan19Adv} for a rigorous proof. 
\begin{lemma}\label{Lem2.11}
There is a generic constant $C>0$ such that 
\begin{align}\label{Eq2.42}
	\Big|w_{\b}\mathrm{\Gamma}_{\text{gain}}(\mathrm{f},\mathrm{g})\Big|+\Big|w_{\b}\mathrm{\Gamma}_{\text{gain}}(\mathrm{g},\mathrm{f})\Big|  \leq \frac{C \|w_{\b}\mathrm{f}\|_{L_{v}^\infty}}{1+|v|} \Big(\int_{\mathbb{R}^3}(1+|\eta|)^{4}|e^{\f{\sigma_0}{1+t}|\eta|^2}\mathrm{g}(\eta)|^2\,\mathrm{d}\eta\Big)^{\f12},
\end{align}
for all $v\in \R^3$.
In particular, for $\b\geq 4$, one has
\begin{align}\label{Eq2.43}
	\left|w_{\b}\mathrm{\Gamma}_{\text{gain}}(\mathrm{f},\mathrm{g})\right| +\left|w_{\b}\mathrm{\Gamma}_{\text{gain}}(\mathrm{g},\mathrm{f})\right|\leq C \|w_{\b}\mathrm{f}\|_{L^\infty}\|w_{\b}\mathrm{g}\|_{L^\infty},
\end{align}
for all $v\in \R^3$.	
\end{lemma}

\begin{lemma}\label{Lem2.12}
	There is a generic constant $C>0$ such that 
	\begin{align}\label{Eq2.44}
		\left|w_{\b}(t,v)\mathrm{\Gamma}_{\text{loss}}(\mathrm{f},\mathrm{g})(v)\right| \leq C\nu(v)|w_{\b}(t,v)\mathrm{f}(v)|\cdot\|w_{\b}\mathrm{g}\|_{L^\infty}^{\f{1}{2}}\Big(\int_{\mathbb{R}_{v}^3} |\mathrm{g}(v)|\mathrm{d}v\Big)^{\f{1}{2}},
	\end{align}
	for all $v\in \R^3$.
	In particular,  for $\b\geq 3$,
	\begin{align}\label{Eq2.45}
		\left|w_{\b}\mathrm{\Gamma}_{\text{loss}}(\mathrm{f},\mathrm{g})\right| \leq C\nu(v) \|w_{\b}\mathrm{f}\|_{L^\infty}\|w_{\b}\mathrm{g}\|_{L^\infty},
	\end{align}
	for all $v\in \R^3$. 
\end{lemma}
\begin{proof}
	It can be calculated directly that
	\begin{align*}
	\left|w_{\b}\mathrm{\Gamma}_{\text{loss}}(\mathrm{f},\mathrm{g})\right| 	&=\Big|\frac{w_{\b}(t,v)}{\sqrt{\mu}(v)}Q_{\text{loss}}(\sqrt{\mu}\mathrm{f},\sqrt{\mu}\mathrm{g})\Big|   \nonumber\\
		&=\Big|w_{\b}(t,v)\mathrm{f}(v)\Big|\cdot\Big|\int_{\mathbb{R}^{3}} \int_{\mathbb{S}^{2}}B(v-u,\t) \sqrt{\mu(u)}\mathrm{g}(u) \mathrm{d} \omega \mathrm{d} u\Big|   \nonumber\\
		&\leq C\nu(v)|w_{\b}(t,v)\mathrm{f}(v)|\cdot\|w_{\b}\mathrm{g}\|_{L^\infty}^{\f{1}{2}}\Big(\int_{\mathbb{R}_{v}^3} |\mathrm{g}(v)|\mathrm{d}v\Big)^{\f{1}{2}},
	\end{align*}
which yields that \eqref{Eq2.44} and \eqref{Eq2.45} hold. This completes the proof.
\end{proof}
With the help of the preceding Lemmas \ref{Lem2.11} and \ref{Lem2.12}, one can obtain that
\begin{corollary}\label{Coro2.13}
	There is a generic constant $C>0$ such that 
	\begin{align}\label{Eq2.46}
		\left|w_{\b}(t, v)\mathrm{\Gamma}^{\pm}(\mathbf{f},\mathbf{f})(v)\right| \leq C\nu(v) \|w_{\b}\mathbf{f}\|_{L^\infty}^{\f{3}{2}} \Big(\int_{\mathbb{R}^3}e^{\f{\sigma_0}{1+t}|\eta|^2}|\mathbf{f}(\eta)|\,\mathrm{d}\eta\Big)^{\f12},
	\end{align}
	for all $v\in \R^3$ and $\b\geq 4$.
\end{corollary}

Using Lemmas \ref{Lem2.11} and \ref{Lem2.12}, we will prove the following two key lemmas, which play an important role in the subsequent proof.
\begin{lemma}\label{Lem2.14}
	There is a generic constant $C>0$ such that, for $4\leq\b_1\leq\b -1$, 
\begin{align}
\left|\partial_{x}\left(w_{\b_1}\mathrm{\Gamma}_{\text{gain}}(\mathrm{f},\mathrm{g})\right)\right|&+\left|\partial_{x}\left(w_{\b_1}\mathrm{\Gamma}_{\text{gain}}(\mathrm{g},\mathrm{f})\right)\right| \leq C\|w_{\b_1}\partial_{x}\mathrm{f}\|_{L^\infty}\|w_{\b_1}\mathrm{g}\|_{L^\infty},    \label{new2.1}    \\
	\left|\partial_{x}\left(w_{\b_1}\mathrm{\Gamma}_{\text{loss}}(\mathrm{f},\mathrm{g})\right)\right|&  
\leq C\nu(v)\|w_{\b_1}\partial_{x}\mathrm{f}\|_{L^\infty}\|w_{\b_1}\mathrm{g}\|_{L^\infty}^{\f{1}{2}}
\Big(\int_{\mathbb{R}_{v}^3} |\mathrm{g}(v)|\mathrm{d}v\Big)^{\f{1}{2}}\nonumber\\
&\quad+ C\|w_{\b}\mathrm{f}\|_{L^\infty}\|w_{\b_1}\partial_x\mathrm{g}\|_{L^\infty}.   \label{new2.3}  
\end{align}
\end{lemma}
\begin{proof}
Notice that
\begin{align}\label{Eq2.48}
		\partial_{x}\left(w_{\b_1}\mathrm{\Gamma}_{\text{gain}}(\mathrm{f},\mathrm{g})\right)=w_{\b_1}\mathrm{\Gamma}_{\text{gain}}(\partial_{x}\mathrm{f},\mathrm{g})+w_{\b_1}\mathrm{\Gamma}_{\text{gain}}(\mathrm{f},\partial_{x}\mathrm{g}).
		\end{align}
	Thus, by the rotation and \eqref{Eq2.42}, \eqref{new2.1} immediately holds.

	For the loss term, it is obvious that
\begin{align}\label{new2.5}
	\partial_{x}\left(w_{\b_1}\mathrm{\Gamma}_{\text{loss}}(\mathrm{f},\mathrm{g})\right)=w_{\b_1}\mathrm{\Gamma}_{\text{loss}}(\partial_{x}\mathrm{f},\mathrm{g})+w_{\b_1}\mathrm{\Gamma}_{\text{gain}}(\mathrm{f},\partial_{x}\mathrm{g}).
\end{align}
Use \eqref{Eq2.44}, one has
\begin{align}\label{new6.8}
	\left|w_{\b_1}\mathrm{\Gamma}_{\text{loss}}(\partial_{x}\mathrm{f},\mathrm{g})\right|
	\leq \nu(v)\|w_{\b_1}\partial_{x}\mathrm{f}\|_{L^\infty}\|w_{\b_1}\mathrm{g}\|_{L^\infty}^{\f{1}{2}}\Big(\int_{\mathbb{R}_{v}^3} |\mathrm{g}(v)|\mathrm{d}v\Big)^{\f{1}{2}},
\end{align}
and
\begin{align}\label{new6.9}
	\Big|w_{\b_1}\mathrm{\Gamma}_{\text{loss}}(\mathrm{f},\partial_{x}\mathrm{g})\Big|
	&\leq \nu(v)|w_{\b_1}(t,v)\mathrm{f}(v)|\cdot\Big|\int_{\mathbb{R}^{3}} \int_{\mathbb{S}^{2}}B(v-u,\t) \sqrt{\mu(u)}\partial_x\mathrm{g}(u) \mathrm{d} \omega \mathrm{d} u\Big|   \nonumber\\
	&\leq \|w_{\b}\mathrm{f}\|_{L^\infty}\|w_{\b_1}\partial_x\mathrm{g}\|_{L^\infty}.
\end{align}
Thus we conclude \eqref{new2.3}  from \eqref{new2.5} $-$ \eqref{new6.9}. 
Therefore the proof of Lemma \ref{Lem2.14} is completed.
\end{proof}

\begin{lemma}\label{Lem2.15}
	There is a generic constant $C>0$ such that, for $4\leq\b_1\leq\b -2$, 
	\begin{align}
	&\left|\partial_{v}\left(w_{\b_1}\mathrm{\Gamma}_{\text{gain}}(\mathrm{f},\mathrm{g})\right)\right|+\left|\partial_{v}\left(w_{\b_1}\mathrm{\Gamma}_{\text{gain}}(\mathrm{g},\mathrm{f})\right)\right|  
	\leq C\|\partial_{v}(w_{\b_1}\mathrm{f})\|_{L^\infty}\|w_{\b}\mathrm{g}\|_{L^\infty}   \nonumber\\
	&+C\|\partial_{v}(w_{\b_1}\mathrm{g})\|_{L^\infty}\|w_{\b}\mathrm{f}\|_{L^\infty} 
	+C\|w_{\b}\mathrm{f}\|_{L^\infty}\|w_{\b}\mathrm{g}\|_{L^\infty},  \label{new2.2}    \\
		&\left|\partial_{v}\left(w_{\b_1}\mathrm{\Gamma}_{\text{loss}}(\mathrm{f},\mathrm{g})\right)\right| 
		\leq C\nu(v)\|\partial_{v}(w_{\b_1}\mathrm{f})\|_{L^\infty}\|w_{\b_1}\mathrm{g}\|_{L^\infty}^{\f{1}{2}}\Big(\int_{\mathbb{R}_{v}^3} |\mathrm{g}(v)w_{\b_1}(t,v)|\mathrm{d}v\Big)^{\f{1}{2}} \nonumber\\
		&+C\|\partial_{v}(w_{\b_1}\mathrm{g})\|_{L^\infty}\|w_{\b}\mathrm{f}\|_{L^\infty}
		+C\|w_{\b}\mathrm{f}\|_{L^\infty}\|w_{\b}\mathrm{g}\|_{L^\infty}.   \label{new2.4}
	\end{align}
\end{lemma}
\begin{proof}
	Noting $\partial_{v}\left(w_{\b_1}\mathrm{\Gamma}_{\text{gain}}(\mathrm{f},\mathrm{g})\right)=\partial_vw_{\b_1}\mathrm{\Gamma}_{\text{gain}}(\mathrm{f},\mathrm{g})+w_{\b_1}\partial_v\left(\mathrm{\Gamma}_{\text{gain}}(\mathrm{f},\mathrm{g})\right)$, one has
	\begin{align}\label{new6.13}
		\left|\partial_v\left(w_{\b_1}\mathrm{\Gamma}_{\text{gain}}(\mathrm{f},\mathrm{g})\right)\right| 
		\leq \left|\partial_vw_{\b_1}\mathrm{\Gamma}_{\text{gain}}(\mathrm{f},\mathrm{g})\right|+\left|w_{\b_1}\partial_v\left(\mathrm{\Gamma}_{\text{gain}}(\mathrm{f},\mathrm{g})\right)\right|.
	\end{align}	
	Using the change of variables $u'=v+z_{\perp}$, $v'=v+z_{\shortparallel}$ with $z=u-v$, $z_{\shortparallel}=(z\cdot \omega)\omega$, $z_{\perp}=z-z_{\shortparallel}$, one gets
	\begin{align*}
		\mathrm{\Gamma}_{\text{gain}}(\mathrm{f},\mathrm{g})
		&=\int_{\mathbb{R}^{3}} \int_{\mathbb{S}^{2}}B(v-u,\t) \sqrt{\mu(u)}\mathrm{f}(v')\mathrm{g}(u') \mathrm{d} \omega \mathrm{d} u\nonumber\\
		&=\int_{\mathbb{R}^{3}} \int_{\mathbb{S}^{2}}B(z,\t) \sqrt{\mu(v+z)}\mathrm{f}(v+z_{\shortparallel})\mathrm{g}(v+z_{\perp}) \mathrm{d} \omega \mathrm{d}z,
	\end{align*}	
	which yields that
	\begin{align}\label{new6.14}
		\partial_v\left(\mathrm{\Gamma}_{\text{gain}}(\mathrm{f},\mathrm{g})\right)
		=&\int_{\mathbb{R}^{3}} \int_{\mathbb{S}^{2}}B(z,\t) 	\partial_v\left(\sqrt{\mu(v+z)}\right)\mathrm{f}(v+z_{\shortparallel})\mathrm{g}(v+z_{\perp}) \mathrm{d} \omega \mathrm{d}z   \nonumber\\
		&+\int_{\mathbb{R}^{3}} \int_{\mathbb{S}^{2}}B(z,\t) 	\sqrt{\mu(v+z)}\partial_v\mathrm{f}(v+z_{\shortparallel})\mathrm{g}(v+z_{\perp}) \mathrm{d} \omega \mathrm{d}z \nonumber\\
		&+\int_{\mathbb{R}^{3}} \int_{\mathbb{S}^{2}}B(z,\t) 	\sqrt{\mu(v+z)}\mathrm{f}(v+z_{\shortparallel})\partial_v\mathrm{g}(v+z_{\perp}) \mathrm{d} \omega \mathrm{d}z.
	\end{align}	
	It follows from \eqref{new6.14} and \eqref{Eq2.43} that for $4\leq\b_1\leq\b -2$,
	\begin{align}\label{new6.15}
		\left|w_{\b_1}\partial_v\left(\mathrm{\Gamma}_{\text{gain}}(\mathrm{f},\mathrm{g})\right)\right|&\leq C\left(\|w_{\b_1}\partial_{v}\mathrm{f}\|_{L^\infty}\|w_{\b}\mathrm{g}\|_{L^\infty}
		+\|w_{\b_1}\partial_{v}g\|_{L^\infty}\|w_{\b}\mathrm{f}\|_{L^\infty}+ \|w_{\b}\mathrm{f}\|_{L^\infty}\|w_{\b}\mathrm{g}\|_{L^\infty}\right) \nonumber\\
		&\leq C\left(\|\partial_{v}(w_{\b_1}\mathrm{f})\|_{L^\infty}\|w_{\b}\mathrm{g}\|_{L^\infty}
		+\|\partial_{v}(w_{\b_1}\mathrm{g})\|_{L^\infty}\|w_{\b}\mathrm{f}\|_{L^\infty}\right)  \nonumber\\
		&\quad+C\|w_{\b}\mathrm{f}\|_{L^\infty}\|w_{\b}\mathrm{g}\|_{L^\infty}, 
	\end{align}	
where we have used the fact $ \|w_{\b_1}\partial_{v}\mathrm{f}\|_{L^\infty}\leq C(\|\partial_{v}(w_{\b_1}\mathrm{f})\|_{L^\infty}+\|w_{\b}\mathrm{f}\|_{L^\infty})$. 
From \eqref{new6.13}, \eqref{new6.15} and \eqref{Eq2.43}, we obtain \eqref{new2.2}.
	
It is clear that
\begin{align}\label{new2.6}
	\left|\partial_v\left(w_{\b_1}\mathrm{\Gamma}_{\text{loss}}(\mathrm{f},\mathrm{g})\right)\right| 
	\leq \left|\partial_vw_{\b_1}\mathrm{\Gamma}_{\text{loss}}(\mathrm{f},\mathrm{g})\right|+\left|w_{\b_1}\partial_v\left(\mathrm{\Gamma}_{\text{loss}}(\mathrm{f},\mathrm{g})\right)\right|.
\end{align}	
 Making a change of variables $z=u-v$ in \eqref{new2.6}, one has 
\begin{align*}
	\mathrm{\Gamma}_{\text{loss}}(\mathrm{f},\mathrm{g})
	=\mathrm{f}(v)\int_{\mathbb{R}^{3}} \int_{\mathbb{S}^{2}}B(z,\t) \sqrt{\mu(v+z)}\mathrm{g}(v+z) \mathrm{d} \omega \mathrm{d}z,
\end{align*}	
which yields that
\begin{align} \label{new6.16}
	\partial_{v}\left(\mathrm{\Gamma}_{\text{loss}}(\mathrm{f},\mathrm{g})\right)
	=&\partial_{v}\mathrm{f}(v)\int_{\mathbb{R}^{3}} \int_{\mathbb{S}^{2}}B(z,\t) \sqrt{\mu(v+z)}\mathrm{g}(v+z) \mathrm{d} \omega \mathrm{d}z   \nonumber\\
	&+f(v)\int_{\mathbb{R}^{3}} \int_{\mathbb{S}^{2}}B(z,\t) \partial_{v}\left(\sqrt{\mu(v+z)}\right)\mathrm{g}(v+z) \mathrm{d} \omega \mathrm{d}z \nonumber\\
	&+\mathrm{f}(v)\int_{\mathbb{R}^{3}} \int_{\mathbb{S}^{2}}B(z,\t) \sqrt{\mu(v+z)}\partial_{v}\mathrm{g}(v+z) \mathrm{d} \omega \mathrm{d}z.
\end{align}	
Using $\eqref{new6.16}$ and similar arguments as in \eqref{new6.8} and \eqref{new6.9},  one can obtain that for $4\leq\b_1\leq\b -2$,
\begin{align}\label{new6.17}
	\left|w_{\b_1}\partial_v\left(\mathrm{\Gamma}_{\text{loss}}(\mathrm{f},\mathrm{g})\right)\right|\leq& \nu(v)\|\partial_{v}(w_{\b_1}\mathrm{f})\|_{L^\infty}\|w_{\b_1}\mathrm{g}\|_{L^\infty}^{\f{1}{2}}\Big(\int_{\mathbb{R}_{v}^3} |\mathrm{g}(v)w_{\b_1}(t,v)|\mathrm{d}v\Big)^{\f{1}{2}}\nonumber\\
	&+\|\partial_{v}(w_{\b_1}\mathrm{g})\|_{L^\infty}\|w_{\b}\mathrm{f}\|_{L^\infty}
	+\|w_{\b}\mathrm{f}\|_{L^\infty}\|w_{\b}\mathrm{g}\|_{L^\infty}.
\end{align}	
We conclude \eqref{new2.4} from \eqref{new2.6}, \eqref{new6.17} and \eqref{Eq2.44}. Therefore the proof of Lemma \ref{Lem2.15} is complete.
	\end{proof}

From Lemmas \ref{Lem2.14} and \ref{Lem2.15}, we have following results.
\begin{corollary}\label{Coro2.16}
	There is a generic constant $C>0$ such that, for $4\leq\b_1\leq\b -2$, 
	\begin{align}
		\left|\partial_x\left(w_{\b_1}\mathrm{\Gamma}^{\pm}(\mathbf{f},\mathbf{f})\right)\right| \leq& \nu(v)\|w_{\b_1}\partial_{x}\mathbf{f}\|_{L^\infty}\|w_{\b_1}\mathbf{f}\|_{L^\infty}^{\f{1}{2}}\Big(\int_{\mathbb{R}_{v}^3}    |\mathbf{f}(v)w_{\b_1}(t,v)|\mathrm{d}v\Big)^{\f{1}{2}}\nonumber\\
		&+\|w_{\b}\mathbf{f}\|_{L^\infty}\|w_{\b_1}\partial_x\mathbf{f}\|_{L^\infty},   \label{new6.11}   \\
	\left|\partial_v\left(w_{\b_1}\mathrm{\Gamma}^{\pm}(\mathbf{f},\mathbf{f})\right)\right| \leq&
	\nu(v)\|\partial_{v}(w_{\b_1}\mathbf{f})\|_{L^\infty}\|w_{\b_1}\mathbf{f}\|_{L^\infty}^{\f{1}{2}}\Big(\int_{\mathbb{R}_{v}^3} |\mathbf{f}(v)w_{\b_1}(t,v)|\mathrm{d}v\Big)^{\f{1}{2}} \nonumber\\ 
	&+\|\partial_{v}(w_{\b_1}\mathbf{f})\|_{L^\infty}\|w_{\b}\mathbf{f}\|_{L^\infty}
	+\|w_{\b}\mathbf{f}\|_{L^\infty}^2.   \label{new6.18}
\end{align}
\end{corollary}

\section{Local-in-time existence}
Due to the influence of electric field, it seems that there is no reference on the existence of strong solution of VPB for large data.
To prove Theorem \ref{Thm1.1}, we need to figure out more quantitative properties of the local existence about the lifespan of local $L^\infty$ solution.  More precisely, since the derivative of the initial data is very large in our own setting, (see Remark \ref{Remark 1.3}), we have to establish the local existence with  the lifespan depending only on $\|\mathbf{h}_0\|_{L^\infty}$, but independent of $\|\partial_{x, v}\mathbf{\tilde{h}}_{0}\|_{L^\infty}$. 
Then the key point is to overcome the difficulties arising from the nonlinear terms. Fortunately, by utilizing the logarithmic estimate on  derivative of electric field, (see Lemma \ref{Lem2.5}), one can obtain the uniform estimate of  derivative with lifespan of the local  solution  depending only on $\|\mathbf{h}_0\|_{L^\infty}$.

\begin{proposition}[Local Existence]\label{prop2.1}
	Let $\b_1\geq 4$,  $\b>\b_1+4$. Assume the initial data satisfy $F_{\pm, 0}(x,v)=\mu(v)+\sqrt{\mu(v)}f_{\pm, 0}(x,v)\ge 0$, $\|h_{\pm, 0}\|_{L^\infty}<\infty$ and $\|\partial_{x, v}\tilde{h}_{\pm, 0}\|_{L^\infty}<\infty$.  Then there exists a positive time 
	\begin{align}\label{LT}
		t_1:=
		\frac{1}{16C_1C_2C_3C_4(\|\mathbf{h}_0\|_{L^\infty}+1)^2} 
	\end{align}
	such that the VPB  \eqref{Eq1.4} admits a unique solution $F_{\pm}(t,x,v)=\mu(v)+\sqrt{\mu(v)}f_{\pm}(t,x,v)\geq 0$ satisfying 
	\begin{align}
		&\sup_{0\leq t\leq t_1}\|\mathbf{h}(t)\|_{L^\infty}\leq 2C_1(1+\|\mathbf{h}_0\|_{L^\infty}), \label{E2.2} \\
		& \sup_{0\leq t\leq t_1}\|\partial_{x, v}\mathbf{\tilde{h}}(t)\|_{L^\infty}\leq C_5(1+\|\partial_{x, v}\mathbf{\tilde{h}}_{0})\|_{L^\infty})^2, \label{E2.3}
	\end{align}
	where the positive constants $C, C_i$ ($i=1,2,3,4,5$) depend only on $\b, \b_1$. Moreover,  the conservation laws of defect mass, momentum, energy \eqref{Eq2.4}-\eqref{Eq2.7} as well as the additional defect entropy inequality \eqref{Eq2.8} hold. 
	Finally, if initial data $f_{\pm,0}$ are continuous, then  $f_{\pm}(t,x,v)$ are continuous in $[0,t_1]\times\mathbb{T}^3\times\mathbb{R}^3$.
	
	Moreover, if  $\|f_{+, 0}-f_{-, 0}\|_{L^\infty}\leq \varepsilon_0$, it holds that
	\begin{equation}\label{E2.4}
		\sup_{0\leq t\leq t_1}\|f_{+, 0}-f_{-, 0}(t)\|_{L^\infty}\leq C\varepsilon_0 \exp\{(1+\|\partial_{x, v}\mathbf{\tilde{h}}_{0}\|_{L^\infty})^2\}.
	\end{equation}	 
\end{proposition}

\begin{proof}
 Since the proof is lengthy, we divide it into seven steps. 

\medskip
\noindent{ Step 1.} To consider the local existence of solutions for the VPB  \eqref{Eq1.4}, we start from  
the iteration that for $n=0,1,2,\cdots$, 

\begin{equation}\label{L1}
	\left\{
	\begin{aligned}
		&\{\partial_t+v \cdot \nabla_x-\nabla_x \phi^n \cdot \nabla_v\} F_{+}^{n+1}+\iint B(v-u, w)(F_{+}^n+F_{-}^n)(u) \mathrm{d} w \mathrm{d} u\cdot F_{+}^{n+1} \\
		&=Q_{g a i n}(F_{+}^n, F_{+}^n+F_{-}^n),  \\
		& \{\partial_t+v \cdot \nabla_x+\nabla_x \phi^n \cdot \nabla_v\} F_{-}^{n+1}+\iint B(v-u, w)(F_{+}^n+F_{-}^n)(u) \mathrm{d} w \mathrm{d} u\cdot F_{-}^{n+1} \\
		&=Q_{g a i n}(F_{-}^n, F_{+}^n+F_{-}^n), \\
		&-\Delta_{x} \phi^n=\int_{\mathbb{R}^{3}}(F_{+}^{n}-F_{-}^{n})\mathrm{d}v,       \quad \int_{\mathbb{T}^3}\phi^n \mathrm{d}x= 0, \\
		&F_{\pm}^{n+1}(t,x,v)\Big|_{t=0}=F_{\pm,0}(x,v)\geq0, \quad F_{\pm}^0(t,x,v)=\mu(v).
	\end{aligned}
	\right.
\end{equation}

Denote 
\begin{equation*}
	f_{\pm}^{n+1}=\f{F_{\pm}^{n+1}-\mu}{\sqrt{\mu}}.\notag
\end{equation*}
Then $\eqref{L1}_{1}$ and $\eqref{L1}_{2}$  can be written equivalently as 
\begin{equation}\label{L2}
	\left\{
	\begin{aligned}
		&\{\partial_t+v \cdot \nabla_x-\nabla_x \phi^n \cdot \nabla_v\} f_{+}^{n+1}+g^n f_{+}^{n+1}+\nabla_x \phi^n \cdot \frac{v}{2} f_{+}^{n+1}+v\cdot \nabla_x \phi^n \sqrt{\mu} \\
		&= \mathrm{K}^{+} \mathbf{f}^n+\mathrm{\Gamma}_{\text {gain }}^{+}(\mathbf{f}^n, \mathbf{f}^n), \\
		&\{\partial_t+v \cdot \nabla_x+\nabla_x \phi^n \cdot \nabla_v\} f_{-}^{n+1}+g^n f_{-}^{n+1}-\nabla_x \phi^n \cdot \frac{v}{2} f_{-}^{n+1}-v_{\cdot} \nabla_x \phi^n \sqrt{\mu}   \\
		&=  \mathrm{K}^{-} \mathbf{f}^n+\mathrm{\Gamma}_{\text {gain }}^{-}(\mathbf{f}^n, \mathbf{f}^n), \\
		&-\Delta_{x} \phi^n=\int_{\mathbb{R}^{3}}\sqrt{\mu}(f_{+}^{n}-f_{-}^{n})\mathrm{d}v,  \quad \int_{\mathbb{T}^3}\phi^n \mathrm{d}x= 0,  \\
		&f_{\pm}^{n+1}(0,x,v)=f_{\pm,0}(x,v),\quad f_{\pm}^{0}(0,x,v)=0,
	\end{aligned}
	\right.
\end{equation}
with
\begin{align}\label{L3}
	g^n(t, x, v) & =\iint B(v-u, \omega)\left(F_+^n+F_{-}^n\right)(t, x, u)  \mathrm{d} \omega  \mathrm{d} u    \nonumber\\
	& =\iint B(v-u, \omega)\left(2 \mu(u)+\sqrt{\mu}(u)\big(f_{+}^n(t, x, u)+f_{-}^n(t, x, u)\big)\right) \mathrm{d} \omega  \mathrm{d} u .
\end{align}

\medskip
\noindent{ Step 2.} 
Next, we shall use the induction argument on $n=0,1,\cdots$ to  prove that there exists a positive time $t_1>0$,  independent of $n$, such that \eqref{L1} ( equivalently \eqref{L2}), admits a unique mild solution on the time interval $[0,t_1]$,  with
\begin{align}\label{Pri0}
F_{\pm}^{n}(t,x,v)\geq0, \quad	\|\mathbf{h}^n(t)\|_{L^\infty}\leq 2C_1(\|\mathbf{h}_0\|_{L^\infty}+1), \quad 
	 \|\partial_{x, v} \mathbf{\tilde{h}}^n(t)\|_{L^\infty}\leq B(t), \quad \forall t\in [0,t_1],
\end{align}
where $B(t)$ is some continuous positive function. 

Firstly, we consider the positivity of $F_{\pm}^{n+1}$.  By induction on $n$,  it follows from \eqref{L1} and \eqref{L3} that $F_{\pm}^{n+1}\geq0,~n=0,1,\cdots$, if $F_{\pm}^{n}\geq0$.

Next, we consider the  uniform estimate for the approximation sequence.  And it is more convenient to use the equivalent form $f^{n+1}_{\pm}$.  To control the bad term $\nabla_x \phi^n \cdot \frac{v}{2} f_{\pm}^{n+1}$, we need the  time-dependent weight function $w_{\b}(t,v)$ defined in \eqref{Eq1.22},
and denote $h_{\pm}^{n}(t,x,v):=w_{\b}(t, v)f_{\pm}^{n}(t, x, v)$. It follows from \eqref{L2} that
\begin{align}\label{L5}
	\left\{\partial_t+v \cdot \nabla_x\mp\nabla_x\phi^n \cdot \nabla_{v}+\tilde{g}_{\pm}^{n}\right\} h_{\pm}^{n+1}=\left(\mp v \cdot \nabla_x \phi^n \sqrt{\mu} + \mathrm{K}^{\pm} \mathbf{f}^n+ \mathrm{\Gamma}_{\text {gain }}^{\pm}(\mathbf{f}^n , \mathbf{f}^n)\right)\cdot w_{\b},
\end{align}
where 
\begin{align}\label{L6}
	\tilde{g}_{ \pm}^{n}(t, x, v)=\frac{\sigma_0}{(1+t)^{2}}|v|^2 \pm \nabla_x \phi^n \cdot v\left\{\frac{1}{2}+\frac{\beta}{1+|v|^2}+\frac{2 \sigma_0}{1+t}\right\}+g^n(t, x, v).
\end{align}

We define the backward characteristics $(X_{\pm}^{n}(\tau ; t, x, v)$, $V_{\pm}^{n}(\tau ; t, x, v))$ passing though $(t, x, v)$ such that
\begin{equation}\label{L7}
	\left\{
	\begin{aligned}
		&\frac{d X_{\pm}^{n}(\tau ; t, x, v)}{\mathrm{d} \tau}=V_{\pm}^{n}(\tau ; t, x, v), \\
		&\frac{d V_{\pm}^{n}(\tau ; t, x, v)}{\mathrm{d} \tau}=\mp\nabla_{x} \phi^{n}(\tau, X_{\pm}^{n}(\tau ; t, x, v)), \\
	&X_{\pm}^{n}(t ; t, x, v)=x, \quad	V_{\pm}^{n}(t ; t, x, v)=v .
	\end{aligned}
	\right.
\end{equation}
For convenience, we abbreviate $(X_{\pm}^{n}(\tau ), V_{\pm}^{n}(\tau )):=(X_{\pm}^{n}(\tau ; t, x, v), V_{\pm}^{n}(\tau ; t, x, v))$ and $\tilde{g}_{ \pm}^{n}(s):=\tilde{g}_{ \pm}^{n}(s, X_{\pm}^{n}(s ), V_{\pm}^{n}(s ))$. Then integrating along  the characteristics, the mild solution of \eqref{L5} can be represented as
\begin{align} \label{L8}
	h_{ \pm}^{n+1}(t, x, v)&=h_{\pm, 0}\left(X_{ \pm}^n(0) , V_{ \pm}^n(0)\right) e^{-\int_0^{t} \tilde{g}_{ \pm}^{n}(\tau) \mathrm{d} \tau} \nonumber\\
	&\quad \mp \int_0^t e^{-\int_s^t \tilde{g}_{ \pm}^{n}(\tau) \mathrm{d} \tau}\left(v \cdot \nabla_x \phi^n\sqrt{\mu} w_{\b}\right)(s, X_{\pm}^n(s), V_{ \pm}^{n}(s)) \mathrm{d} s \nonumber\\
	&\quad +\int_0^t e^{-\int_s^{t} \tilde{g}_{ \pm}^{n}(\tau) \mathrm{d} \tau}\left(w_{\b} \mathrm{K}^{ \pm} \mathbf{f}^n\right)(s, X_{ \pm}^n(s), V_{ \pm}^n(s)) \mathrm{d} s \nonumber\\
	&\quad +\int_0^t e^{-\int_s^t \tilde{g}_{ \pm}^{n}(\tau) \mathrm{d} \tau}\left(w_{\b} \mathrm{\Gamma}_{\text {gain }}^{ \pm}(\mathbf{f}^n, \mathbf{f}^n)\right)(s, X_{ \pm}^n(s), V_{ \pm}^n(s)) \mathrm{d} s.
\end{align}

For simplicity, we always take $t_1<1$ in the following proof. For convenience, we agree on $C_0=\frac{\sigma_0}{4}$. If $\sup\limits_{0\leq s\leq t_1}\|\mathbf{h}^{n}(s)\|_{L^\infty}\leq 2C_1(\|\mathbf{h}_0\|_{L^\infty}+1)$, where $C_1\geq e$, a direct calculation shows that
\begin{equation*}
	\tilde{g}_{ \pm}^{n}(t, x, v)\geq C_0|v|^2-2CC_1(\|\mathbf{h}_0\|_{L^\infty}+1)|v| \geq -C_2(\|\mathbf{h}_0\|_{L^\infty}+1)^2,
\end{equation*}
where $C_2=\f{C^2C_1^2}{C_0}$. Taking $t_1\leq\frac{1}{C_2(\|\mathbf{h}_0\|_{L^\infty}+1)^2}$, which yields that
\begin{equation}\label{L9}
	\exp\Big\{-\int_{0}^{t} \tilde{g}_{ \pm}^{n}(\tau) \mathrm{d} \tau\Big\}\leq C_1.
\end{equation}
For $0 \leq t \leq t_1$, a direct calculation shows that
\begin{align}\label{L10}
	\|\mathbf{h}^{n+1}(t)\|_{L^\infty} \leq C_1\|\mathbf{h}_0\|_{L^\infty}+C_3t_1 \sup _{0 \leq s \leq t_1}\|\mathbf{h}^{n}(s)\|_{L^\infty}+C_3 t_1 \sup _{0 \leq s \leq t_1}\|\mathbf{h}^{n}(s)\|_{L^\infty}^2,
\end{align}%
where $C_3>0$ is a specific constant. Taking $t_1\leq\frac{1}{8C_1C_2C_3(\|\mathbf{h}_0\|_{L^\infty}+1)^2}$, we have 
\begin{align}\label{L11}
	\sup _{0 \leq s \leq t_1}\|\mathbf{h}^{n+1}(s)\|_{L^\infty} \leq \frac{3}{2}C_1(\|\mathbf{h}_0\|_{L^\infty}+1).
\end{align}


Next we consider the estimate of $\partial_{x, v}\mathbf{\tilde{h}}^{n+1}(s)$. We denote
\begin{equation}\label{L13}
	\tilde{h}_{\pm}^n(t, x, v):=f_{ \pm}^n(t, x, v) w_{\b_1}(t, v),\quad \text{with} \quad
	w_{\b_1}(t, v)=(1+|v|^2)^{\frac{\b_1}{2}} e^{\frac{\sigma_0}{1+t}|v|^2}.
\end{equation}
It follows from \eqref{L2} that
\begin{align}\label{L14}
	\Big\{\partial_t+v\cdot \nabla_x \mp \nabla_x \phi^n \cdot \nabla_v +\tilde{g}_{\pm,1}^{n}\Big\} \tilde{h}_{\pm}^{n+1}=\Big( \mp v \cdot \nabla_x \phi^n \sqrt{\mu}+\mathrm{K}^{\pm} \mathbf{f}^n+ \mathrm{\Gamma}_{\text {gain }}^{\pm}(\mathbf{f}^n, \mathbf{f}^n)\Big)\cdot w_{\b_1}, 
\end{align}
where $\tilde{g}_{\pm,1}^{ n }(t, x, v)=\frac{\sigma_0}{(1+t)^2} |v|^2 \pm \left\{\frac{1}{2}+\frac{\b_1}{1+|v|^2}+\frac{2 \sigma_0}{1+t}\right\}v\cdot\nabla_x \phi^n +g^n$.
Denote $\partial_{i}=\frac{\partial}{\partial x_i}$,
then it follows from \eqref{L13} that
\begin{align}\label{L15}
	& \left\{\partial_t+v \cdot \nabla_x \mp \nabla_x \phi^n \cdot \nabla_v\right\} \partial_i \tilde{h}_{\pm}^{n+1}+\tilde{g}_{\pm,1}^{n} \partial_i \tilde{h}_{\pm}^{n+1}=\pm\nabla_x \partial_i \phi^n \cdot \nabla_v \tilde{h}_{+}^{n+1}-\partial_i \tilde{g}_{\pm,1}^{n} \tilde{h}_{\pm}^{n+1}     \nonumber\\
	& \mp v \cdot \nabla_x \partial_i \phi^n \sqrt{\mu}(v) w_{\b_1}(t , v)+\partial_i\left(w_{\b_1} \mathrm{K}^{\pm} \mathbf{f}^n\right)+\partial_i\left(w_{\b_1} \mathrm{\Gamma}_{\text {gain }}^{\pm}(\mathbf{f}^n , \mathbf{f}^n)\right).
\end{align}
Integrating \eqref{L15} along the characteristics, we have
\begin{align}\label{L16}
	\partial_i \tilde{h}_{ \pm}^{n+1}(t, x, v)=&\partial_i \tilde{h}_{\pm, 0}(X_{ \pm}(0), V_{ \pm}(0)) e^{-\int_0^{t} \tilde{g}_{\pm,1}^{ n}(\tau) \mathrm{d} \tau} \nonumber\\
	& \pm \int_0^t e^{-\int_s^t \tilde{g}_{\pm,1}^{ n }(\tau) \mathrm{d} \tau}\left(\nabla_x \partial_i \phi^n \cdot \nabla_v \tilde{h}_{ \pm}^{n+1}\right)(s, X_{ \pm}^n(s) , V_{ \pm}^n(s)) \mathrm{d} s \nonumber\\
	& -\int_0^t e^{-\int_s^t \tilde{g}_{\pm,1}^{ n }(\tau) \mathrm{d} \tau}\left(\partial_i \tilde{g}_{\pm,1}^{ n } \tilde{h}_{ \pm}^{n+1}\right)(s, X_{ \pm}^n(s) , V_{ \pm}^n(s)) \mathrm{d} s \nonumber\\
	& \mp \int_0^t e^{-\int_s^t \tilde{g}_{\pm,1}^{ n }(\tau) \mathrm{d} \tau}\left(v \cdot \nabla_x \partial_i \phi^n\sqrt{\mu}(v) w_{\b_1}(t, v)\right)(s, X_{ \pm}^n(s), V_{ \pm}^n(s)) \mathrm{d} s \nonumber\\
	& +\int_0^t e^{-\int_s^t \tilde{g}_{\pm,1}^{ n }(\tau)\mathrm{d} \tau} \left(\partial_i\big(w_{\b_1} \mathrm{K}^{\pm} \mathbf{f}^n\big)\right)(s, X_{ \pm}^n(s), V_{ \pm}^n(s)) \mathrm{d} s \nonumber\\
	& +\int_0^t e^{-\int_s^t \tilde{g}_{\pm,1}^{ n }(\tau) \mathrm{d} \tau}\left(\partial_i\big(w_{\b_1} \mathrm{\Gamma}_{\text {gain }}^{ \pm}(\mathbf{f}^n , \mathbf{f}^n)\big)\right)(s,X_{ \pm}^n(s), V_{ \pm}^n(s)) \mathrm{d} s \nonumber\\
	:=&\sum_{i=0}^{5}D_i.
\end{align}
Similar to \eqref{L9}, it holds for $t\in [0,t_1]$ that
\begin{equation}\label{L17}
	\exp\Big\{-\int_0^{t} \tilde{g}_{\pm,1}^{ n}(\tau) \mathrm{d} \tau\Big\}\leq C_1.
\end{equation}
For $D_1$, taking $d=\f{1}{C(\|\nabla \rho\|_{L^\infty}+1)}$ and $R=\f14$ in \eqref{Eq2.20},
then one gets that
\begin{equation}\label{L18}
	\|\nabla^2 \phi\|_{L^\infty} \leq C(1+\|\rho\|_{L^\infty})\left(1+\ln(\|\nabla \rho\|_{L^\infty}+1)\right).
\end{equation}
which together with \eqref{L18} yields that for $0 \leq t \leq t_1$
\begin{align}\label{L19}
	D_1 &\leq C\int_{0}^{t}(1+\|\mathbf{h}^{n}(s)\|_{L^\infty})\left(1+\ln (1+\|\partial_{x}\mathbf{\tilde{h}}^{n}(s)\|_{L^\infty})\right)\|\partial_{v}\mathbf{\tilde{h}}^{n+1}(s)\|_{L^\infty} \mathrm{d}s \nonumber \\
	&\leq C(1+\|\mathbf{h}_0\|_{L^\infty})\int_{0}^{t}\left(1+\ln(1+ \|\partial_{x}\mathbf{\tilde{h}}^{n}(s)\|_{L^\infty})\right)\|\partial_{v}\mathbf{\tilde{h}}^{n+1}(s)\|_{L^\infty} \mathrm{d}s
\end{align}
For $D_2$,  taking $d=\f18$ and $R=\f14$ in \eqref{Eq2.20}, one has
\begin{equation}\label{L20}
	\|\nabla^2 \phi\|_{L^\infty} \leq C(\|\rho\|_{L^\infty}+\|\nabla \rho\|_{L^\infty}).
\end{equation}
Recalling the definition of $\tilde{g}_{\pm,1}^{ n }(t, x, v)$ in \eqref{L14}, one has from \eqref{L20} that
\begin{align*}
	|\partial_i \tilde{g}_{\pm,1}^{ n}|  &\leq C(1+|v|)(\|\nabla_x \partial_i \phi^n \|_{L^\infty}+\|\partial_x \mathbf{\tilde{h}}^n\|_{L^\infty}) \\
	&\leq C(1+|v|)(\|\mathbf{h}^{n}(s)\|_{L^\infty}+\|\partial_x \mathbf{\tilde{h}}^n(s)\|_{L^\infty}),
\end{align*}
which yields that for $0 \leq t \leq t_1$
\begin{align}\label{L21}
	D_2 &\leq C\int_{0}^{t}\|\mathbf{h}^{n+1}(s)\|_{L^\infty}(\|\nabla_x \partial_i \phi^n(s) \|_{L^\infty}+\|\partial_x \mathbf{\tilde{h}}^n(s)\|_{L^\infty})\mathrm{d}s  \nonumber \\
	&\leq C\int_{0}^{t}\|\mathbf{h}^{n+1}(s)\|_{L^\infty}(\|\mathbf{h}^{n}(s)\|_{L^\infty}+\|\partial_x \mathbf{\tilde{h}}^n(s)\|_{L^\infty})\mathrm{d}s   \nonumber  \\
	&\leq C+C(1+\|\mathbf{h}_0\|_{L^\infty})\int_{0}^{t}\|\partial_x \mathbf{\tilde{h}}^n(s)\|_{L^\infty}\mathrm{d}s. 
\end{align}
Similarly, it follows from \eqref{L20} that for $0 \leq t \leq t_1$
\begin{align}\label{L22}
	D_3 &\leq C\int_{0}^{t}(\|\mathbf{h}^{n}(s)\|_{L^\infty}+\|\partial_x \mathbf{\tilde{h}}^n(s)\|_{L^\infty})\mathrm{d}s   \nonumber  \\
	&\leq C+C\int_{0}^{t}\|\partial_x \mathbf{\tilde{h}}^n(s)\|_{L^\infty}\mathrm{d}s. 
\end{align}
From \eqref{Eq2.47} and \eqref{Eq2.41}, it is clear that
\begin{align}\label{L23}
	D_4 &\leq C\int_{0}^{t}\|\partial_{i}\mathbf{\tilde{h}}^{n}(s)\|_{L^\infty}\int_{\mathbb{R}^3} \f{w_{\b_1}(s, V_{\pm}^{n}(s))}{w_{\b_1}(s, u)}\left(|\mathrm{k}^{+}(V_{\pm}^{n}(s), u)|+|\mathrm{k}^{-}(V_{\pm}^{n}(s), u)|\right) \mathrm{d} u\mathrm{d}s    \nonumber  \\
	&\leq C\int_{0}^{t}\|\partial_x \mathbf{\tilde{h}}^n(s)\|_{L^\infty}\mathrm{d}s. 
\end{align}
For the last term on the RHS of \eqref{L16}, it follows from \eqref{new2.1} that
\begin{align}\label{L24}
	D_5 &\leq C\int_{0}^{t}\|\mathbf{h}^{n}(s)\|_{L^\infty}\|\partial_x \mathbf{\tilde{h}}^n(s)\|_{L^\infty}\mathrm{d}s     \nonumber  \\
	&\leq C(1+\|\mathbf{h}_0\|_{L^\infty})\int_{0}^{t}\|\partial_x \mathbf{\tilde{h}}^n(s)\|_{L^\infty}\mathrm{d}s. 
\end{align}
Substituting  \eqref{L19} and \eqref{L21} $-$ \eqref{L24} into \eqref{L16}, one obtains that
\begin{align}\label{L25}
	\|\partial_x \tilde{h}^{n+1}\|_{L^\infty}  &\leq C(1+ \|\partial_{x, v} \mathbf{\tilde{h}}_0\|_{L^\infty}) + C (1+\|\mathbf{h}_0\|_{L^\infty})\int_{0}^{t}\|\partial_x \mathbf{\tilde{h}}^n(s)\|_{L^\infty}\mathrm{d}s    \nonumber \\
	&\quad+C(1+\|\mathbf{h}_0\|_{L^\infty})\int_{0}^{t}\left(1+\ln (1+\|\partial_{x}\mathbf{\tilde{h}}^{n}(s)\|_{L^\infty})\right)\|\partial_{v}\mathbf{\tilde{h}}^{n+1}(s)\|_{L^\infty} \mathrm{d}s.
\end{align}

To close the estimate of $\|\partial_x \tilde{h}^{n+1}\|_{L^\infty}$, we still need to control $\|\partial_v \tilde{h}^{n+1}\|_{L^\infty}$. We denote $\partial^j=\partial_{v_j}$,  it follows from \eqref{L14} that
\begin{align}\label{L26}
	\{\partial_t+v \cdot \nabla_x \mp \nabla_x \phi^n \cdot \nabla_v\} \partial^j \tilde{h}_{\pm}^{n+1}+\tilde{g}_{\pm,1}^{n} \partial^j \tilde{h}_{\pm}^{n+1}&=-\partial_j \tilde{h}_{\pm}^{n+1}-\partial^j \tilde{g}_{\pm,1}^{n} \tilde{h}_{\pm}^{n+1} \mp \partial^j(\sqrt{\mu} w_{\b_1} v) \cdot \nabla_x \phi^n \nonumber\\
	& +\partial^j\Big(w_{\b_1} \mathrm{K}^{\pm}\mathbf{f}^n\Big)+\partial^j\Big(w_{\b_1} \mathrm{\Gamma}_{\text {gain }}^{\pm}(\mathbf{f}^n , \mathbf{f}^n)\Big).
\end{align}
Integrating \eqref{L26} along  the characteristics, we have
\begin{align}\label{L27}
	\partial^j \tilde{h}_{ \pm}^{n+1}(t, x, v)=&\Big(\partial^j \tilde{h}_{\pm, 0}\Big)(X_{ \pm}^n(0), V_{ \pm}^n(0)) e^{-\int_0^t \tilde{g}_{\pm,1}^{n}(\tau) \mathrm{d} \tau} \nonumber\\
	& -\int_0^t e^{-\int_{s}^{t} \tilde{g}_{\pm,1}^{n}(\tau) \mathrm{d} \tau} \Big(\partial_j \tilde{h}_{ \pm}^{n+1}\Big)(s, X_{ \pm}^n(s), V_{ \pm}^n(s)) \mathrm{d} s \nonumber\\
	& -\int_0^t e^{-\int_s^t \tilde{g}_{\pm,1}^{n}(\tau) \mathrm{d} \tau}\Big(\partial^j \tilde{g}_{\pm,1}^{n} \cdot\tilde{h}_{\pm}^{n+1}\Big)(s, X_{\pm}^n(s) \cdot V_{ \pm}^n(s)) \mathrm{d} s \nonumber\\
	& \mp\int_0^t e^{-\int_s^t \tilde{g}_{\pm,1}^{n}(\tau) \mathrm{d} \tau}\Big(\partial^j\big(\sqrt{\mu} w_{\b_1} v\big)\cdot \nabla_x \phi^n\Big)(s, X_{ \pm}^n(s), V_{ \pm}^n(s)) \mathrm{d} s \nonumber\\
	& +\int_0^t e^{-\int_s^{t} \tilde{g}_{\pm,1}^{n}(\tau) \mathrm{d} \tau}\Big(\partial^j\big(w_{\b_1} \mathrm{K}^{\pm} \mathbf{f}^n\big)\Big)(s, X_{ \pm}^n(s), V_{ \pm}^n(s)) \mathrm{d} s \nonumber\\
	& +\int_0^t e^{-\int_s^t \tilde{g}_{\pm,1}^{n}(\tau) \mathrm{d} \tau}\left(\partial^j\big(w_{\b_1} \mathrm{\Gamma}_{\text {gain }}^{ \pm}(\mathbf{f}^{n}, \mathbf{f}^n)\big)\right)(s, X_{ \pm}^n(s), V_{ \pm}^n(s)) \mathrm{d} s \nonumber\\
	:=&\sum_{i=0}^{5}E_i.
\end{align}
It follows from \eqref{L17} that for $0 \leq t \leq t_1$
\begin{align}\label{L28}
	\sum_{i=0,1,3}E_i& \leq  C(\|\partial_{x,v} \mathbf{\tilde{h}}_0\|_{L^\infty}+1)+C\int_{0}^{t}\|\partial_{x}\mathbf{\tilde{h}}^{n+1}(s)\|_{L^\infty}\mathrm{d}s+C\int_{0}^{t}\|\mathbf{h}^{n}(s)\|_{L^\infty}\mathrm{d}s  \nonumber\\
	&\leq C(\|\partial_{x,v} \mathbf{\tilde{h}}_0\|_{L^\infty}+1)+C\int_{0}^{t}\|\partial_{x}\mathbf{\tilde{h}}^{n+1}(s)\|_{L^\infty}\mathrm{d}s.
\end{align}

For $E_2$, we note that
\begin{equation*}
	|(\partial_{v_j} \tilde{g}_{\pm,1}^{n}) (t, x, v)| \leq C(1+|v|+\|\nabla_x \phi^n\|_{L^\infty}+\|\mathbf{h}^{n}\|_{L^\infty}),
\end{equation*}
which yields that for $0 \leq t \leq t_1$
\begin{align}\label{L29}
	E_2 \leq & C\int_{0}^{t}\left(1+\|\nabla_x \phi^n(s)\|_{L^\infty}+\|\mathbf{h}^{n}(s)\|_{L^\infty}\right)\|\mathbf{h}^{n+1}(s)\|_{L^\infty} \mathrm{d}s   \nonumber \\
	\leq& C\int_{0}^{t}(1+\|\mathbf{h}^{n}(s)\|_{L^\infty})\|\mathbf{h}^{n+1}(s)\|_{L^\infty} ds  \leq C.
\end{align}

For $E_4$, noting $\partial^{j}w_{\b_1}\leq Cw_{\b}$.  From \eqref{Eq2.50} and \eqref{Eq2.41}, it is obvious that
\begin{align*}
	|\partial^{j}(w_{\b_1}\mathrm{K}^{\pm}\mathbf{f}^n)| \leq& |\partial^{j}(w_{\b_1})\mathrm{K}^{\pm}\mathbf{f}^n|+|w_{\b_1}\partial^{j}(\mathrm{K}^{\pm}\mathbf{f}^n)|   \nonumber\\
	\leq& C(\| w_{\b_1} \mathbf{f}^n\|_{L^\infty}+\|w_{\b_1} \partial^{j}\mathbf{f}^n\|_{L^\infty}+\| w_{\b} \mathbf{f}^n \|_{L^\infty})\nonumber\\
	\leq& C(\| \mathbf{h}^n\|_{L^\infty}+\| \partial^{j}\mathbf{\tilde{h}}^{n}\|_{L^\infty}),
\end{align*}
which yields that  for $0 \leq t \leq t_1$
\begin{align}\label{L30}
	E_4 \leq C\int_{0}^{t}(\| \mathbf{h}^{n}(s)\|_{L^\infty}+\|\partial_{v} \mathbf{\tilde{h}}^{n}(s)\|_{L^\infty})\mathrm{d}s   
	\leq C+C\int_{0}^{t}\|\partial_{v} \mathbf{\tilde{h}}^{n}(s)\|_{L^\infty}\mathrm{d}s.
\end{align}
For $E_5$, we have from \eqref{new2.2} that
\begin{align}\label{L31}
	E_5 &\leq C\int_{0}^{t}(\|\mathbf{h}^{n}(s)\|_{L^\infty}^2+\|\mathbf{h}^{n}(s)\|_{L^\infty}\|\partial_{v} \mathbf{\tilde{h}}^{n}(s)\|_{L^\infty})\mathrm{d}s   \nonumber\\
	&\leq C+C(1+\|\mathbf{h}_0\|_{L^\infty})\int_{0}^{t}\|\partial_v \mathbf{\tilde{h}}^n(s)\|_{L^\infty}\mathrm{d}s. 
\end{align}
Combining \eqref{L27} $-$ \eqref{L31}, one has
\begin{align}\label{L32}
	\|\partial_v \mathbf{\tilde{h}}^{n+1}(t)\|_{L^\infty} 
	\leq& C(\|\partial_{x, v} \mathbf{\tilde{h}}_0\|_{L^\infty} +1) +C(1+\|\mathbf{h}_0\|_{L^\infty})\int_{0}^{t}\|\partial_v \mathbf{\tilde{h}}^n(s)\|_{L^\infty}\mathrm{d}s \nonumber\\
	&\quad+\int_{0}^{t}\|\partial_{x}\mathbf{\tilde{h}}^{n+1}(s)\|_{L^\infty}\mathrm{d}s.
\end{align}

It follows from \eqref{L25} and \eqref{L32} that
\begin{align}\label{L33}
	\|\partial_{x, v} \mathbf{\tilde{h}}^{n+1}(t)\|_{L^\infty}  &\leq C(1+ \|\partial_{x, v} \mathbf{\tilde{h}}_0\|_{L^\infty}) + C (1+\|\mathbf{h}_0\|_{L^\infty})\int_{0}^{t}\|\partial_{x, v} \mathbf{\tilde{h}}^n(s)\|_{L^\infty}\mathrm{d}s    \nonumber \\
	&\quad+C(1+\|\mathbf{h}_0\|_{L^\infty})\int_{0}^{t}\left(1+\ln \|\partial_{x, v}\mathbf{\tilde{h}}^{n}(s)\|_{L^\infty}\right)\|\partial_{x, v}\mathbf{\tilde{h}}^{n+1}(s)\|_{L^\infty} \mathrm{d}s.
\end{align}
We make the {\it a priori} assumption $\|\partial_{x, v} \mathbf{\tilde{h}}^{n}(t)\|_{L^\infty}\leq B(t)$. Then \eqref{L33} yields that
\begin{align}\label{newL33}
	\|\partial_{x, v} \mathbf{\tilde{h}}^{n+1}(t)\|_{L^\infty}  &\leq C_{4}(1+ \|\partial_{x, v} \mathbf{\tilde{h}}_0\|_{L^\infty}) + C_4 (1+\|\mathbf{h}_0\|_{L^\infty})\int_{0}^{t}B(s) \mathrm{d}s    \nonumber \\
	&\quad+C_4(1+\|\mathbf{h}_0\|_{L^\infty})\int_{0}^{t}\left(1+\ln B(s)\right)\|\partial_{x, v}\mathbf{\tilde{h}}^{n+1}(s)\|_{L^\infty} \mathrm{d}s,
\end{align}
where $C_4$ is a specific constant.

 We are now in a position to determine $B(t)$. In fact, we take $B(t)$ as the unique solution of the following integral equation:
\begin{align}\label{newL34}
	B(t)&= C_{4}(1+ \|\partial_{x, v} \mathbf{\tilde{h}}_0\|_{L^\infty}) + C_4 (1+\|\mathbf{h}_0\|_{L^\infty})\int_{0}^{t}B(s)\mathrm{d}s    \nonumber \\
	&\quad+C_4(1+\|\mathbf{h}_0\|_{L^\infty})\int_{0}^{t}\left(1+\ln (1+B(s))\right)B(s) \mathrm{d}s.
\end{align}
After simple manipulation, \eqref{newL34} implies that
\begin{equation*}
	B(t)\leq e^{-1}\cdot\Big\{\Big(C_4(1+ \|\partial_{x, v} \mathbf{\tilde{h}}_0\|_{L^\infty})+1\Big)\cdot e\Big\}^{\exp\{2C_4(1+\|\mathbf{h}_0\|_{L^\infty})t\}}-1.	
\end{equation*}
Moreover, combining \eqref{newL33} and \eqref{newL34}, Gronwall's inequality yields that $\|\partial_{x, v} \mathbf{\tilde{h}}^{n+1}(t)\|_{L^\infty}\leq B(t)$.

Without loss of generality, taking $t_1=\frac{1}{16C_1C_2C_3C_4(\|\mathbf{h}_0\|_{L^{\infty}}+1)^2}$, for $0 \leq t \leq t_1$ one has
\begin{equation*}
	B(t)\leq C_5(1+ \|\partial_{x, v} \mathbf{\tilde{h}}_0\|_{L^\infty})^2.
\end{equation*}
By induction on $n$, we can prove that if 
\begin{align}\label{L34}
	\|\partial_{x, v} \mathbf{\tilde{h}}^n(t)\|_{L^\infty}\leq B(t)\leq C_5(1+ \|\partial_{x, v} \mathbf{\tilde{h}}_0\|_{L^\infty})^2,
\end{align}
then it follows from \eqref{L34}  that  $0 \leq t \leq t_1$
\begin{equation}\label{L35}
	\|\partial_{x, v} \mathbf{\tilde{h}}^{n+1}(t)\|_{L^\infty}\leq B(t) \leq C_5(1+ \|\partial_{x, v} \mathbf{\tilde{h}}_0\|_{L^\infty})^2,
\end{equation}
for all $n\geq 0$.

\medskip
\noindent{ Step 3.} 
Now we  prove that $\tilde{h}^{n+1},~n=0,1,2,\cdots$ is a Cauchy sequence.  
It follows from \eqref{L14} that
\begin{align}\label{L36}
	& \{\partial_t+v\cdot \nabla_x-\nabla_x \phi^n \cdot \nabla_v\} (\tilde{h}_{+}^{n+1}-\tilde{h}_{+}^{n})+\tilde{g}_{+,1}^{n} (\tilde{h}_{+}^{n+1}-\tilde{h}_{+}^{n})\nonumber\\
	=& \nabla_x (\phi^n-\phi^{n-1})\nabla_v \tilde{h}_{+}^{n} -(\tilde{g}_{+,1}^{n}-\tilde{g}_{+,1}^{n-1})\tilde{h}_{+}^{n}-v \cdot \nabla_x (\phi^n-\phi^{n-1})\sqrt{\mu}(v) w_{\b_1}(t, v)\nonumber\\
	&+w_{\b_1}\Big(\mathrm{K}^{+}\mathbf{f}^n-\mathrm{K}^{+}\mathbf{f}^{n-1}\Big)+
	w_{\b_1} \Big(\mathrm{\Gamma}_{\text {gain }}^{+}(\mathbf{f}^n, \mathbf{f}^n)-\mathrm{\Gamma}_{\text {gain }}^{+}(\mathbf{f}^{n-1}, \mathbf{f}^{n-1})\Big),
\end{align}
and
\begin{align}\label{L37}
	& \{\partial_t+v\cdot \nabla_x+\nabla_x \phi^n \cdot \nabla_v\} (\tilde{h}_{-}^{n+1}-\tilde{h}_{-}^{n})+\tilde{g}_{-,1}^{n} (\tilde{h}_{-}^{n+1}-\tilde{h}_{-}^{n})\nonumber\\
	=& -\nabla_x (\phi^n-\phi^{n-1})\nabla_v \tilde{h}_{-}^{n} -(\tilde{g}_{-,1}^{n}-\tilde{g}_{-,1}^{n-1})\tilde{h}_{-}^{n}+v \cdot \nabla_x (\phi^n-\phi^{n-1})\sqrt{\mu}(v) w_{\b_1}(t, v)\nonumber\\
	&+w_{\b_1}\left(\mathrm{K}^{-}\mathbf{f}^n-\mathrm{K}^{-}\mathbf{f}^{n-1}\right)
	+w_{\b_1} \left(\mathrm{\Gamma}_{\text {gain }}^{-}(\mathbf{f}^n, \mathbf{f}^n)-\mathrm{\Gamma}_{\text {gain }}^{-}(\mathbf{f}^{n-1}, \mathbf{f}^{n-1})\right).
\end{align}
Then integrating \eqref{L36} and \eqref{L37}  along  the characteristics, one has
\begin{align}\label{L38}
	&\Big(\tilde{h}_{ \pm}^{n+1}-\tilde{h}_{ \pm}^{n}\Big)(t, x, v)\nonumber\\
	=& \pm \int_0^t e^{-\int_s^t \tilde{g}_{\pm,1}^{ n }(\tau) \mathrm{d} \tau}\Big(\nabla_x (\phi^n-\phi^{n-1}) \cdot \nabla_v \tilde{h}_{ \pm}^{n+1}\Big)(s, X_{ \pm}^n(s) , V_{ \pm}^n(s)) \mathrm{d} s \nonumber\\
	& -\int_0^t e^{-\int_s^t \tilde{g}_{\pm,1}^{ n }(\tau) \mathrm{d} \tau}\Big((\tilde{g}_{\pm,1}^{n}-\tilde{g}_{\pm,1}^{n-1})\cdot\tilde{h}_{\pm}^{n}\Big)(s, X_{ \pm}^n(s) , V_{ \pm}^n(s)) \mathrm{d} s \nonumber\\
	& \mp \int_0^t e^{-\int_s^t \tilde{g}_{\pm,1}^{ n }(\tau) \mathrm{d} \tau}\Big(v \cdot \nabla_x (\phi^n-\phi^{n-1})\sqrt{\mu}(v) w_{\b_1}(t, v)\Big)(s, X_{ \pm}^n(s), V_{ \pm}^n(s)) \mathrm{d} s \nonumber\\
	& +\int_0^t e^{-\int_s^t \tilde{g}_{\pm,1}^{ n }(\tau)\mathrm{d} \tau}  \Big(w_{\b_1}(\mathrm{K}^{\pm} \mathbf{f}^n-\mathrm{K}^{\pm} \mathbf{f}^{n-1})\Big)(s, X_{ \pm}^n(s), V_{ \pm}^n(s)) \mathrm{d} s \nonumber\\
	& +\int_0^t e^{-\int_s^t \tilde{g}_{\pm,1}^{ n }(\tau) \mathrm{d}\tau}\Big(w_{\b_1} \big(\mathrm{\Gamma}_{\text {gain }}^{\pm}(\mathbf{f}^n, \mathbf{f}^n)-\mathrm{\Gamma}_{\text {gain }}^{\pm}(\mathbf{f}^{n-1}, \mathbf{f}^{n-1})\big)\Big)(s,X_{ \pm}^n(s), V_{ \pm}^n(s)) \mathrm{d} s  \nonumber\\
	=&:\sum_{i=1}^{5}J_i.
\end{align}
Utilizing \eqref{L35} and \eqref{Eq2.19}, one gets that
\begin{align}\label{L39}
	J_1\leq C(1+ \|\partial_{x, v} \mathbf{\tilde{h}}_0\|_{L^\infty})^2\int_{0}^{t}\|(\mathbf{\tilde{h}}^n-\mathbf{\tilde{h}}^{n-1})(s)\|_{L^\infty}\mathrm{d}s.
\end{align}

For $J_2$, we note from \eqref{L13} and \eqref{L3} that
\begin{align*}
	\left|(\tilde{g}_{\pm,1}^{n}-\tilde{g}_{\pm,1}^{n-1})(t, x, v)\right|\leq C(1+|v|)\|(\mathbf{\tilde{h}}^{n}-\mathbf{\tilde{h}}^{n-1})(t)\|_{L^\infty},
\end{align*}
which, together with \eqref{Pri0},  yields that
\begin{align}\label{L40}
	J_2+J_3\leq C(1+\|\mathbf{h}_0\|_{L^\infty})\int_{0}^{t}\|(\mathbf{\tilde{h}}^n-\mathbf{\tilde{h}}^{n-1})(s)\|_{L^\infty}\mathrm{d}s.
\end{align}

For $J_4$, it is easy from \eqref{newK5} and \eqref{Eq2.41} to see that
\begin{align*}
	\Big|w_{\b_1}(\mathrm{K}^{\pm} \mathbf{f}^n-\mathrm{K}^{\pm} \mathbf{f}^{n-1})\Big|
	=&\Big|w_{\b_1}\int_{\mathbb{R}^3}\mathtt{k}^{(2)}(v,u)(f_{\pm}^{n}-f_{\pm}^{n-1})(u)\,\mathrm{d}u\Big| \nonumber\\
	&+\Big|w_{\b_1}\int_{\mathbb{R}^3}\mathtt{k}^{(1)}(v,u)(f^{n}_{\mp}-f^{n-1}_{\mp})(u)\,\mathrm{d}u\Big|  \nonumber\\
	\leq& C\|\mathbf{\tilde{h}}^{n}-\mathbf{\tilde{h}}^{n-1}\|_{L^\infty},
\end{align*}
which implies that
\begin{align}\label{L41}
J_4\leq C\int_{0}^{t}\|(\mathbf{\tilde{h}}^n-\mathbf{\tilde{h}}^{n-1})(s)\|_{L^\infty}\mathrm{d}s.
\end{align}

For $J_5$, we have from \eqref{Eq1.15} and \eqref{Eq2.43} that
\begin{equation*}
	w_{\b_1}\Big|\mathrm{\Gamma}_{\text {gain }}^{\pm}(\mathbf{f}^n, \mathbf{f}^n)-\mathrm{\Gamma}_{\text {gain }}^{\pm}(\mathbf{f}^{n-1}, \mathbf{f}^{n-1})\Big|\leq C(
	\|\mathbf{\tilde{h}}^{n}\|_{L^\infty}+\|\mathbf{\tilde{h}}^{n-1}\|_{L^\infty})\|\mathbf{\tilde{h}}^{n}-\mathbf{\tilde{h}}^{n-1})\|_{L^\infty},
\end{equation*}
which, together with \eqref{Pri0}, yields that 
\begin{align}\label{L42}
	J_5\leq C(1+\|\mathbf{h}_0\|_{L^\infty})\int_{0}^{t}\|(\mathbf{\tilde{h}}^n-\mathbf{\tilde{h}}^{n-1})(s)\|_{L^\infty}\mathrm{d}s.
\end{align}

Substituting \eqref{L39}$-$\eqref{L42} into \eqref{L38}, one obtains that for $0 \leq t \leq t_1$,
\begin{align*}
	\|\mathbf{\tilde{h}}^{n+1}-\mathbf{\tilde{h}}^{n}(t)\|_{L^\infty}
	\leq C(1+\|\mathbf{h}_0\|_{L^\infty}+\|\partial_{x, v}\mathbf{\tilde{h}}_0\|_{L^\infty})^2\int_{0}^{t}\|(\mathbf{\tilde{h}}^n-\mathbf{\tilde{h}}^{n-1})(s)\|_{L^\infty}\mathrm{d}s, 
\end{align*}
and by induction on $n$
\begin{equation*}
	\sup _{0 \leq s \leq t_1}\|\mathbf{\tilde{h}}^{n+1}-\mathbf{\tilde{h}}^{n}(s)\|_{L^\infty}\leq \frac{C^n(1+\|\mathbf{h}_0\|_{L^\infty}+\|\partial_{x, v}\mathbf{\tilde{h}}_0\|_{L^\infty})^{2n}}{n!}t_1^n\leq \frac{C^n(1+\|\mathbf{h}_0\|_{L^\infty}+\|\partial_{x, v}\mathbf{\tilde{h}}_0\|_{L^\infty})^{2n}}{n!},
\end{equation*}
which yields immediately that $\tilde{h}_{\pm}^{n+1},~n=0,1,2,\cdots$ is a Cauchy sequence. Therefore, there exists a limit $f_{\pm}$ such that 
\begin{equation}\notag
	\sup_{0\leq s\leq t_1}\|w_{\b_1}(\mathbf{f}^n-\mathbf{f})(s)\|_{L^\infty_{x,v}}\rightarrow0 ~~\mbox{as}~~n\rightarrow +\infty.
\end{equation}
It is clear to know that $f_{\pm}$ is a mild solution of \eqref{Eq1.17}, thus $F_{\pm}=\mu(v)+\sqrt{\mu(v)}f_{\pm} \geq 0$ is a mild solution of VPB \eqref{Eq1.4}. It follows from \eqref{L11}, \eqref{L34} and \eqref{L35} that 
\begin{align*}
	&\sup_{0\leq s\leq t_1}\|\mathbf{h}(s)\|_{L^\infty}\leq 2C_1(1+\|\mathbf{h}_0\|_{L^\infty}),\\ 
	&\sup _{0 \leq s \leq t_1}\|\partial_{x, v}\mathbf{\tilde{h}}(s)\|_{L^\infty} \leq C_5(1+ \|\partial_{x, v} \mathbf{\tilde{h}}_0\|_{L^\infty})^2.
\end{align*}

\medskip
\noindent{ Step 4.} 
Now we consider the uniqueness. Let $\mathbf{g}(t,x,v)$ be another solution of \eqref{Eq1.17}  with 
 \begin{align*}
 &\sup\limits_{0\leq t\leq t_1}\|w_{\b}\mathbf{g}(t)\|_{L^\infty}\leq 2C_1(1+\|\mathbf{h}_0\|_{L^\infty}),  \\
 &\sup\limits_{0\leq t\leq t_1}\|\partial_{x, v}(w_{\b_1}\mathbf{g})\|_{L^\infty})\leq C_5(1+ \|\partial_{x, v} \mathbf{\tilde{h}}_0\|_{L^\infty})^2.
\end{align*}
 By similar arguments as in \eqref{L36} $-$ \eqref{L42},  it is direct to obtain that 
\begin{align*}\notag
	\|w_{\b_1}(\mathbf{f}-\mathbf{g})(t)\|_{L^\infty}\leq 
	C(1+\|\mathbf{h}_0\|_{L^\infty}+\|\partial_{x, v}\mathbf{\tilde{h}}_0\|_{L^\infty})^2 \int_0^t\|w_{\b_1}(\mathbf{f}-\mathbf{g})(s)\|_{L^\infty}\mathrm{d}s,
\end{align*}
which, together with the Gronwall's  inequality, yields the uniqueness, i.e., $\mathbf{f}=\mathbf{g}$.

\medskip
\noindent{ Step 5.} 
 Taking the limit  $n\rightarrow+\infty$ in \eqref{L1},  and then multiplying $\eqref{L1}_1$ by  $1,v,|v|^2$ and $\ln F_{+}$, and multiplying  $\eqref{L1}_2$  by  $1,v,|v|^2$ and $\ln F_{-}$, integrating by parts, together with $\eqref{L1}_3$, one can obtain \eqref{Eq2.4}$-$\eqref{Eq2.7} and \eqref{Eq2.8}.  

\medskip
\noindent{ Step 6.} 
If $F_{\pm, 0}$ (or equivalent $f_{\pm, 0}$) is continuous, it is direct to check that $F_{\pm}^{n+1}(t,x,v)$ (or equivalent $f_{\pm}^{n+1}(t,x,v)$) is continuous in $[0,\infty)\times\mathbb{T}^3 \times\mathbb{R}^3$. The continuous of $\mathbf{f}(t,x,v)$  is an immediate consequence of  $\sup\limits_{0\leq s\leq t_1}\|(\mathbf{f}^{n+1}-\mathbf{f})(s)\|_{L^\infty}\rightarrow 0$ as $n\rightarrow +\infty$. 

\medskip
\noindent{ Step 7.} 
Let $\|f_{+, 0}-f_{-, 0}\|_{\infty}\leq \varepsilon_0$. 
It follows from \eqref{L2} that
\begin{align}\label{L43}
	& \{\partial_t+v \cdot \nabla_x\} (f_{+}-f_{-})+g(t, x, v) (f_{+}-f_{-})  \nonumber\\
	=&\nabla_x \phi \cdot \nabla_v(f_{+}+f_{-})
	-\nabla_x \phi \cdot \frac{v}{2} (f_{+}+f_{-})
	-2v \cdot \nabla_x \phi \sqrt{\mu} 
	+ \mathrm{K}^{+} \mathbf{f}-\mathrm{K}^{-} \mathbf{f}    \nonumber\\
	&+\mathrm{\Gamma}_{\text {gain }}^{+}(\mathbf{f}, \mathbf{f})-\mathrm{\Gamma}_{\text {gain }}^{-}(\mathbf{f}, \mathbf{f}),
\end{align}
where
\begin{align*}
	g(t, x, v) & =\iint B(v-u, \omega)\left(F_++F_{-}\right)(t, x, u)  \mathrm{d} \omega  \mathrm{d} u \geq 0.
\end{align*}
From \eqref{L43}, it is direct to have that
\begin{align}\label{L44}
	\left(f_{+}-f_{-}\right)(t, x, v)
	=&\left(f_{+, 0}-f_{-, 0}\right)(x-tv,v) e^{-\int_0^{t} g(\tau, \mathfrak{X}(\tau), v) \mathrm{d} \tau} \nonumber\\
	& + \int_0^t e^{-\int_s^t g(\tau, \mathfrak{X}(\tau), v) \mathrm{d} \tau}\Big(\nabla_x  \phi \cdot \nabla_v (f_{+}+f_{-})\Big)(s,\mathfrak{X}(\tau), v) \mathrm{d} s \nonumber\\
	& -\int_0^t e^{-\int_s^t g(\tau, \mathfrak{X}(\tau), v) \mathrm{d} \tau}\Big(\nabla_x  \phi\cdot \frac{v}{2} (f_{+}+f_{-})\Big)(s,\mathfrak{X}(\tau), v) \mathrm{d} s  \nonumber\\
	& -2 \int_0^t e^{-\int_s^t g(\tau, \mathfrak{X}(\tau), v) \mathrm{d} \tau}\Big(v \cdot \nabla_x \phi\sqrt{\mu}\Big)(s,\mathfrak{X}(\tau), v) \mathrm{d} s \nonumber\\
	& +\int_0^t e^{-\int_s^t g(\tau, \mathfrak{X}(\tau), v)\mathrm{d} \tau} \Big( \mathrm{K}^{+} \mathbf{f}-\mathrm{K}^{-} \mathbf{f}\Big)(s, \mathfrak{X}(\tau), v) \mathrm{d} s \nonumber\\
	& +\int_0^t e^{-\int_s^t g(\tau, \mathfrak{X}(\tau), v) \mathrm{d} \tau}\Big( \mathrm{\Gamma}_{\text {gain }}^{ +}(\mathbf{f}, \mathbf{f})-\mathrm{\Gamma}_{\text {gain }}^{ -}(\mathbf{f} , \mathbf{f})\Big)(s,\mathfrak{X}(\tau), v) \mathrm{d} s        \nonumber\\
	:=&\sum_{i=0}^{5}H_i,
\end{align}
where $\mathfrak{X}(s):=x-v(t-s)$.
Noting \eqref{E2.2}, \eqref{E2.3}, we have
\begin{equation}\label{L45}
	\sum_{i=0}^{3}H_i\leq \|f_{+, 0}-f_{-, 0}\|_{L^\infty} +C(1+\|\mathbf{h}_0\|_{L^\infty}+\|\partial_{x, v}\mathbf{\tilde{h}}_0\|_{L^\infty})^2\int_{0}^{t}\|(f_{+}-f_{-})(s)\|_{L^\infty}\mathrm{d}s.
\end{equation}

For $H_4$, by a rotation, it follows from \eqref{newK2}  that
\begin{align*}
	\mathrm{K}^{+}\mathbf{f}-\mathrm{K}^{-}\mathbf{f}
	=&2\int_{\mathbb{R}^3}\int_{\mathbb{S}^2} B(v-u,\t)(f_+-f_-)(v')\sqrt{\mu(u)} \sqrt{\mu(u')} d u d \omega \\
	=&\int_{\mathbb{R}^3}\mathrm{k}_{2}(v,u)(f_+-f_-)(u)\,\mathrm{d}\eta ,
\end{align*}
which combined with \eqref{newK4} and \eqref{Eq2.41} yield that
\begin{equation}\label{L46}
	|\mathrm{K}^{+}\mathbf{f}-\mathrm{K}^{-}\mathbf{f}|\leq C\|f_+-f_-\|_{L^\infty}.
\end{equation}
Then it follows from \eqref{L46} that
\begin{equation}\label{L47}
	H_4\leq  C\int_{0}^{t}\|(f_{+}-f_{-})(s)\|_{L^\infty}\mathrm{d}s.
\end{equation}
For $H_5$, by a rotation, it is noted that
\begin{align*}
	\mathrm{\Gamma}_{\text {gain }}^{+}(\mathbf{f} , \mathbf{f})-\mathrm{\Gamma}_{\text {gain }}^{-}(\mathbf{f} , \mathbf{f})
	=&\frac{1}{\sqrt{\mu}}Q_{\text {gain }}\left(\sqrt{\mu}(f_{+}-f_{-}), \sqrt{\mu}(f_{+}+f_{-})\right) \nonumber\\
	=&\iint B(v-u,\t)\sqrt{\mu(u)}(f_{+}-f_{-})(v')(f_{+}+f_{-})(u')  \mathrm{d} \omega  \mathrm{d} u \nonumber\\
	=&\iint B(v-u,\t)\sqrt{\mu(u)}(f_{+}-f_{-})(u')(f_{+}+f_{-})(v')  \mathrm{d} \omega  \mathrm{d} u =:I_A.
\end{align*}
To $I_A$, as in \cite{Glassey96Book}, we use the change of variables $u'=v+z_{\perp}$, $v'=v+z_{\shortparallel}$ and $\eta=v+z_{\shortparallel}$ with $z=u-v$, $z_{\shortparallel}=(z\cdot \omega)\omega$, $z_{\perp}=z-z_{\shortparallel}$. Moreover, 
\begin{align*}
	|B(v-u, \t)|\leq  C(|z_{\shortparallel}|^2+|z_{\perp}|^2)^{\frac{\gamma-1}{2}}|z_{\shortparallel}|,
\end{align*}
and 
\begin{align*}
	\mathrm{d}\omega \mathrm{d}u=\mathrm{d}z_{\shortparallel}\mathrm{d}z_{\perp}\mathrm{d}\omega =2|z_{\shortparallel}|^2\mathrm{d}|z_{\shortparallel}|\mathrm{d}\omega\frac{\mathrm{d}z_{\perp}}{|z_{\shortparallel}|^2}=\frac{2}{|z_{\shortparallel}|^2} \mathrm{d}z_{\shortparallel}\mathrm{d}z_{\perp}=\frac{2}{|\eta-v|^2}\mathrm{d}\eta \mathrm{d}z_{\perp}.
\end{align*}
Then one has 
\begin{align}\label{L48}
	I_A&\leq \|(f_{+}-f_{-})\|_{L^\infty}\int_{\eta} \int_{z_{\perp}}\f{1}{|\eta-v|}|(z_{\shortparallel}|^2+|z_{\perp}|^2)^{\frac{\gamma-1}{2}}e^{-\f{|\eta+z_{\perp}|^2}{4}}|\mathbf{f}(\eta)|\mathrm{d}\eta \mathrm{d}z_{\perp}    \nonumber\\
	&\leq C\|f_{+}-f_{-}\|_{L^\infty}\|\mathbf{h}\|_{L^\infty}.
\end{align}
which yields that
\begin{equation}\label{L49}
	H_5\leq  C(1+\|\mathbf{h}_0\|_{L^\infty})\int_{0}^{t}\|(f_{+}-f_{-})(s)\|_{L^\infty}\mathrm{d}s.
\end{equation}
Submitting \eqref{L45}, \eqref{L47} and \eqref{L49} into \eqref{L44},  one has
\begin{align*}
	\|(f_{+}-f_{-}(t)\|_{L^\infty} \leq \varepsilon_0 \exp\left\{C(1+\|\mathbf{h}_0\|_{L^\infty}+\|\partial_{x, v}\mathbf{\tilde{h}}_0\|_{L^\infty})^2t\right\}.
\end{align*}
Then for  $0 \leq t \leq t_1$,
\begin{align}\label{L50}
	\|(f_{+}-f_{-}(t)\|_{L^\infty} \leq C\varepsilon_0 \exp\left\{C(1+\|\partial_{x, v}\mathbf{\tilde{h}}_0\|_{L^\infty})^2t\right\}\leq C\varepsilon_0 \exp\left\{C(1+\|\partial_{x, v}\mathbf{\tilde{h}}_0\|_{L^\infty})^2\right\}.
\end{align}
Therefore the proof of  Proposition \ref{prop2.1} is completed. 
\end{proof}



\section{Uniform $L^{\infty}$-Estimate}   \label{Section 4}
From now on, we make the {\it a priori} assumption \eqref{Eq2.21}, i.e.,
\begin{align*}
	\|\nabla_x\phi(t)\|_{L^{\infty}}\le \delta (1+t)^{-2},\quad \|\nabla^2_x\phi(t)\|_{L^{\infty}}\le \delta (1+t)^{-\frac{5}{2}},
\end{align*}
where $\delta$ is a  small positive constant to be determined later.

\subsection{$L^{\infty}$ {\it a priori} estimate}
Recall the weight function $w_{\b}(t, v)$ in \eqref{Eq1.22}.
 It follows from $\eqref{Eq1.17}_{1, 2}$  that
\begin{align}
	\left(\partial_t+v \cdot \nabla _x-\nabla_x \phi \cdot \nabla_v\right)h_{+}+\tilde{\nu}_{+}h_{+}= w_{\b}\mathrm{K}^{+}\mathbf{f}+w_{\b}\mathrm{\Gamma}^{+}(\mathbf{f},\mathbf{f})-\nabla_x \phi \cdot v w_{\b}\sqrt{\mu},\label{Eq3.5}\\
	\left(\partial_t+v \cdot \nabla _x+\nabla_x \phi \cdot \nabla_v\right)h_{-}+\tilde{\nu}_{-}h_{-}=  w_{\b}\mathrm{K}^{-}\mathbf{f}+w_{\b}\mathrm{\Gamma}^{-}(\mathbf{f},\mathbf{f})+\nabla_x \phi \cdot v w_{\b}\sqrt{\mu},\label{Eq3.6}
\end{align}
where
\begin{align}\label{Eq3.7}
	\tilde{\nu}_{\pm}(t,x,v):=\frac{ \sigma_0}{(1+t)^{2}}|v|^2+\nu(v)\pm\nabla_x\phi\cdot v\Big(\frac{1}{2}+\frac{\beta}{1+|v|^2}+\frac{2\sigma_0}{1+t}\Big).
\end{align}

For $\delta\ll 1$, it follows from  \eqref{Eq2.21} and \eqref{Eq3.7} that 
\begin{align}\label{Eq3.10}
	\tilde{\nu}_{\pm}(t,x,v)&\geq \frac{1}{(1+t)^2}\left(\sigma_0|v|^2-C\delta |v|\right)+\nu(v)  \nonumber\\
	&=\frac{\sigma_0}{(1+t)^2}\Big(|v|-\f{C\delta}{2\sigma_0}\Big)^2-\frac{C^2\delta^2}{4\sigma_0(1+t)^2}+\nu(v)
	\geq \frac{1}{2}\nu(v)
	\geq \tilde{\nu}_0>0,
\end{align}
for a positive constant $\tilde{\nu}_0$.

Recall the backward characteristics $\left(X_{\pm}(\tau), V_{\pm}(\tau)\right)=\left(X_{\pm}(\tau;t,x,v), V_{\pm}(\tau;t,x,v)\right)$ in \eqref{Eq2.22}, the mild formulations of \eqref{Eq3.5},  \eqref{Eq3.6} take the form
\begin{align}\label{Eq3.11}
	h_{\pm}(t,x,v)
	&=\left(h_{\pm, 0}\right)(X_{\pm}(0),V_{\pm}(0))e^{-\int_0^t\tilde{\nu}_{\pm}\left(\tau, X_{\pm}(\tau), V_{\pm}(\tau)\right)\mathrm{d}\tau}\nonumber\\
	&\quad +\int_0^te^{-\int_s^t\tilde{\nu}_{\pm}\left(\tau, X_{\pm}(\tau), V_{\pm}(\tau)\right)\mathrm{d}\tau}\Big(w_{\b}\mathrm{K}^{\pm}\mathbf{f}\Big)(s,X_{\pm}(s),V_{\pm}(s))\mathrm{d}s\nonumber\\
	&\quad \mp\int_0^te^{-\int_s^t\tilde{\nu}_{\pm}\left(\tau, X_{\pm}(\tau), V_{\pm}(\tau)\right)\mathrm{d}\tau}\Big(\nabla_x\phi\cdot vw_{\b}\sqrt{\mu}\Big)(s,X_{\pm}(s),V_{\pm}(s))\mathrm{d}s\nonumber\\
	&\quad +\int_0^te^{-\int_s^t\tilde{\nu}_{\pm}\left(\tau, X_{\pm}(\tau), V_{\pm}(\tau)\right)\mathrm{d}\tau}\Big(w_{\b}\mathrm{\Gamma}^{\pm}(\mathbf{f},\mathbf{f})\Big)(s,X_{\pm}(s),V_{\pm}(s))\mathrm{d}s.
\end{align}

\begin{lemma}\label{Lem3.1}
	Under the condition \eqref{Eq2.21}, it holds that for $\b \geq 4$ and $\b_1\geq 0$, 
	\begin{align}\label{E4.5-1}
		\sup_{0\leq s\leq t}\|\mathbf{h}(s)\|_{L^\infty}
		\le& C_6\{\|\mathbf{h}_{0}\|_{L^{\infty}}+\|\mathbf{h}_{0}\|_{L^{\infty}}^2+\sqrt{\mathcal{E}(\mathbf{F}_0)}\} \nonumber\\
		&\quad+C_6\sup_{\substack{t_1\le s\le t\\ y \in \mathbb{T}^3}}\Big\{\|\mathbf{h}(s)\|^{\frac{3}{2}}_{L^{\infty}} \Big(\int_{\mathbb{R}^3}|\mathbf{\tilde{h}}(s,y,u)|\mathrm{d}u\Big)^{\frac{1}{2}}\Big\},
	\end{align}
	where the positive constant $C_6\geq1$ depends only on $\gamma$, $\b$, Here $t_1>0$ is the lifespan defined in \eqref{LT}.
\end{lemma}
\begin{proof}
	From \eqref{Eq2.46} and \eqref{Eq3.11}, it is easy to have that
	\begin{align}\label{Eq3.12}
		|h_{\pm}(t,x,v)|&\le \|\mathbf{h}_{0}\|_{L^{\infty}}e^{-\tilde{\nu}_0 t} +\int_0^te^{-\tilde{\nu}_0(t-s)}|(w_{\b}\mathrm{K}^{\pm}\mathbf{f})(s,X_{\pm}(s),V_{\pm}(s))|\mathrm{d}s\nonumber\\
		&\quad +C\int_0^te^{-\tilde{\nu}_0(t-s)}\mu^{\frac{1}{4}}(V_{\pm}(s))|\nabla_x\phi(s,X_{\pm}(s),V_{\pm}(s))|\mathrm{d}s\nonumber\\
		&\quad +C\int_0^te^{-\int_s^t\tilde{\nu}_{\pm}(\tau)\mathrm{d}\tau}\cdot \nu(V_{\pm}(s))\|\mathbf{h}(s)\|^{\f{3}{2}}_{L^{\infty}} \Big(\int_{\mathbb{R}^3}|\mathbf{\tilde{h}}(s,X_{\pm}(s),u)|\mathrm{d}u\Big)^{\f{1}{2}}\mathrm{d}s.
	\end{align}
	It follows from \eqref{Eq2.16} that 
	\begin{align}\label{Eq3.13}
		&\int_0^te^{-\tilde{\nu}_0(t-s)}\mu^{\frac{1}{4}}(V_{\pm}(s))|\nabla_x\phi(s,X_{\pm}(s))|\mathrm{d}s\nonumber\\
		&\le \f{C}{N} \int_0^te^{-\tilde{\nu}_0(t-s)}\|\mathbf{h}(s)\|_{L^{\infty}}\mathrm{d}s+C_{N} \int_0^te^{-\tilde{\nu}_0(t-s)}\|\nabla_x\phi(s)\|_{L^{2}}\mathrm{d}s\nonumber\\
		&\le  \f{C}{N}\sup_{0\le s\le t}\|\mathbf{h}(s)\|_{L^{\infty}}+C_{N}\sup_{0\le s\le t}\|\nabla_x\phi(s)\|_{L^{2}}.
	\end{align}
	Substituting \eqref{Eq3.13} into \eqref{Eq3.12}, one has
	\begin{align}\label{Eq3.14}
		|h_{\pm}(t,x,v)|&\le \|\mathbf{h}_{0}\|_{L^{\infty}}e^{-\tilde{\nu}_0 t}+ \f{C}{N}\sup_{0\le s\le t}\|\mathbf{h}(s)\|_{L^{\infty}}+ C_{N}\sup_{0\le s\le t}\|\nabla_x\phi(s)\|_{L^{2}}\nonumber\\ 
		&\quad +C\sup_{\substack{0\le s\le t\\ y \in \mathbb{T}^3}}\Big\{\|\mathbf{h}(s)\|^{\f{3}{2}}_{L^{\infty}}\cdot \Big(\int_{\mathbb{R}^3}|\mathbf{\tilde{h}}(s,y,u)|\mathrm{d}u\Big)^{\f{1}{2}}\Big\}\nonumber\\
		&\quad +C\int_0^te^{-\tilde{\nu}_0(t-s)}\mathrm{d}s\int_{\mathbb{R}^3}|\mathtt{k}_{w_{\b}}^{(2)}(V_{\pm}(s),u)h_{\pm}(s,X_{\pm}(s),u)|\mathrm{d}u\nonumber\\
		&\quad +C\int_0^te^{-\tilde{\nu}_0(t-s)}\mathrm{d}s\int_{\mathbb{R}^3}|\mathtt{k}_{w_{\b}}^{(1)}(V_{\pm}(s),u)h_{\mp}(s,X_{\pm}(s),u)|du,
	\end{align}
	where $\mathtt{k}_{w_{\b}}^{(i)}(v,u)=\mathtt{k}^{(i)}(v,u)\cdot\frac{w_{\b}(t, v)}{w_{\b}(t, u)}$ ($i=1,2$). We  only  consider $|h_{+}(t,x,v)|$ because $|h_{-}(t,x,v)|$ can be dealt similarly. For simplicity of presentation, we denote
	\begin{equation}\label{Hat1}
	(\hat{X}_{\pm}(\tau),\hat{V}_{\pm}(\tau)):=(X_{\pm}(\tau;s,X_{+}(s),u), V_{\pm}(\tau;s,X_{+}(s),u)).
	\end{equation}
	 Using \eqref{Eq3.14} again, one has
	\begin{align}\label{Eq3.15}
		|h_{+}(t,x,v)|
		\le& C\|\mathbf{h}_{0}\|_{L^{\infty}}+ \f{C}{N}\sup_{0\le s\le t}\|\mathbf{h}(s)\|_{L^{\infty}}+C_{N}\sup_{0\le s\le t}\|\nabla_x\phi(s)\|_{L^{2}}\nonumber\\ 
		&+C\sup_{\substack{0\le s\le t\\ y \in \mathbb{T}^3}}\Big\{\|\mathbf{h}(s)\|^{\f{3}{2}}_{L^{\infty}}\cdot \Big(\int_{\mathbb{R}^3}|\mathbf{\tilde{h}}(s,y,u)|\mathrm{d}u\Big)^{\f{1}{2}}\Big\} +\sum_{i=1}^{4}I_i,
      \end{align}
	where,
		\begin{align*}
		 I_{1}:=&\int_0^t\int_0^se^{-\tilde{\nu}_0(t-s)}e^{-\tilde{\nu}_0(s-\tau)}\mathrm{d}\tau\mathrm{d}s      \nonumber\\
		 &\times\int_{\mathbb{R}^3}\int_{\mathbb{R}^3}|\mathtt{k}_{w_{\b}}^{(2)}(V_{+}(s),u)\mathtt{k}_{w_{\b}}^{(2)}(\hat{V}_{+}(\tau),u')| |h_{+}(\tau,\hat{X}_{+}(\tau),u')|\mathrm{d}u'du; \nonumber\\
		I_{2}:=&\int_0^t\int_0^se^{-\tilde{\nu}_0(t-s)}e^{-\tilde{\nu}_0(s-\tau)}\mathrm{d}\tau\mathrm{d}s      \nonumber\\
		&\times\int_{\mathbb{R}^3}\int_{\mathbb{R}^3}|\mathtt{k}_{w_{\b}}^{(2)}(V_{+}(s),u)\mathtt{k}_{w_{\b}}^{(1)}(\hat{V}_{+}(\tau),u')| |h_{-}(\tau,\hat{X}_{+}(\tau),u')|\mathrm{d}u'du; \nonumber\\
		I_{3}:=&\int_0^t\int_0^se^{-\tilde{\nu}_0(t-s)}e^{-\tilde{\nu}_0(s-\tau)}\mathrm{d}\tau\mathrm{d}s      \nonumber\\
		&\times\int_{\mathbb{R}^3}\int_{\mathbb{R}^3}|\mathtt{k}_{w_{\b}}^{(1)}(V_{+}(s),u)\mathtt{k}_{w_{\b}}^{(2)}(\hat{V}_{-}(\tau),u')| |h_{-}(\tau,\hat{X}_{-}(\tau),u')|\mathrm{d}u'du; \nonumber\\
		I_{4}:=&\int_0^t\int_0^se^{-\tilde{\nu}_0(t-s)}e^{-\tilde{\nu}_0(s-\tau)}\mathrm{d}\tau\mathrm{d}s      \nonumber\\
		&\times\int_{\mathbb{R}^3}\int_{\mathbb{R}^3}|\mathtt{k}_{w_{\b}}^{(1)}(V_{+}(s),u)\mathtt{k}_{w_{\b}}^{(1)}(\hat{V}_{-}(\tau),u')| |h_{+}(\tau,\hat{X}_{-}(\tau),u')|\mathrm{d}u'du.
	\end{align*}
	
	We fix $N>0$ large enough so that  for any $0\le \tau\le s\le t$
	\begin{align}\label{Eq3.17}
		\sup_{0\le s\le t}|V_{\pm}(s)-v|\le C\delta \le \frac{N}{8},\quad \sup_{0\le \tau\le s\le t}|\hat{V}_{\pm}(\tau)-u|\le C\delta \le \frac{N}{8}.
	\end{align}
	We first discuss $I_1$ of \eqref{Eq3.15} and split it into three cases.\\
	\text{Case 1:} $|v|\ge N$. In this case, we have $|V_{+}(s)|\ge \frac{N}{2}$. It is noted that
	\begin{align}\label{Eq3.18}
		\int_{\mathbb{R}^3}|\mathtt{k}_{w_{\b}}^{(i)}(V_{+}(s),u)|\mathrm{d}u\le \frac{C}{1+|V_{+}(s)|}, \quad i=1,2,
	\end{align}
	which yields that
	\begin{align}\label{Eq3.19}
		I_1\le \frac{C}{N}\sup_{0\le s\le t}\|\mathbf{h}(s)\|_{L^{\infty}}.
	\end{align}
	\text{Case 2:} For either $|v|\le N, |u|\ge 2N$ or $|u|\le 2N, |u'|\ge 3N$,
	we get either
	\begin{align*}
		|V_{+}(s)-u|\geq |u-v|-|V_{+}(s)-v| \geq |u|-|v|-|V_{+}(s)-v|\geq \frac{N}{2},
	\end{align*}
	or
	\begin{align*}
		|\hat{V}_{+}(\tau)-u'|\geq |u'-u|-|\hat{V}_{+}(\tau)-u| \geq |u'|-|u|-|\hat{V}_{+}(\tau)-u|\geq \frac{N}{2},
	\end{align*}
	for $N$ suitably large, $s, \tau \in [0,t]$. Then either one of the following is valid: for $i=1,2$,
	\begin{equation}\label{Eq3.20}
		\begin{split}
			|\mathtt{k}_{w_{\b}}^{(i)}(V_{+}(s), u)|&\leq e^{-\frac{N^{2}}{64}}|\mathtt{k}_{w_{\b}}^{(i)}(V_{+}(s), u)|e^{\frac{|V_{+}(s)-u|^{2}}{16}},\quad |V_{+}(s)-u|\ge \frac{N}{2},\\
			|\mathtt{k}_{w_{\b}}^{(i)}(\hat{V}_{+}(\tau), u')|&\leq e^{-\frac{N^{2}}{64}}|\mathtt{k}_{w_{\b}}^{(i)}(\hat{V}_{+}(\tau), u')|e^{\frac{|\hat{V}_{+}(\tau)-u'|^{2}}{16}}, \quad |\hat{V}_{+}(\tau)-u'|\ge \frac{N}{2}.
		\end{split}
	\end{equation}
	We also note for $i=1,2$,
	\begin{equation}\label{Eq3.21}
		\begin{split}
			\int_{\mathbb{R}^3}|\mathtt{k}_{w_{\b}}^{(i)}(V_{+}(s), u)|e^{\frac{|V_{+}(s)-u|^{2}}{16}}du&\leq \frac{C}{1+|V_+(s)|},\\
			\int_{\mathbb{R}^3}|\mathtt{k}_{w_{\b}}^{(i)}(\hat{V}_{+}(\tau), u')|e^{\frac{|\hat{V}_{+}(\tau)-u'|^{2}}{16}}\mathrm{d}u'&\leq \frac{C}{1+|\hat{V}_{+}(\tau)|}.
		\end{split}
	\end{equation}
	Using \eqref{Eq3.20}--\eqref{Eq3.21}, one has 
	\begin{align}\label{Eq3.22}
		&\int_0^te^{-\tilde{\nu}_0(t-s)}\mathrm{d}s\int_{0}^se^{-\tilde{\nu}_0(s-\tau)}\mathrm{d}\tau
		 \Big\{\int_{|u|\ge 2N}\int_{\mathbb{R}^3}+\int_{|u|\le 2N}\int_{|u'|\ge 3N}\Big\}|h_{+}(\tau,\hat{X}_{+}(\tau),u')|   \nonumber\\
		 &\quad \times|\mathtt{k}_{w_{\b}}^{(2)}(V_{+}(s),u)\mathtt{k}_{w_{\b}}^{(2)}(\hat{V}_{+}(\tau),u')|
		 \mathrm{d}u'\mathrm{d}u\nonumber\\
		&\le Ce^{-\frac{N^2}{64}}\sup_{0\le s\le t}\|\mathbf{h}(s)\|_{L^{\infty}}.
	\end{align}
	\text{Case 3:} $|v|\le N, |u|\le 2N, |u'|\le 3N, s-\f{1}{N}\le \tau\le s$. In this case, it is easy to see that 
	\begin{align}\label{Eq3.23}
		&\int_0^te^{-\tilde{\nu}_0(t-s)}\mathrm{d}s\int_{s-\f{1}{N}}^se^{-\tilde{\nu}_0(s-\tau)}\mathrm{d}\tau
		\int_{|u|\le 2N}\int_{|u'|\le 3N}|h_{+}(\tau,\hat{X}_{+}(\tau),u')|\nonumber\\
		&\quad\times |\mathtt{k}_{w_{\b}}^{(2)}(V_{+}(s),u)\mathtt{k}_{w_{\b}}^{(2)}(\hat{V}_{+}(\tau),u')| \mathrm{d}u'\mathrm{d}u\nonumber\\
		&\le \f{C}{N} \sup_{0\le s\le t}\|\mathbf{h}(s)\|_{L^{\infty}}.
	\end{align}
	\text{Case 4:} $|v|\le N, |u|\le 2N, |u'|\le 3N, 0\le \tau \le s-\f{1}{N}$. A direct calculation shows that
	\begin{align}\label{Eq3.24}
		&\int_0^te^{-\tilde{\nu}_0(t-s)}\mathrm{d}s\int_{0}^{s-\f{1}{N}}e^{-\tilde{\nu}_0(s-\tau)}\mathrm{d}\tau
		\int_{|u|\le 2N}\int_{|u'|\le 3N}|h_{+}(\tau,\hat{X}_{+}(\tau),u')|\nonumber\\
		&\quad  \times |\mathtt{k}_{w_{\b}}^{(2)}(V_{+}(s),u)\mathtt{k}_{w_{\b}}^{(2)}(\hat{V}_{+}(\tau),u')|\mathrm{d}u'\mathrm{d}u\nonumber\\
		&\le 	\int_0^te^{-\tilde{\nu}_0(t-s)}\mathrm{d}s\int_{0}^{s-\f{1}{N}}e^{-\tilde{\nu}_0(s-\tau)}\mathrm{d}\tau\Big(\int_{|u|\le 2N}\int_{|u'|\le 3N} |h_{+}(\tau,\hat{X}_{+}(\tau),u')|^2\mathrm{d}u'\mathrm{d}u\Big)^{\frac{1}{2}}\nonumber\\
		&\le 	C_N\int_0^te^{-\tilde{\nu}_0(t-s)}\mathrm{d}s\int_{0}^{s-\f{1}{N}}e^{-\tilde{\nu}_0(s-\tau)}\mathrm{d}\tau\Big(\int_{|u|\le 2N}\int_{|u'|\le 3N} |f_{+}(\tau,\hat{X}_{+}(\tau),u')|^2\mathrm{d}u'\mathrm{d}u\Big)^{\frac{1}{2}}.
	\end{align}
	Let $y=\hat{X}_{+}(\tau)\equiv X_{+}(\tau;s,X_{+}(s),u)$, then it follows from Corollary \ref{Coro2.8} that 
	\begin{align}\label{Eq3.25}
		\Big|\frac{dy}{du}\Big|\ge \frac{1}{2}(s-\tau)^3, \quad |y-X_{+}(s)|\le C(s-\tau)N.
	\end{align}
	Using Lemma \ref{Lem2.3}, we have
	\begin{align}\label{Eq3.26}
		&\int_{|u|\le 2N}\int_{|u'|\le 3N} |f_{+}(\tau,\hat{X}_{+}(\tau),u')|^2\mathrm{d}u'\mathrm{d}u\nonumber\\
		=&\int_{|u|\le 2N}\int_{|u'|\le 3N} |f_{+}(\tau,\hat{X}_{+}(\tau),u')|^2\mathbf{1}_{\{|F_{+}(\tau,\hat{X}_{+}(\tau),u')-\mu(u')|\ge \mu(u')\}}\mathrm{d}u'\mathrm{d}u\nonumber\\
		&+\int_{|u|\le 2N}\int_{|u'|\le 3N} |f_{+}(\tau,\hat{X}_{+}(\tau),u')|^2\mathbf{1}_{\{|F_{+}(\tau,\hat{X}_{+}(\tau),u')-\mu(u')|< \mu(u')\}}\mathrm{d}u'\mathrm{d}u\nonumber\\
		\le & C_N\|h_{+}(\tau)\|_{L^{\infty}}\int_{|u|\le 2N}\int_{|u'|\le 3N} |F_{+}(\tau,y,u')-\mu(u')|\mathbf{1}_{\{|F_{+}(\tau,y,u')-\mu(u')|\ge \mu(u')\}}\mathrm{d}u'\mathrm{d}u\nonumber\\
		&+ C_N\int_{|u|\le 2N}\int_{|u'|\le 3N} \frac{|F_{+}(\tau,y,u')-\mu(u')|^2}{\mu(u')}\mathbf{1}_{\{|F_{+}(\tau,y,u')-\mu(u')|< \mu(u')\}}\mathrm{d}u'\mathrm{d}u\nonumber\\
		\le & C_N\frac{1+(s-\tau)^3}{(s-\tau)^3}\|h_{+}(\tau)\|_{L^{\infty}}\int_{\mathbb{T}^3}\int_{|u'|\le 3N} |F_{+}(\tau,y,u')-\mu(u')|\mathbf{1}_{\{|F_{+}(\tau,y,u')-\mu(u')|\ge \mu(u')\}}\mathrm{d}u'\mathrm{d}y\nonumber\\
		&+C_N\frac{1+(s-\tau)^3}{(s-\tau)^3}\int_{\mathbb{T}^3}\int_{|u'|\le 3N} \frac{|F_{+}(\tau,y,u')-\mu(u')|^2}{\mu(u')}\mathbf{1}_{\{|F_{+}(\tau,y,u')-\mu(u')|< \mu(u')\}}\mathrm{d}u'\mathrm{d}y\nonumber\\
		\le & C_{N}N^3\Big(\|h_{+}(\tau)\|_{L^{\infty}}\mathcal{E}(\mathbf{F}_0)+\mathcal{E}(\mathbf{F}_0)\Big),
	\end{align}
	which yields that 
	\begin{align}\label{Eq3.27}
		\Big(\int_{|u|\le 2N}\int_{|u'|\le 3N} |f_{+}(\tau,\hat{X}_{+}(\tau),u')|^2\mathrm{d}u'\mathrm{d}u\Big)^{\frac{1}{2}}\le C_{N}\Big(\|h_{+}(\tau)\|^{\frac{1}{2}}_{L^{\infty}}\sqrt{\mathcal{E}(\mathbf{F}_0)}+\sqrt{\mathcal{E}(\mathbf{F}_0)}\Big).
	\end{align}
	Combining \eqref{Eq3.24} and \eqref{Eq3.27}, we have
	\begin{align}\label{Eq3.28}
		&\int_0^te^{-\tilde{\nu}_0(t-s)}\mathrm{d}s\int_{0}^{s-\f{1}{N}}e^{-\tilde{\nu}_0(s-\tau)}\mathrm{d}\tau\nonumber\\
		&\quad \times\int_{|u|\le 2N}\int_{|u'|\le 3N}|\mathtt{k}_{w_{\b}}^{(2)}(V_{+}(s),u)\mathtt{k}_{w_{\b}}^{(2)}(\hat{V}_{+}(\tau),u')|\cdot |h_{+}(\tau,\hat{X}_{+}(\tau),u')|\mathrm{d}u'\mathrm{d}u\nonumber\\
		\le &\f{1}{N} \|h_{+}(\tau)\|_{L^{\infty}}+C_{N}\Big(\mathcal{E}(\mathbf{F}_0)+\sqrt{\mathcal{E}(\mathbf{F}_0)}\Big).
	\end{align}
Then it is a consequence of	\eqref{Eq3.19}, \eqref{Eq3.22}-\eqref{Eq3.23} and \eqref{Eq3.28} that 
	\begin{align}\label{Eq3.29}
		I_1&\le \frac{C}{N}\sup_{0\le s\le t}\|\mathbf{h}(s)\|_{L^{\infty}}+C_{N}\sqrt{\mathcal{E}(\mathbf{F}_0)}.
	\end{align}
	An argument similar to the one used in \eqref{Eq3.29} shows that
	\begin{align}\label{Eq3.30}
		\sum_{i=1}^{4}I_i&\le \frac{C}{N}\sup_{0\le s\le t}\|\mathbf{h}(s)\|_{L^{\infty}}+C_{N}\sqrt{\mathcal{E}(\mathbf{F}_0)}.
	\end{align}
	Combining \eqref{Eq3.15} and \eqref{Eq3.30}, one has 
	\begin{align}\label{E3.31}
		|h_{+}(t,x,v)|
		&\le C\|\mathbf{h}_{0}\|_{L^{\infty}}+\frac{C}{N}\sup_{0\le s\le t}\|\mathbf{h}(s)\|_{L^{\infty}}+C_{N}\sqrt{\mathcal{E}(\mathbf{F}_0)}\nonumber\\
		&\quad +C\sup_{\substack{0\le s\le t\\ y \in \mathbb{T}^3}}\Big\{\|\mathbf{h}(s)\|^{\f{3}{2}}_{L^{\infty}}\cdot \Big(\int_{\mathbb{R}^3}|\mathbf{\tilde{h}}(s,y,u)|\mathrm{d}u\Big)^{\f{1}{2}}\Big\}.
	\end{align}
	For $h_{-}(t,x,v)$, following an argument similar to the proof of \eqref{E3.31}, one gets
	\begin{align}\label{E3.32}
		|h_{-}(t,x,v)|
		&\le C\|\mathbf{h}_{0}\|_{L^{\infty}}+\frac{C}{N}\sup_{0\le s\le t}\|\mathbf{h}(s)\|_{L^{\infty}}+C_{N}\sqrt{\mathcal{E}(\mathbf{F}_0)}\nonumber\\
		&\quad +C\sup_{\substack{0\le s\le t\\ y \in \mathbb{T}^3}}\Big\{\|\mathbf{h}(s)\|^{\f{3}{2}}_{L^{\infty}}\cdot \Big(\int_{\mathbb{R}^3}|\mathbf{\tilde{h}}(s,y,u)|\mathrm{d}u\Big)^{\f{1}{2}}\Big\}.
	\end{align}
	Choosing $N$ sufficiently large in \eqref{E3.31} and \eqref{E3.32} so that $\frac{C}{N}\leq \frac{1}{2}$, one obtains that for $\b \geq 4$ and $\b_1\geq 0$,
	\begin{align}\label{4.31}
		\sup_{0\leq s\leq t}\|\mathbf{h}(s)\|_{L^\infty}
	\le& C\Big(\|\mathbf{h}_{0}\|_{L^{\infty}}+\sqrt{\mathcal{E}(\mathbf{F}_0)}\Big)
	+C\sup_{\substack{0\le s\le t\\ y \in \mathbb{T}^3}}\Big\{\|\mathbf{h}(s)\|^{\frac{3}{2}}_{L^{\infty}}\cdot \Big(\int_{\mathbb{R}^3}|\mathbf{\tilde{h}}(s,y,u)|\mathrm{d}u\Big)^{\frac{1}{2}}\Big\}.
	\end{align}
	
Using Proposition \ref{prop2.1}, one has that for $\b \geq 4$ and $\b_1\geq 0$,
	\begin{align*}
		\sup_{\substack{0\le s\le t_1\\ y \in \mathbb{T}^3}}\Big\{\|\mathbf{h}(s)\|^{\frac{3}{2}}_{L^{\infty}}\cdot \Big(\int_{\mathbb{R}^3}|\mathbf{\tilde{h}}(s,y,u)|\mathrm{d}u\Big)^{\frac{1}{2}}\Big\}\leq C\sup_{0\leq s\leq t_1}\|\mathbf{h}(s)\|_{L^\infty}^2
		\leq C \|\mathbf{h}_{0}\|_{L^{\infty}}^2,
	\end{align*}
which, together with \eqref{4.31}, yields \eqref{E4.5-1}.
	Therefore the proof of Lemma \ref{Lem3.1} is completed.
\end{proof}


\subsection{$L^{\infty}_xL^1_v$-estimate}

In this part, we mainly focus on the estimate of 
\begin{align*}
	\int_{\mathbb{R}^3}|\mathbf{\tilde{h}}(t,x,v)|\mathrm{d}v.
\end{align*}
Motivated by \cite{Duan17ARMA}, if $\mathcal{E}(\mathbf{F}_0)+\|\mathbf{\tilde{h}}_0\|_{L^1_xL^\infty_v}$ is small, due to the hyperbolicity of the VPB system, one should be able to show that  $\int_{\mathbb{R}^3}|\mathbf{\tilde{h}}(t,x,v)|\mathrm{d}v$ is small for $t\geq t_1$, even though it could be initially large, i.e.,  $\int_{\mathbb{R}^3}|\mathbf{\tilde{h}}_{0}(x,v)|\mathrm{d}v$ is large.  Indeed, we have the following lemma which plays a key role in this paper.
\begin{lemma}\label{Lem3.2}
	Under the condition \eqref{Eq2.21}, it holds for $0\leq\b_1<\b-4$, and $t\geq t_1$, that
	\begin{align*}
		\int_{\mathbb{R}^3}|\mathbf{\tilde{h}}(t,x,v)|\mathrm{d}v
		&\le C\Big(\|\mathbf{h}_0\|^{\frac{3}{\b-\b_1}}_{L^{\infty}}\|\mathbf{\tilde{h}}_0\|^{1-\frac{3}{\b-\b_1}}_{L^1_x L_v^{\infty}}+(\|\mathbf{h}_0\|_{L^{\infty}}+1)^6\|\mathbf{\tilde{h}}_0\|_{L^1_x L_v^{\infty}}\Big)
		 +C_{N, \lambda}\sqrt{\mathcal{E}(\mathbf{F}_0)}\nonumber\\
		&\quad 
		+C(\lambda+\frac{1}{N^{\b-\b_1-4}}) (\sup_{0\le s\le t}\|\mathbf{h}(s)\|_{L^{\infty}}+\sup_{0\le s\le t}\|\mathbf{h}(s)\|^2_{L^{\infty}}),
	\end{align*}
	where $\l>0$ and  $N\geq1$ will be chosen later. 
\end{lemma}
\begin{proof}
	\medskip
	\noindent{ 1.} 
	From $\eqref{Eq1.17}_1$, one has
	\begin{align}\label{Eq3.31}
		(\partial_t+v \cdot \nabla _x\mp\nabla_x \phi \cdot \nabla_v)\tilde{h}_{\pm}+\tilde{\nu}_{\pm,1}\tilde{h}_{\pm}= w_{\b_1}\mathrm{K}^{\pm}\mathbf{f}+w_{\b_1}\mathrm{\Gamma}^{\pm}(\mathbf{f},\mathbf{f})\mp\nabla_x \phi \cdot v w_{\b_1}\sqrt{\mu}
	\end{align}
	where 
	\begin{align}\label{Eq3.32}
		\tilde{\nu}_{\pm,1}(t,x,v):=\frac{\sigma_0}{(1+t)^{2}}|v|^2+\nu(v)\pm\nabla_x\phi\cdot v\Big(\frac{1}{2}+\frac{\beta_1}{1+|v|^2}+\frac{2\sigma_0}{1+t}\Big).
	\end{align}
	Analogous to \eqref{Eq3.10}, we have
	\begin{align}\label{Eq3.33}
		|\tilde{\nu}_{+,1}(t, x, v)|\ge \frac{1}{2}\nu(v)\ge \tilde{\nu}_0.
	\end{align}
It suffices to prove the lemma for $\tilde{h}_{+}(t,x,v)$, since $\tilde{h}_{-}(t,x,v)$ can be obtained similarly.
	
	The mild solution of \eqref{Eq3.31} takes the form
	\begin{align}\label{Eq3.34}
		\tilde{h}_{+}(t,x,v)
		&=\tilde{h}_{+, 0}(X_{\pm}(0),V_{\pm}(0))e^{-\int_0^t\tilde{\nu}_{+,1}(\tau, X_{+}(\tau), V_{+}(\tau))\mathrm{d}\tau}\nonumber\\
		&\quad +\int_0^te^{-\int_s^t\tilde{\nu}_{+,1}(\tau, X_{+}(\tau), V_{+}(\tau))\mathrm{d}\tau}\Big(w_{\b_1}\mathrm{K}^{+}\mathbf{f}\Big)(s,X_{\pm}(s),V_{\pm}(s))\mathrm{d}s\nonumber\\
		&\quad -\int_0^te^{-\int_s^t\tilde{\nu}_{+,1}(\tau, X_{+}(\tau), V_{+}(\tau))\mathrm{d}\tau}\Big(\nabla_x\phi\cdot vw_{\b_1}\sqrt{\mu}\Big)(s,X_{+}(s),V_{+}(s))\mathrm{d}s\nonumber\\
		&\quad +\int_0^te^{-\int_s^t\tilde{\nu}_{+,1}(\tau, X_{+}(\tau), V_{+}(\tau))\mathrm{d}\tau}\Big(w_{\b_1}\mathrm{\Gamma}^{+}(\mathbf{f},\mathbf{f})\Big)(s,X_{+}(s),V_{+}(s))\mathrm{d}s,
	\end{align}
	where we have denoted $\tilde{\nu}_{+,1}(\tau):=\tilde{\nu}_{+,1}(\tau, X_{+}(\tau), V_{+}(\tau))$. Similar to \eqref{Eq3.14}, we denote $\mathtt{k}_{w_{\b_1}}^{(i)}(v,u)=\mathtt{k}^{(i)}(v, u)\cdot\frac{w_{\b_1}(t, v)}{w_{\b_1}(t, u)}$ ($i=1,2$). Using \eqref{Eq3.33}-\eqref{Eq3.34}, one has
	\begin{align}\label{Eq3.35}
		\int_{\mathbb{R}^3}|\tilde{h}_{+}(t,x,v)|\mathrm{d}v\le& e^{-\tilde{\nu}_0t}\int_{\mathbb{R}^3}|\tilde{h}_{+, 0}(X_{+}(0),V_{+}(0))|\mathrm{d}v \nonumber\\
		&+\int_0^t\int_{\mathbb{R}^3}e^{-\int_s^t\tilde{\nu}_{+,1}(\tau)\mathrm{d}\tau}\int_{\mathbb{R}^3}|\mathtt{k}_{w_{\b_1}}^{(2)}(V_{+}(s),v')\tilde{h}_{+}(s,X_{+}(s),v')|\mathrm{d}v'\mathrm{d}v\mathrm{d}s \nonumber\\
		&+\int_0^t\int_{\mathbb{R}^3}e^{-\int_s^t\tilde{\nu}_{+,1}(\tau)\mathrm{d}\tau}\int_{\mathbb{R}^3}|\mathtt{k}_{w_{\b_1}}^{(1)}(V_{+}(s),v')\tilde{h}_{-}(s,X_{+}(s),v')|\mathrm{d}v'\mathrm{d}v\mathrm{d}s \nonumber\\
		&+\int_0^t\int_{\mathbb{R}^3}e^{-\int_s^t\tilde{\nu}_{+,1}(\tau)\mathrm{d}\tau}|\Big(w_{\b_1}\mathrm{\Gamma}^+(\mathbf{f},\mathbf{f})\Big) (s,X_{+}(s),V_{+}(s))|\mathrm{d}v\mathrm{d}s\nonumber\\
		&+\int_0^t\int_{\mathbb{R}^3}e^{-\int_s^t\tilde{\nu}_{+,1}(\tau)\mathrm{d}\tau}\nabla_x \phi(s,X_{+}(s)) \cdot V_{+}(s) w_{\b_1}(s, V_{+}(s))\sqrt{\mu(V_{+}(s))}\mathrm{d}v\mathrm{d}s\nonumber\\
		&:=H_1+H_2^{+}+H_{2}^{-}+H_3+H_4.
	\end{align}

		\medskip
	\noindent{ 2.} It is noted that
	\begin{align}\label{Eq3.36}
		H_2^{+}&\le \int_{t-\lambda}^t\int_{\mathbb{R}^3}e^{-\int_s^t\tilde{\nu}_{+,1}(\tau)\mathrm{d}\tau}\int_{\mathbb{R}^3}|\mathtt{k}_{w_{\b_1}}^{(2)}(V_{+}(s),v')\tilde{h}_{+}(s,X_{+}(s),v')|\mathrm{d}v'\mathrm{d}v\mathrm{d}s\nonumber\\
		&\quad +\int_{0}^{t-\lambda}\int_{\mathbb{R}^3}e^{-\int_s^t\tilde{\nu}_{+,1}(\tau)\mathrm{d}\tau}\int_{\mathbb{R}^3}|\mathtt{k}_{w_{\b_1}}^{(2)}(V_{+}(s),v')\tilde{h}_{+}(s,X_{+}(s),v')|\mathrm{d}v'\mathrm{d}v\mathrm{d}s\nonumber\\
		&:=H_{21}^{+}+H_{22}^{+}.
	\end{align}

	Using the fact that $1+|v|\cong 1+|V_{\pm}(s)|$ for any $0\le s\le t$, we have
	\begin{align}\label{Eq3.37}
		H_{21}^{+}&\le C\sup_{0\le s\le t}\|\mathbf{h}(s)\|_{L^{\infty}}\int_{t-\lambda}^t\int_{\mathbb{R}^3} e^{-\int_s^t\tilde{\nu}_{+,1}(\tau)\mathrm{d}\tau}(1+|V_{+}(s)|^2)^{-\f{\b-\b_1}{2}}\mathrm{d}v \mathrm{d}s\nonumber\\
		&\le C\lambda \sup_{0\le s\le t}\|\mathbf{h}(s)\|_{L^{\infty}}.
	\end{align}

	To estimate $H_{22}^{+}$, we notice that
	\begin{align}\label{Eq3.38}
		H_{22}^{+}=&\int_0^{t-\lambda}\int_{|v|\geq N} e^{-\int_s^t\tilde{\nu}_{+,1}(\tau)\mathrm{d}\tau}\int_{\mathbb{R}^3}|\mathtt{k}_{w_{\b_1}}^{(2)}(V_{+}(s),v')\tilde{h}_{+}(s,X_{+}(s), v')|\mathrm{d}v'\mathrm{d}v \mathrm{d}s\nonumber\\
		&+\int_0^{t-\lambda}\int_{|v|\le N} e^{-\int_s^t\tilde{\nu}_{+,1}(\tau)\mathrm{d}\tau}\int_{|v'|\ge 2N}|\mathtt{k}_{w_{\b_1}}^{(2)}(V_{+}(s),v')\tilde{h}_{+}(s,X_{+}(s), v')|\mathrm{d}v'\mathrm{d}v \mathrm{d}s\nonumber\\
		&+\int_0^{t-\lambda}\int_{|v|\le N} e^{-\int_s^t\tilde{\nu}_{+,1}(\tau)\mathrm{d}\tau}\int_{|v'|\le 2N}|\mathtt{k}_{w_{\b_1}}^{(2)}(V_{+}(s),v')\tilde{h}_{+}(s,X_{+}(s), v')|\mathrm{d}v'\mathrm{d}v \mathrm{d}s\nonumber\\
		:=&H_{221}^{+}+H_{222}^{+}+H_{223}^{+}.
	\end{align}

	It is straightforward to show that
	\begin{align}\label{Eq3.39}
		H_{221}^{+}&\le \sup_{0\le s\le t}\|\mathbf{h}(s)\|_{L^{\infty}} \int_0^{t-\lambda}\int_{|v|\geq N} e^{-\int_s^t\tilde{\nu}_{+,1}(\tau)\mathrm{d}\tau}(1+|V_{+}(s)|)^{-(\b-\b_1)-1}\mathrm{d}v \mathrm{d}s\nonumber\\
		&\le C\sup_{0\le s\le t}\|\mathbf{h}(s)\|_{L^{\infty}} \int_0^{t-\lambda}\int_{|v|\geq N} e^{-\tilde{\nu}_0(t-s)}(1+|v|)^{-(\b-\b_1)-1}\mathrm{d}v \mathrm{d}s\nonumber\\
		&\le \frac{C}{N^{\beta-\beta_{1}-2}}\sup_{0\le s\le t}\|\mathbf{h}(s)\|_{L^{\infty}}.
	\end{align}

	For $H_{222}^{+}$, it holds that $|V_{+}(s)-v'|\ge |v'-v|-|V_{+}(s)-v|\ge |v'|-|v|-|V_{+}(s)-v|\ge \frac{N}{2}$. Then we have
	\begin{align}\label{Eq3.40}
		H_{222}^{+}&\le e^{-\frac{N^2}{64}}\sup_{0\le s\le t}\|\mathbf{h}(s)\|_{L^{\infty}} \int_0^{t-\lambda}\int_{|v|\le N} e^{-\tilde{\nu}_0(t-s)}(1+|V_{+}(s)|)^{-(\b-\b_1)}\mathrm{d}v \mathrm{d}s\nonumber\\
		&\quad\times \int_{|v'|\ge 2N}|\mathtt{k}_{w_{\b}}^{(2)}(V_{+}(s),v')|e^{\frac{|V_{+}(s)-v'|^2}{16}}\mathrm{d}v'\nonumber\\
		&\le Ce^{-\frac{N^2}{64}}\sup_{0\le s\le t}\|\mathbf{h}(s)\|_{L^{\infty}} \int_0^{t-\lambda}\int_{|v|\le N} e^{-\tilde{\nu}_0(t-s)}(1+|V_{+}(s)|)^{-(\b-\b_1)}\mathrm{d}v \mathrm{d}s\nonumber\\
		&\le Ce^{-\frac{N^2}{64}}\sup_{0\le s\le t}\|\mathbf{h}(s)\|_{L^{\infty}}.
	\end{align}

	For $H_{223}^{+}$, similar argument as in \eqref{Eq3.26}, we have 
	\begin{align}\label{Eq3.41}
		H_{223}^{+}
		\le & C_N\int_0^{t-\lambda}e^{-\tilde{\nu}_0(t-s)}\mathrm{d}s\Big(\int_{|v|\le N}\int_{|v'|\le 2N}|f_{+}(s,X_{+}(s), v')|^2\mathrm{d}v' \mathrm{d}v\Big)^{\frac{1}{2}} \nonumber\\
		\le & C_{N,\lambda}\Big(\|h_{+}(\tau)\|^{\frac{1}{2}}_{L^{\infty}}\sqrt{\mathcal{E}(\mathbf{F}_0)}+\sqrt{\mathcal{E}(\mathbf{F}_0)}\Big)\nonumber\\
		\le & \lambda \sup_{0\le s\le t}\|\mathbf{h}(s)\|_{L^{\infty}}+ C_{N, \lambda}\sqrt{\mathcal{E}(\mathbf{F}_0)}.
	\end{align}

	Combining \eqref{Eq3.36}-\eqref{Eq3.41}, we obtain
	\begin{align}\label{Eq3.42}
		H_2^{+}\le C(\lambda+\frac{1}{N}) \sup_{0\le s\le t}\|\mathbf{h}(s)\|_{L^{\infty}}+ C_{N, \lambda}\sqrt{\mathcal{E}(\mathbf{F}_0)}.
	\end{align}

Similar to \eqref{Eq3.42}, one has
	\begin{align}\label{Eq3.42-}
	H_2^{-}\le C(\lambda+\frac{1}{N}) \sup_{0\le s\le t}\|\mathbf{h}(s)\|_{L^{\infty}}+ C_{N,\lambda}\sqrt{\mathcal{E}(\mathbf{F}_0)}.
\end{align}

\medskip
\noindent{ 3.} For $H_3$, we note
	\begin{align}\label{Eq3.43}
		H_3&\le C\int_0^te^{-\tilde{\nu}_0(t-s)}\int_{\mathbb{R}^3}w_{\b_1}(s, V_{+}(s))\int_{\mathbb{R}^3\times \mathbb{S}^2}B(V_{+}(s)-u, \t)e^{-\frac{|u|^2}{4}}\nonumber\\
		&\qquad\times |f_{+}(s,X_{+}(s),V_{+}(s))|\left(|f_{+}(s,X_{+}(s),u)|+|f_{-}(s,X_{+}(s),u)|\right)\mathrm{d}\omega \mathrm{d}u \mathrm{d}v \mathrm{d}s\nonumber\\
		&\quad+C\int_0^te^{-\tilde{\nu}_0(t-s)}\int_{\mathbb{R}^3}w_{\b_1}(s, V_{+}(s))\int_{\mathbb{R}^3\times \mathbb{S}^2}B(V_{+}(s)-u, \t)e^{-\frac{|u|^2}{4}}\nonumber\\
		&\qquad \times |f_{+}(s,X_{+}(s),V'_{+}(s))|\left(|f_{+}(s,X_{+}(s),u')|+|f_{-}(s,X_{+}(s),u')|\right)\mathrm{d}\omega \mathrm{d}u \mathrm{d}v \mathrm{d}s\nonumber\\
		&:=H_{31}+H_{32}.
	\end{align}
	For $H_{31}$, it is clear to see that
	\begin{align}\label{Eq3.44}
		H_{31}&\le C\int_{t-\lambda}^te^{-\tilde{\nu}_0(t-s)}\mathrm{d}s\int_{\mathbb{R}^3}w_{\b_1}(s, V_{+}(s))\int_{\mathbb{R}^3}|V_{+}(s)-u|^{\gamma}e^{-\frac{|u|^2}{4}}\nonumber\\
		&\qquad \times |f_{+}(s,X_{+}(s),V_{+}(s))|\left(|f_{+}(s,X_{+}(s),u)|+|f_{-}(s,X_{+}(s),u)|\right) \mathrm{d}u \mathrm{d}v \nonumber\\
		&\quad+C\int_{0}^{t-\lambda}e^{-\tilde{\nu}_0(t-s)}\mathrm{d}s\Big\{\int_{|v|\ge N}\int_{\mathbb{R}^3}+\int_{|v|\le N}\int_{|u|\ge N}+\int_{|v|\le N}\int_{|u|\le N}\Big\}|V_{+}(s)-u|^{\gamma}e^{-\frac{|u|^2}{4}}\nonumber\\
		&\qquad \times w_{\b_1}(s, V_{+}(s))|f_{+}(s,X_{+}(s),V_{+}(s))|\left(|f_{+}(s,X_{+}(s),u)|+|f_{-}(s,X_{+}(s),u)|\right)\mathrm{d}u \mathrm{d}v \nonumber\\
		&:=H_{311}+H_{312}+H_{313}+H_{314}.
	\end{align}
	Noting $\b-\b_1>4$, one has
	\begin{align}\label{Eq3.45}
		H_{311}&\le C\int_{t-\lambda}^te^{-\tilde{\nu}_0(t-s)}\mathrm{d}s\int_{\mathbb{R}^3}|\mathbf{\tilde{h}}(s, X_{+}(s), V_{+}(s))|\mathrm{d}v\int_{\mathbb{R}^3\times \mathbb{S}^2}|(V_{+}(s)-u)|^{\gamma}e^{-\frac{|u|^2}{4}}\mathrm{d}\omega \mathrm{d}u\|\mathbf{h}(s)\|_{L^{\infty}}  \nonumber\\
		&\leq C\|\mathbf{h}(s)\|^2_{L^{\infty}}\int_{t-\lambda}^te^{-\tilde{\nu}_0(t-s)}\mathrm{d}s\int_{\mathbb{R}^3}\nu(V_{+}(s))(1+|V_{+}(s)|)^{-(\b-\b_1)}\mathrm{d}v\nonumber\\
		&\le C\lambda \sup_{0\le s\le t}\|\mathbf{h}(s)\|^2_{L^{\infty}},
	\end{align}
	and 
	\begin{align}\label{Eq3.46}
		H_{312}&\le C\sup_{0\le s\le t}\|\mathbf{h}(s)\|^2_{L^{\infty}}\int_{0}^{t-\lambda}e^{-\tilde{\nu}_0(t-s)}\mathrm{d}s\int_{|v|\ge N}\nu(V_{+}(s))(1+|V_{+}(s)|)^{-(\b-\b_1)}\mathrm{d}v\nonumber\\
		&\le \frac{C}{N^{\b-\b_1-4}} \sup_{0\le s\le t}\|\mathbf{h}(s)\|^2_{L^{\infty}}.
	\end{align}
	For $H_{313}$, it is direct to have
	\begin{align}\label{Eq3.47}
		H_{313}&\le Ce^{-\frac{N^2}{8}}\sup_{0\le s\le t}\|\mathbf{h}(s)\|^2_{L^{\infty}}\int_{0}^{t-\lambda}e^{-\tilde{\nu}_0(t-s)}\mathrm{d}s\int_{|v|\le N}\nu(V_{+}(s))(1+|V_{+}(s)|)^{-(\b-\b_1)}\mathrm{d}v\nonumber\\
		&\le Ce^{-\frac{N^2}{8}}\sup_{0\le s\le t}\|\mathbf{h}(s)\|^2_{L^{\infty}}.
	\end{align}
	For $H_{314}$, by an argument similar to that in \eqref{Eq3.26}, we obtain
	\begin{align}\label{Eq3.48}
		H_{314}&\le C\sup_{0\le s\le t}\|\mathbf{h}(s)\|_{L^{\infty}}\int_{0}^{t-\lambda}e^{-\tilde{\nu}_0(t-s)}\mathrm{d}s\nonumber\\
		&\quad \times \int_{|v|\le N}\int_{|u|\le N}|V_{+}(s)-u|^{\gamma}(1+|V_{+}(s)|)^{-(\b-\b_1)}e^{-\frac{|u|^2}{4}}|\mathbf{f}(s,X_{+}(s),u)|\mathrm{d}u\mathrm{d}v\nonumber\\
		&\le C\sup_{0\le s\le t}\|\mathbf{h}(s)\|_{L^{\infty}}\int_{0}^{t-\lambda}e^{-\tilde{\nu}_0(t-s)}\mathrm{d}s\Big(\int_{|v|\le N}\int_{|u|\le N}|\mathbf{f}(s,X_{+}(s),u)|^2\mathrm{d}u\mathrm{d}v\Big)^{\frac{1}{2}}\nonumber\\
		&\quad \times \Big(\int_{|v|\le N}\int_{|u|\le N}|V_{+}(s)-u|^{2\gamma}(1+|V_{+}(s)|)^{-2(\b-\b_1)}e^{-\frac{|u|^2}{2}}\mathrm{d}u\mathrm{d}v\Big)^{\frac{1}{2}}\nonumber\\
		&\le C\lambda \sup_{0\le s\le t}\|\mathbf{h}(s)\|^2_{L^{\infty}}+C_{N,\lambda}\mathcal{E}(\mathbf{F}_0).
	\end{align}
	Combining \eqref{Eq3.44} $-$ \eqref{Eq3.48}, we obtain
	\begin{align}\label{Eq3.49}
		H_{31}\le C(\lambda+\frac{1}{N^{\b-\b_1-4}}) \sup_{0\le s\le t}\|\mathbf{h}(s)\|^2_{L^{\infty}}+C_{N,\lambda}\mathcal{E}(\mathbf{F}_0).
	\end{align}

	For $H_{32}$, using the fact $w_{\b_1}(s, V_{+}(s))\leq w_{\b_1}(s, V'_{+}(s))w_{\b_1}(s, u')$, here $V'_{+}(s)$ represents post-collision velocity with respect to pre-collision velocity $V_{+}(s)$,  one has
	\begin{align}\label{Eq3.50}
		H_{32}&\le C\int_{t-\lambda}^te^{-\tilde{\nu}_0(t-s)}\mathrm{d}s\int_{\mathbb{R}^3}\int_{\mathbb{R}^3}|V_{+}(s)-u|^{\gamma}e^{-\frac{|u|^2}{4}}
		|\tilde{h}_{+}(s,X_{+}(s),V'_{+}(s))||\mathbf{\tilde{h}}(s,X_{+}(s),u')| \mathrm{d}u \mathrm{d}v \nonumber\\
		&\quad +C\int_{0}^{t-\lambda}e^{-\tilde{\nu}_0(t-s)}\mathrm{d}s\Big\{\int_{|v|\ge N}\int_{\mathbb{R}^3}\int_{\mathbb{S}^2}+\int_{|v|\le N}\int_{|u|\ge N}\int_{\mathbb{S}^2}+\int_{|v|\le N}\int_{|u|\le N}\int_{\mathbb{S}^2}\Big\}e^{-\frac{|u|^2}{4}}\nonumber\\
		&\qquad \times B(V_{+}(s)-u, \t)|\tilde{h}_{+}(s,X_{+}(s),V'_{+}(s))||\mathbf{\tilde{h}}(s,X_{+}(s),u')|\mathrm{d}\omega \mathrm{d}u\mathrm{d}v \nonumber\\
		&:=H_{321}+H_{322}+H_{323}+H_{324}.
	\end{align}

	Using the fact $(1+|V_{+}(s)|^2)^{\f{\b-\b_1}{2}}\leq (1+|u'|^2)^{\f{\b-\b_1}{2}}(1+|V'_{+}(s)|^2)^{\f{\b-\b_1}{2}}$, one has
	\begin{align}\label{Eq3.51}
		H_{321}&\leq \int_{t-\lambda}^te^{-\tilde{\nu}_0(t-s)}\mathrm{d}s\int_{\mathbb{R}^3}(1+|V_+(s)|^2)^{-\f{\b-\b_1}{2}}\nu(V_+(s)) \mathrm{d}v\|\mathbf{h}(s)\|^2_{L^{\infty}}\nonumber\\
		&\leq C\lambda \sup_{0\le s\le t}\|\mathbf{h}(s)\|^2_{L^{\infty}}.
	\end{align}

	Analogous to \eqref{Eq3.46} $-$ \eqref{Eq3.47}, it holds that
	\begin{align}\label{Eq3.52}
		H_{322}+H_{323}\le \frac{C}{N^{(\b-\b_1)-4}}\sup_{0\le s\le t}\|\mathbf{h}(s)\|^2_{L^{\infty}}.
	\end{align}

	Noting
	\begin{align}\label{Eq3.53}
		\left\{
		\begin{aligned}
			&V'_{+}(s)=V_{+}(s)+\{(u-V_{+}(s))\cdot \omega\}\cdot \omega, \\
			&u'=u-\{(u-V_{+}(s))\cdot \omega\}\cdot \omega,
		\end{aligned}
		\right.
	\end{align}
	we  make a change of variables 
	\begin{align}\label{Eq3.54}
		z=u-V_{+}(s),\quad z_{\shortparallel}=(z\cdot \omega)\omega,\quad z_{\perp}=z-z_{\shortparallel}, \quad \eta=V_{+}(s)+z_{\shortparallel}.
	\end{align}
Owing to
	\begin{align*}
		|B(V_{+}(s)-u, \t)|\leq C|V_{+}(s)-u|^{\gamma}|\cos\t|\leq C(|z_{\shortparallel}|^2+|z_{\perp}|^2)^{\frac{\gamma-1}{2}}|z_{\shortparallel}|,
	\end{align*}
	and 
	\begin{align*}
		\mathrm{d}\omega \mathrm{d}u=\mathrm{d}z_{\shortparallel}\mathrm{d}z_{\perp}\mathrm{d}\omega=2|z_{\shortparallel}|^2d|z_{\shortparallel}|\mathrm{d}\omega\frac{\mathrm{d}z_{\perp}}{|z_{\shortparallel}|^2}=\frac{2}{|z_{\shortparallel}|^2}\mathrm{d}z_{\shortparallel}\mathrm{d}z_{\perp}=\frac{2}{|\eta-V_{+}(s)|^2}\mathrm{d}\eta \mathrm{d}z_{\perp},
	\end{align*}
	then one has
	\begin{align}\label{Eq3.55}
		H_{324}&\le C \int_{0}^{t-\lambda}e^{-\tilde{\nu}_0(t-s)}\mathrm{d}s\int_{|v|\le N}\int_{|u|\le N}\int_{\mathbb{S}^2}B(V_{+}(s)-u, \t)(1+|V_{+}(s)|^2)^{-\f{\b-\b_1}{2}}\nonumber\\
		&\quad \times e^{-\frac{|u|^2}{4}}|h_{+}(s,X_{+}(s),V'_{+}(s))|\cdot \|\mathbf{h}(s)\|_{L^{\infty}}\mathrm{d}\omega \mathrm{d}u\mathrm{d}v \nonumber\\
		&\le C\int_{0}^{t-\lambda}e^{-\tilde{\nu}_0(t-s)}\mathrm{d}s\int_{|v|\le N}\int_{|z|\le 3N}\int_{\mathbb{S}^2}
		e^{-\frac{|z_{\shortparallel}+z_{\perp}+V_{+}(s)|^2}{4}}(|z_{\shortparallel}|^2+|z_{\perp}|^2)^{\frac{\gamma-1}{2}}|z_{\shortparallel}|  \nonumber\\
		&\quad\times(1+|V_{+}(s)|^2)^{-\f{\b-\b_1}{2}}  |h_{+}(s,X_{+}(s),V_{+}(s)+z_{\shortparallel})|\mathrm{d}vd|z_{\shortparallel}|\mathrm{d}z_{\perp} \mathrm{d}\omega\|\mathbf{h}(s)\|_{L^{\infty}}   \nonumber   \\
		&\le C\int_{0}^{t-\lambda}e^{-\tilde{\nu}_0(t-s)}\mathrm{d}s\int_{|v|\le N}\int_{|\eta|\le 5N}\int_{z_{\perp}}
		e^{-\frac{|\eta+z_{\perp}|^2}{4}}(|\eta-V_{+}(s)|^2+|z_{\perp}|^2)^{\frac{\gamma-1}{2}} \nonumber\\
		&\quad\times |\eta-V_{+}(s)|^{-1}(1+|V_{+}(s)|^2)^{-\f{\b-\b_1}{2}}  |h_{+}(s,X_{+}(s),\eta)|\mathrm{d}v \mathrm{d}z_{\perp}\mathrm{d}\eta  \|\mathbf{h}(s)\|_{L^{\infty}}   \nonumber   \\
		&\le C\int_{0}^{t-\lambda}e^{-\tilde{\nu}_0(t-s)}\mathrm{d}s\int_{|v|\le N}\int_{|\eta|\le 5N}|\eta-V_{+}(s)|^{-1}(1+|V_{+}(s)|^2)^{-\f{\b-\b_1}{2}}   \nonumber\\
		&\quad\times|h_{+}(s,X_{+}(s),\eta)|\mathrm{d}v \mathrm{d}\eta  \|\mathbf{h}(s)\|_{L^{\infty}} \nonumber   \\
		&\le C_{N}\int_{0}^{t-\lambda}e^{-\tilde{\nu}_0(t-s)}\mathrm{d}s\|\mathbf{h}(s)\|_{L^{\infty}}\Big(\int_{|v|\le N}\int_{|\eta|\le 5N}|\mathbf{f}(s,X_{+}(s),\eta)|^2\mathrm{d}\eta \mathrm{d}v\Big)^{\f{1}{2}}                                     \nonumber   \\
		&\le C\lambda  \sup_{0\le s\le t}\|\mathbf{h}(s)\|^2_{L^{\infty}}+C_{N,\lambda}\mathcal{E}(\mathbf{F}_0).
	\end{align}

	Hence we have from \eqref{Eq3.50}$-$\eqref{Eq3.55} that
	\begin{align}\label{Eq3.56}
		H_{32}\le C(\lambda+\frac{1}{N^{\b-\b_1-4}}) \sup_{0\le s\le t}\|\mathbf{h}(s)\|^2_{L^{\infty}}+C_{N,\lambda}\mathcal{E}(\mathbf{F}_0).
	\end{align}

	Combining \eqref{Eq3.43}, \eqref{Eq3.49} and \eqref{Eq3.56}, we have 
	\begin{align}\label{Eq3.57}
		H_3\le C(\lambda+\frac{1}{N^{\b-\b_1-4}}) \sup_{0\le s\le t}\|\mathbf{h}(s)\|^2_{L^{\infty}}+C_{N,\lambda}\mathcal{E}(\mathbf{F}_0).
	\end{align}

	\medskip
	\noindent{ 4.}
	For $H_4$, using \eqref{Eq2.16}, one gets that
	\begin{align}\label{Eq3.58}
		H_{4}\le &\int_0^te^{-\tilde{\nu}_0(t-s)}\mu^{\frac{1}{4}}(V_{\pm}(s))\|\nabla_x\phi(s)\|_{L^{\infty}}\mathrm{d}v\mathrm{d}s\nonumber\\
		&\le C\lambda \int_0^te^{-\tilde{\nu}_0(t-s)}\|\mathbf{h}(s)\|_{L^{\infty}}\mathrm{d}s+\frac{C}{\lambda} \int_0^te^{-\tilde{\nu}_0(t-s)}\|\nabla_x\phi(s)\|_{L^{2}}\mathrm{d}s\nonumber\\
		&\le C\lambda\sup_{0\le s\le t}\|\mathbf{h}(s)\|_{L^{\infty}}+\frac{C}{\lambda}\sqrt{\mathcal{E}(\mathbf{F}_0)}.
	\end{align}

	
	\medskip
	\noindent{ 5.}
	It remains to deal with the initial term. For $t\ge t_1>0$, it holds that
	\begin{align}\label{Eq3.60}
		&\int_{\mathbb{R}^3}|\tilde{h}_{+, 0}(X_{+}(0), V_{+}(0))|\mathrm{d}v
		\leq\left\{\int_{|v|\geq \tilde{N}}+\int_{|v|\leq \tilde{N}}\right\}|\tilde{h}_{+, 0}(X_{+}(0), V_{+}(0))|\mathrm{d}v\nonumber
		\\
		\leq& C\|\mathbf{h}_0\|_{L^{\infty}}\int_{|v|\geq \tilde{N}}(1+V_{+}(0))^{-(\b-\b_1)}dv +\int_{|v|\leq \tilde{N}}|\tilde{h}_{+, 0}(X_{+}(0), V_{+}(0))|\mathrm{d}v\nonumber\\
		\leq& C\|\mathbf{h}_0\|_{L^{\infty}}\int_{|v|\geq \tilde{N}}(1+|v|^2)^{-\frac{\b-\b_1}{2}}dv +\frac{C}{t^3}\int_{|y-x|\leq C\tilde{N}t}|\tilde{h}_{+, 0}(y,  V_{+}(0))|\mathrm{d}y\nonumber
		\\
		\leq& C\|\mathbf{h}_0\|_{L^{\infty}}\tilde{N}^{3-(\b-\b_1)}+C\frac{1+\tilde{N}^3t_1^3}{t_1^3}\|\tilde{h}_{+,0}\|_{L^1_x L_v^{\infty}}\nonumber
		\\
		\leq& C\|\mathbf{h}_0\|^{\frac{3}{\b-\b_1}}_{L^{\infty}}\|\tilde{h}_{+,0}\|^{1-\frac{3}{\b-\b_1}}_{L^1_x L_v^{\infty}}+C(\|\mathbf{h}_0\|_{L^\infty}+1)^6\|\tilde{h}_{+,0}\|_{L^1_x L_v^{\infty}},
	\end{align}
	where we have chosen $\tilde{N}=\|\mathbf{h}_0\|_{L^\infty}^{\frac{1}{\b-\b_1}}\|\tilde{h}_{+,0}\|_{L^1_xL^\infty_v}^{-\frac{1}{\b-\b_1}}$ and $	t_1:=
	\frac{1}{16C_1C_2C_3C_4(\|\mathbf{h}_0\|_{\infty}+1)^2}$.
	
\medskip
\noindent{ 6.}
	Inserting \eqref{Eq3.42}, \eqref{Eq3.57}, \eqref{Eq3.58} and \eqref{Eq3.60} into \eqref{Eq3.35} shows 
	\begin{align}\label{Eq3.61}
		\int_{\mathbb{R}^3}|\tilde{h}_{+}(t,x,v)|\mathrm{d}v
		&\le  C\Big(\|\mathbf{h}_0\|^{\frac{3}{\b-\b_1}}_{L^{\infty}}\|\tilde{h}_{+,0}\|^{1-\frac{3}{\b-\b_1}}_{L^1_x L_v^{\infty}}+(\|\mathbf{h}_0\|_{L^{\infty}}+1)^6\|\tilde{h}_{+,0}\|_{L^1_x L_v^{\infty}}\Big)
		+C_{N,\lambda}\sqrt{\mathcal{E}(\mathbf{F}_0)}\nonumber\\
		&\quad 
		+C(\lambda+\frac{1}{N^{\b-\b_1-4}}) \Big(\sup_{0\le s\le t}\|\mathbf{h}(s)\|_{L^{\infty}}+\sup_{0\le s\le t}\|\mathbf{h}(s)\|^2_{L^{\infty}}\Big).
	\end{align}

	For $\int_{\mathbb{R}^3}|\tilde{h}_{-}(t,x,v)|\mathrm{d}v$, similarly, one has
	\begin{align}\label{Eq3.62}
		\int_{\mathbb{R}^3}|\tilde{h}_{-}(t,x,v)|\mathrm{d}v
		&\le  C\Big(\|\mathbf{h}_0\|^{\frac{3}{\b-\b_1}}_{L^{\infty}}\|\tilde{h}_{-,0}\|^{1-\frac{3}{\b-\b_1}}_{L^1_x L_v^{\infty}}+(\|\mathbf{h}_0\|_{L^{\infty}}+1)^6\|\tilde{h}_{-,0}\|_{L^1_x L_v^{\infty}}\Big)
		+C_{N,\lambda}\sqrt{\mathcal{E}(\mathbf{F}_0)}\nonumber\\
		&\quad 
		+C(\lambda+\frac{1}{N^{\b-\b_1-4}}) \Big(\sup_{0\le s\le t}\|\mathbf{h}(s)\|_{L^{\infty}}+\sup_{0\le s\le t}\|\mathbf{h}(s)\|^2_{L^{\infty}}\Big).
	\end{align}

	Therefore we complete the proof of Lemma \ref{Lem3.2} from \eqref{Eq3.61} and \eqref{Eq3.62}.
\end{proof}

\subsection{Uniform $L^{\infty}$- bound}
With the help of  Lemmas \ref{Lem3.1} and \ref{Lem3.2}, one can obtain the following uniform $L^{\infty}$-bound.
\begin{proposition}\label{prop4.3}
	Assume \eqref{Eq2.21}, 	let  $0\leq\b_1<\b-4$, there exists small $\v_1>0$ such that if $\mathcal{E}(\mathbf{F}_0)+\|\mathbf{\tilde{h}}_{0}\|_{L^1_xL^{\infty}_v}\le \varepsilon_1$,
it holds
\begin{align*}
	\sup_{0\leq s\leq t}\|\mathbf{h}(s)\|_{L^\infty}
	\le CM_{0}^2,
\end{align*}
where the positive constant $C\geq1$ depends only on $\gamma$, $\b$ and $\b_1$.
\end{proposition}
\begin{proof}
	In terms of Lemma \ref{Lem3.1}, we make the {\it a priori} assumption
	\begin{align}\label{4.65} 
		\|\mathbf{h}(t)\|_{L^\infty}\leq 2A_0:=
		2C_6\Big\{2M_0^2+\sqrt{\mathcal{E}(\mathbf{F}_0)}\Big\},
	\end{align}
	where the positive constant $C_6\geq1$ is defined in Lemma \ref{Lem3.1}. Then it follows from Lemma \ref{Lem3.1} and the {\it a priori} assumption \eqref{4.65} that 
	\begin{align}\label{4.66}
		\|\mathbf{h}(t)\|_{L^\infty}\leq A_0+C_6(2A_0)^{\f{3}{2}}\cdot\sup_{\substack{t_1\le s\le t\\ y \in \mathbb{T}^3}}\Big(\int_{\mathbb{R}^3}|\mathbf{\tilde{h}}(s,y,u)|\mathrm{d}u\Big)^{\frac{1}{2}}.
	\end{align}

To estimate the second term on the RHS of \eqref{4.65}, it follows from \eqref{4.65} and Lemma \ref{Lem3.2}  that 
\begin{align}\label{4.67}
	\sup_{\substack{t_1\le s\le t\\ y \in \mathbb{T}^3}}\Big(\int_{\mathbb{R}^3}|\mathbf{\tilde{h}}(s,y,u)|\mathrm{d}u\Big)
	\le&  C\Big(M_0^{\frac{3}{\b-\b_1}}\|\mathbf{\tilde{h}}_0\|^{1-\frac{3}{\b-\b_1}}_{L^1_x L_v^{\infty}}+M_0^6\|\mathbf{\tilde{h}}_0\|_{L^1_x L_v^{\infty}}\Big)
	 +C_{N,\lambda}\sqrt{\mathcal{E}(\mathbf{F}_0)} \nonumber\\
	&+C\Big(\lambda+\frac{1}{N^{\b-\b_1-4}}\Big) \Big(2A_0+(2A_0)^2\Big).
\end{align}
Note $0\leq\b_1<\b-4$. One can firstly choose $\lambda$ sufficiently small, then $N\geq1$ large enough, and finally let $	\mathcal{E}(\mathbf{F}_0)+\|\mathbf{\tilde{h}}_{0}\|_{L^1_xL^{\infty}_v}\le \varepsilon_1$ with $\varepsilon_1$ suitably small depending only on $\b$, $\b_1$, $\gamma$ and $M_0$ such that
\begin{align*}
C_6(2A_0)^{\f{3}{2}}\cdot\sup_{\substack{t_1\le s\le t\\ y \in \mathbb{T}^3}}\Big(\int_{\mathbb{R}^3}|\mathbf{\tilde{h}}(s,y,u)|\mathrm{d}u\Big)^{\frac{1}{2}}\leq \frac{3}{4}A_0,
\end{align*}
which, together with \eqref{4.66}, yields immediately that
\begin{equation*}
	\|\mathbf{h}(t)\|_{L^\infty}\leq \f74A_0,
\end{equation*}
for all $t\geq 0$. Hence we have closed the {\it a priori} assumption \eqref{4.65}. Therefore the proof of Proposition \ref{prop4.3} is completed.
\end{proof}

\begin{remark}\label{Rem4.4}
	Assume \eqref{Eq2.21}, Proposition \ref{prop4.3} and Lemma \ref{Lem3.2} yield immediately that 
	\begin{align*}
		\sup_{\substack{t_1\le s\le t\\ y \in \mathbb{T}^3}}\Big(\int_{\mathbb{R}^3}|\mathbf{\tilde{h}}(s,y,u)|\mathrm{d}u\Big)\ll 1,
	\end{align*}
	for $\mathcal{E}(\mathbf{F}_0)+\|\mathbf{\tilde{h}}_{0}\|_{L^1_xL^{\infty}_v}\le \varepsilon_1$ with $\varepsilon_1$ sufficiently small. 
\end{remark}


\section{Exponential decay of $L^\infty$}
In this section, the aim is to establish the exponential decay in time of $\|\mathbf{h}(t)\|_{L^\infty}$.

\subsection{Estimate  on the hydrodynamic part }
Denote $\bar{\mathbf{f}}(t,x,v):=e^{\lambda t}\mathbf{f}(t,x,v)$, where $\lambda >0$  is a suitably small  constant determined later. Then it follows from  \eqref{Eq1.18} that
\begin{equation}\label{Eq4.2}
	\left\{
	\begin{aligned}
		&\partial_t\bar{\mathbf{f}}+v\cdot\nabla_x\bar{\mathbf{f}}-e^{-\lambda t}\mathbf{q}(\nabla_x\bar{\phi}\cdot\nabla_v)\bar{\mathbf{f}}+\mathbf{Ly}
		=e^{-\lambda t}\mathbf{\Gamma}(\bar{\mathbf{f}},\bar{\mathbf{f}})-e^{-\lambda t}(\nabla_x\bar{\phi}\cdot\frac{v}{2})\mathbf{q}\bar{\mathbf{f}}   \\
		&\quad-\nabla_x\bar{\phi}\cdot v\sqrt{\mu}\mathbf{q_1}+\lambda \bar{\mathbf{f}},\\
		&	-\Delta_x\bar{\phi}=\int\sqrt{\mu}\left(\bar{f}_+-\bar{f}_- \right) dv,   \quad \int_{\mathbb{T}^3}\bar{\phi} ~\mathrm{d}x=0,
	\end{aligned}
	\right.
\end{equation}
where
$\bar{\mathbf{f}}=\begin{bmatrix} 
	\bar{f}_+\\ \bar{f}_- \end{bmatrix}$, $\mathbf{q}=\begin{bmatrix} 
	1&0\\0&-1 \end{bmatrix}$, $\mathbf{q_1}=\begin{bmatrix} 
	1\\ -1 \end{bmatrix}$.

Define the hydrodynamic part $\mathbf P \bar{\mathbf{f}}$ of $\bar{\mathbf{f}}$ as
\begin{equation}\label{Eq4.3}
	\mathbf P \bar{\mathbf{f}}:=\left\lbrace\bar{a}_+(t, x)\begin{bmatrix} 
		\sqrt{\mu}\\ 0 \end{bmatrix}+
	\bar{a}_-(t, x)\begin{bmatrix} 
		0\\ \sqrt{\mu}\end{bmatrix}+
	\bar{\mathbf{b}}(t, x)\cdot\frac{v}{\sqrt{2}}
	\begin{bmatrix} 
		\sqrt{\mu}\\ \sqrt{\mu} \end{bmatrix}+
	\bar{c}(t, x)\frac{|v|^2-3}{2\sqrt{2}}\begin{bmatrix} 
		\sqrt{\mu}\\ \sqrt{\mu} \end{bmatrix}
	\right\rbrace,
\end{equation}
where $\bar{\mathbf{b}}=[\bar{b}_1,\bar{b}_2,\bar{b}_3]^{\mathrm{T}}$.

\begin{lemma}\label{Lem4.1}
	Under the condition \eqref{Eq2.21},  one has that, for $1\leq s\leq t$,
	\begin{equation}\label{Eq4.5}
		\int_s^t\|\mathbf{P} \bar{\mathbf{f}}\|_{L^2}^2\mathrm{d}\tau+\int_s^t\|\nabla_x\bar{\phi}\|_{L^2}^2\mathrm{d}\tau
		\leq G(s)-G(t)+C_7\int_s^t\|\{\mathbf{I}-\mathbf{P}\}\bar{\mathbf{f}}\|_\nu^2\mathrm{d}\tau,
	\end{equation}
	where $|G(s)|\leq C\|\bar{\mathbf{f}}(s)\|_{L^2}^2$.
\end{lemma}
\begin{proof}
	From the Green's identity, a mild solution $\bar{\mathbf{f}}$ of \eqref{Eq4.2} satisfies
	\begin{align}\label{Eq4.6}
		&\left\langle \bar{\mathbf{f}}(t),\mathbf{\Psi}(t) \right\rangle  
		-\left\langle \bar{\mathbf{f}}(s),\mathbf{\Psi}(s) \right\rangle
		-\int_{s}^t\underbrace{\langle \bar{\mathbf{f}},\partial_t\mathbf{\Psi} \rangle}_{\uppercase\expandafter{\romannumeral1}}
		+\underbrace{\langle \mathbf{P}\bar{\mathbf{f}},v\cdot\nabla_x\mathbf{\Psi} \rangle}_{\uppercase\expandafter{\romannumeral2}} \mathrm{d}\tau
		\nonumber  \\
		&-\int_{s}^t\underbrace{\langle\{\mathbf{I}-\mathbf{P}\}\bar{\mathbf{f}},v\cdot\nabla_x\mathbf{\Psi} \rangle -\langle \mathbf{\Psi},\mathbf{L}\{\mathbf{I}-\mathbf{P}\}\bar{\mathbf{f}} \rangle
		+e^{-\lambda\tau}\langle \mathbf{\Psi},\mathbf{\Gamma}(\bar{\mathbf{f}},\bar{\mathbf{f}}) \rangle}_{\uppercase\expandafter{\romannumeral3}} \mathrm{d}\tau\nonumber\\
		=&-\int_{s}^t\underbrace{e^{-\lambda\tau}\langle \mathbf{q}\sqrt{\mu}\bar{\mathbf{f}},\nabla_x\bar{\phi}\cdot\nabla_v(\frac{\mathbf{\Psi}}{\sqrt{\mu}} )  \rangle}_{\uppercase\expandafter{\romannumeral4}}
		+\underbrace{\langle \mathbf{q_1}\nabla_x\bar{\phi}\cdot v\sqrt{\mu},\mathbf{\Psi} \rangle}_{\uppercase\expandafter{\romannumeral5}}
		-\underbrace{\lambda\langle \bar{\mathbf{f}},\mathbf{\Psi} \rangle}_{\uppercase\expandafter{\romannumeral6}} \mathrm{d}\tau .
	\end{align}

Recall the definition of $\bar{\mathbf{f}}$, it is clear from \eqref{Eq2.4} $-$ \eqref{Eq2.7} that
\begin{align}
		&\iint_{\mathbb{T}^3\times\mathbb{R}^3}\bar{f}_{\pm}(t,x,v)\sqrt{\mu(v)}\mathrm{d}x\mathrm{d}v=0, \label{Ynew1}\\
		&\iint_{\mathbb{T}^3\times\mathbb{R}^3}\left(\bar{f}_{+}+\bar{f}_{-}\right)(t,x,v)v\sqrt{\mu(v)}\mathrm{d}x\mathrm{d}v=\mathbf{0}, \label{Ynew2}\\
		&\iint_{\mathbb{T}^3\times\mathbb{R}^3}\left(\bar{f}_{+}+\bar{f}_{-}\right)(t,x,v)|v|^2\sqrt{\mu(v)}\mathrm{d}x\mathrm{d}v+\int_{\mathbb{T}^3} e^{-\lambda t}|\nabla_{x}\bar{\phi}(t,x)|^2\mathrm{d} x=0,		\label{Ynew3}
\end{align}
which implies that for $t\ge 0$
\begin{equation}\label{Ynew5}
	\left\{
	\begin{aligned}
		&\int_{\mathbb{T}^3}\bar{a}_{\pm}(t,x)\mathrm{d}x=\int_{\mathbb{T}^3}\bar{b}_{i}(t,x)\mathrm{d}x=0,\  i=1,2,3,\\
		&	\int_{\mathbb{T}^3}\left(3\sqrt{2}\bar{c}(t,x) + e^{-\lambda t}|\nabla_{x}\bar{\phi}(t,x)|^2\right)\mathrm{d} x=0.
	\end{aligned}
	\right.
\end{equation}

However, due to $\int_{\mathbb{T}^3}\bar{c}(t,x)\mathrm{d} x \not\equiv 0$, another question arises:  the Poisson equation $-\Delta\varphi_{c}(t,x)=\bar{c}(t,x)$, with $ \int_{\mathbb{T}^3}\varphi_{c}\mathrm{d} x=0$  is ill-posed, so we cannot use it to estimate $\bar{c}$ as  in \cite{E-G-K-M13CMP, Cao19ARMA}. To solve this problem, we define a new function 
\begin{equation}\label{new5c}
	\tilde{c}(t,x):=\bar{c}(t,x)+\frac{\sqrt{2}}{6}e^{-\lambda t}|\nabla_{x}\bar{\phi}(t,x)|^2.
\end{equation}
Thus it follows from \eqref{Ynew5} that
 \begin{equation}\label{Ynew6}
	\int_{\mathbb{T}^3}\tilde{c}(t,x)\mathrm{d} x=0.
\end{equation}
 
 We define $\varphi_{a}^{\pm}(t,x)$, $\varphi_{b_{i}}(t,x)$ and $\varphi_{c}(t,x)$ as the solutions of the following Poisson equations, respectively. 
\begin{align}
	&-\Delta\varphi_{a}^{\pm}(t,x)=\bar{a}_{\pm}(t,x),\quad \int_{\mathbb{T}^3}\varphi_{\bar{a}_{\pm}}(t,x)\mathrm{d}x=0,\label{Ynew7}\\ 
	&-\Delta\varphi_{b_{i}}(t,x)=\bar{b}_i(t,x),\quad  \int_{\mathbb{T}^3}\varphi_{b_{i}}(t,x)\mathrm{d}x=0,\label{Ynew8}\\
	&-\Delta\varphi_{c}(t,x)=\tilde{c}(t,x),\quad \int_{\mathbb{T}^3}\varphi_{c}(t,x)\mathrm{d}x=0.\label{Ynew9}
\end{align}   
It follows from the standard elliptic estimate that
\begin{align}\label{Ynew9-1}
	\|\varphi_{a}^{\pm}\|_{H^2}\leq C\|\bar{a}_{\pm}\|_{L^2},\quad \|\varphi_{b_{i}}\|_{H^2}\leq C\|\bar{b}_{i}\|_{L^2},\quad \|\varphi_{c}\|_{H^2}\leq C\|\tilde{c}\|_{L^2}.
\end{align}
For the sake of later use, we define $\bar{\mathbf{a}}=\begin{bmatrix} 
	\bar{a}_+\\ \bar{a}_- \end{bmatrix}$.

Now we divide the proof into four steps.\\
\noindent{\it Step 1. Estimate on $\nabla_{x} \partial_{t} \varphi_{a}^{\pm}$}. Choosing the test function $\displaystyle\mathbf{\Psi}=\begin{bmatrix} 
\chi_{+}(x)\sqrt{\mu}\\ 0 \end{bmatrix}$, with $\chi_{+}(x)$ depending only on $x$ and substituting it into \eqref{Eq4.6} (with time integration over $[t,t+\varepsilon]$), then we obtain 
\begin{align*}
	&\int_{\mathbb{T}^3}[\bar{a}_{+}(t+\v)-\bar{a}_{+}(t)]\chi_{+}(x)\mathrm{d} x
	=\f{1}{\sqrt{2}}\int_{t}^{t+\v}\int_{\mathbb{T}^3}\bar{\mathbf{b}}\cdot\nabla_x\chi_{+}\mathrm{d} x \mathrm{d} \tau
	+\lambda\int_t^{t+\v}\int_{\mathbb{T}^3}\bar{a}_{+}\chi_{+} \mathrm{d} x\mathrm{d} \tau,
\end{align*}
where we have used that $\uppercase\expandafter{\romannumeral1}$, $\uppercase\expandafter{\romannumeral3}$, $\uppercase\expandafter{\romannumeral4}$ and $\uppercase\expandafter{\romannumeral5}$ vanish.
Taking the difference quotient, we have that for almost $t>0$,
\begin{align}\label{Ynew10}
	&\int_{\mathbb{T}^3}\partial_{t}\bar{a}_{+}(t)\chi_{+}(x)\mathrm{d} x
	=\f{1}{\sqrt{2}}\int_{\mathbb{T}^3}\bar{\mathbf{b}}\cdot\nabla_x\chi_{+}\mathrm{d} x 
	+\lambda\int_{\mathbb{T}^3}\bar{a}_{+}\chi_{+} \mathrm{d} x.
\end{align}
Choosing the test function $\mathbf{\Psi}=\begin{bmatrix} 
	0 \\ \chi_{-}(x)\sqrt{\mu} \end{bmatrix}$ in \eqref{Eq4.6}, with $\chi_{-}(x)$ depending only on $x$. Similarly, it holds
	\begin{align}\label{Ynew11}
		&\int_{\mathbb{T}^3}\partial_{t}\bar{a}_{-}(t)\chi_{-}(x)\mathrm{d} x
		=\f{1}{\sqrt{2}}\int_{\mathbb{T}^3}\bar{\mathbf{b}}\cdot\nabla_x\chi_{-}\mathrm{d} x 
		+\lambda\int_{\mathbb{T}^3}\bar{a}_{-}\chi_{-} \mathrm{d} x.
	\end{align}
It follows from \eqref{Ynew10} and \eqref{Ynew11} that	
	\begin{align*}
		\left|\int_{\mathbb{T}^3} \partial_t\bar{a}_{\pm}(t)\chi_{\pm} \mathrm{d}x\right|\lesssim\{\|\bar{\mathbf{b}}(t)\|_{L^2}+\lambda\|\bar{a}_{\pm}(t)\|_{L^2}\}\cdot \|\chi_{\pm}\|_{H^1}.
	\end{align*}
Hence, for almost $t\geq1$, one has
\begin{align*}
	\|\partial_t\bar{a}_{\pm}(t)\|_{(H^1)^*}\lesssim\|\bar{\mathbf{b}}(t)\|_{L^2}+\lambda\|\bar{a}_{\pm}(t)\|_{L^2},
\end{align*}
where $(H^1)^*=\big(H^1(\mathbb{T}^3)\big)^{*}$ is the dual space of $H^1(\mathbb{T}^3)$ with respect to the dual pair $\langle A, B\rangle=\int_{\mathbb{T}^3} A(x) B(x) d x$, for $A \in H^{1}$ and $B \in\left(H^{1}\right)^{*}$. 

Since $\partial_t\bar{a}_{\pm}(t) \in\left(H^{1}\right)^{*}$, by the standard elliptic theory, we can solve $-\Delta \Phi_{a}^{\pm}=\partial_{t} \bar{a}_{\pm}(t)$, $\int_{\mathbb{T}^3}\Phi_{a}^{\pm}dx=0$. Noting  $\Phi_{a}^{\pm}=-\Delta^{-1} \partial_{t} \bar{a}_{\pm}=\partial_{t} \varphi_{a}^{\pm}$ with $\varphi_{a}^{\pm}$ defined in \eqref{Ynew7}, thus we have
	\begin{align}\label{Eq4.15}
		\left\|\nabla_{x} \partial_{t} \varphi_{a}^{\pm}\right\|_{L^2}&\cong \left\|\Delta^{-1} \partial_{t} \bar{a}_{\pm}(t)\right\|_{H^{1}}=\left\|\Phi_{a}^{\pm}\right\|_{H^{1}} \lesssim\left\|\partial_{t} \bar{a}_{\pm}(t)\right\|_{\left(H^{1}\right)^{*}}\nonumber\\ &\le C\Big(\|\bar{\mathbf{b}}(t)\|_{L^2}+\lambda\|\bar{a}_{\pm}(t)\|_{L^2}\Big).
	\end{align}

\noindent \textit {Step 2. Estimate on $\nabla_{x} \partial_{t} \varphi_{c}$ }. Choosing the test function $\displaystyle\mathbf{\Psi}=\mathbf{\Psi}_{c}^t:=
\begin{bmatrix} 
	\frac{|v|^2-3}{2\sqrt{2}}\sqrt{\mu} \chi_{c}(x)\\ \frac{|v|^2-3}{2\sqrt{2}}\sqrt{\mu} \chi_{c}(x) \end{bmatrix}$ in \eqref{Eq4.6}, one has
 \begin{align*}
 	\f{3}{2}\int_{\mathbb{T}^3} [\bar{c}(t+\varepsilon)-\bar{c}(t)] \chi_{c}(x)\mathrm{d}x&=
 	\int_{t}^{t+\varepsilon} \int_{\mathbb{T}^3}\bar{\mathbf{b}}\cdot\nabla_x\chi_{c}\mathrm{d} x \mathrm{d} \tau 
 	+\f{3\lambda}{2}\int_{t}^{t+\varepsilon} \int_{\mathbb{T}^3}\bar{c}\chi_{c}\mathrm{d} x \mathrm{d} \tau
 	\\
 	&\quad+\int_{t}^{t+\varepsilon} \iint_{\mathbb{T}^3 \times \mathbb{R}^3} \{\mathbf{I}-\mathbf{P}\}\bar{\mathbf{f}}\cdot (v\cdot\nabla_x)\mathbf{\Psi}_{c}^t \mathrm{d} v\mathrm{d} x \mathrm{d} \tau \\
 	&\quad+\int_{t}^{t+\varepsilon}e^{-\lambda\tau}\iint_{\mathbb{T}^3 \times \mathbb{R}^3}
 	\mathbf{q}\sqrt{\mu}\bar{\mathbf{f}}\cdot (\nabla_x\bar{\phi}\cdot\nabla_v)\left(\frac{\mathbf{\Psi}_{c}^t}{\sqrt{\mu}} \right)\mathrm{d} v\mathrm{d} x \mathrm{d} \tau,
 \end{align*}
where we have used that $\mathbf{L}\{\mathbf{I}-\mathbf{P}\}\bar{\mathbf{f}}$, $\mathbf{\Gamma}(\bar{\mathbf{f}},\bar{\mathbf{f}})$ and $\nabla_x\bar{\phi}\cdot v\sqrt{\mu}\mathbf{q_1}$, 
integrated against $\mathbf{\Psi}_{c}^t$ are zero.
 Taking the difference quotient, we obtain for almost $t\geq 1$,
 \begin{align}\label{Ynew12}
	\f{3}{2}\int_{\mathbb{T}^3} \partial_{t}\bar{c}(t) \chi_{c}(x)\mathrm{d}x&=
	 \int_{\mathbb{T}^3}\bar{\mathbf{b}}\cdot\nabla_x\chi_{c}\mathrm{d} x 
	 +\f{3\lambda}{2}\int_{\mathbb{T}^3}\bar{c}(t)\chi_{c}(x)\mathrm{d} x \nonumber\\
	&\quad+\iint_{\mathbb{T}^3 \times \mathbb{R}^3} \{\mathbf{I}-\mathbf{P}\}\bar{\mathbf{f}}\cdot (v\cdot\nabla_x)\mathbf{\Psi}_{c}^t \mathrm{d} v\mathrm{d} x  \nonumber\\
	&\quad+e^{-\lambda t}\iint_{\mathbb{T}^3 \times \mathbb{R}^3}
	\mathbf{q}\sqrt{\mu}\bar{\mathbf{f}}\cdot (\nabla_x\bar{\phi}\cdot\nabla_v)\left(\frac{\mathbf{\Psi}_{c}^t}{\sqrt{\mu}}\right)\mathrm{d} v\mathrm{d} x.
\end{align}

For fixed $t\geq 1$, we choose $\chi_{c}(x)=\Phi_{c}(t,x)$ with $-\Delta \Phi_{c}=\partial_{t} \tilde{c}(t)$, $\int_{\mathbb{T}^3}\Phi_{c}dx=0$. It is clear that $\Phi_{c}=-\Delta^{-1} \partial_{t} \tilde{c}(t)=\partial_{t} \varphi_{c}$, where $\varphi_{c}$ is defined in \eqref{Ynew9}.
Using the fact that
\begin{align}\label{new5.1}
\partial_{t}\bar{c}=\partial_{t}\tilde{c}+\frac{\sqrt{2}}{6}\lambda e^{-\lambda t}|\nabla_{x}\bar{\phi}|^2-\frac{\sqrt{2}}{3}e^{-\lambda t}\nabla_{x}\bar{\phi}\cdot\nabla_{x}\partial_{t}\bar{\phi},
\end{align}
we have
\begin{align}\label{new5.2}
\int_{\mathbb{T}^3}\partial_t\bar{c}\partial_t\varphi_{c}\mathrm{d}x=&\int_{\mathbb{T}^3}\partial_t\tilde{c}\partial_t\varphi_{c}\mathrm{d}x +\frac{\sqrt{2}}{6}\lambda \int_{\mathbb{T}^3} e^{-\lambda t}|\nabla_{x}\bar{\phi}|^2\partial_t\varphi_{c}\mathrm{d}x   \nonumber\\
&-\frac{\sqrt{2}}{3}\int_{\mathbb{T}^3} 
e^{-\lambda t}(\nabla_{x}\bar{\phi}\cdot\nabla_{x}\partial_{t}\bar{\phi})\partial_t\varphi_{c}\mathrm{d}x \nonumber\\
=&\|\nabla\Delta^{-1}\partial_t \tilde{c}(t)\|_{L^2}^2+\frac{\sqrt{2}}{6}\lambda \int_{\mathbb{T}^3} e^{-\lambda t}|\nabla_{x}\bar{\phi}|^2\partial_t\varphi_{c}\mathrm{d}x  \nonumber\\
&-\frac{\sqrt{2}}{3}\int_{\mathbb{T}^3} 
e^{-\lambda t}\nabla_{x}\bar{\phi}\cdot\nabla_{x}\partial_{t}\bar{\phi}\partial_t\varphi_{c}\mathrm{d}x.
\end{align}

For the second term on RHS of \eqref{new5.2}, it follows from \eqref{Eq2.21}, \eqref{Eq1.18} and \eqref{Eq4.2} that
\begin{align}\label{new5.3}
\|e^{-\lambda t}\nabla_{x}\bar{\phi}(t)\|_{L^\infty}\leq \delta(1+t)^{-2},
\end{align}
which yields that
\begin{align}\label{new5.4}
	\frac{\sqrt{2}}{6}\lambda \Big|\int_{\mathbb{T}^3} e^{-\lambda t}|\nabla_{x}\bar{\phi}|^2\partial_t\varphi_{c}\mathrm{d}x \Big|&\lesssim \delta \|\nabla_{x}\bar{\phi}(t)\|_{L^2}\|\partial_t\varphi_{c}(t)\|_{L^2}\nonumber\\
	&\lesssim \delta\left(\|\nabla_{x}\bar{\phi}(t)\|_{L^2}^{2}+\|\nabla\Delta^{-1}\partial_t \tilde{c}(t)\|_{L^2}^2\right),
\end{align}
where we have used Poincar\'{e}'s inequality 
\begin{equation}\label{new5.4-1}
\|\partial_t\varphi_{c}(t)\|_{L^2}\lesssim \|\nabla_{x}\partial_t\varphi_{c}(t)\|_{L^2}= \|\nabla\Delta^{-1}\partial_t \tilde{c}(t)\|_{L^2}.
\end{equation}

For the last term on RHS of \eqref{new5.2}, it follows from \eqref{new5.3} that 
\begin{align}\label{new5.5}
\frac{\sqrt{2}}{3}\Big|\int_{\mathbb{T}^3} 
e^{-\lambda t}\nabla_{x}\bar{\phi}\cdot\nabla_{x}\partial_{t}\bar{\phi}\partial_t\varphi_{c}\mathrm{d}x
\Big|
\lesssim&  \delta  \|\nabla_{x}\partial_{t}\bar{\phi}(t)\|_{L^2}\|\partial_t\varphi_{c}(t)\|_{L^2} \nonumber\\
\lesssim&   \delta \left(\|\nabla\Delta^{-1}\partial_t \bar{\mathbf{a}}(t)\|_{L^2}^2+\|\nabla\Delta^{-1}\partial_t \tilde{c}(t)\|_{L^2}^2\right),
\end{align}
where we used $\|\nabla_{x}\partial_{t}\bar{\phi}(t)\|_{L^2}^2\lesssim \left(\|\nabla\Delta^{-1}\partial_t \bar{a}_{+}(t)\|_{L^2}^2+\|\nabla\Delta^{-1}\partial_t \bar{a}_{-}(t)\|_{L^2}^2\right)=: \|\nabla\Delta^{-1}\partial_t \bar{\mathbf{a}}(t)\|_{L^2}^2$.

Substituting \eqref{new5.5} and \eqref{new5.4}  into \eqref{new5.2}, one obtains that
\begin{align}\label{new5.6}
\frac{3}{2}\int_{\mathbb{T}^3}\partial_t\bar{c}\partial_t\varphi_{c}\mathrm{d}x\geq \Big(\f{3}{2}-2\delta\Big)\|\nabla\Delta^{-1}\partial_t \tilde{c}(t)\|_{L^2}^2 
-\delta\left(\|\nabla\Delta^{-1}\partial_t \bar{\mathbf{a}}(t)\|_{L^2}^{2}+\|\nabla_{x}\bar{\phi}(t)\|_{L^2}^2\right).
\end{align}

Next, we  deal with the terms on RHS of \eqref{Ynew12} with $\chi_{c}=\Phi_{c}=\partial_{t} \varphi_{c}$.

For the first term on RHS of \eqref{Ynew12}, it is clear that
\begin{align}\label{Eq4.25}
	\int_{\mathbb{T}^3}\bar{\mathbf{b}}\cdot\nabla_x\Phi_{c}\mathrm{d} x 
	\leq\frac{m}{2} \|\nabla\Delta^{-1}\partial_t \tilde{c}(\tau)\|_{L^2}^2
	+\frac{1}{2m}\|\mathbf{\bar{b}}(\tau)\|_{L^2}^2,
\end{align}
for some  small constant $m>0$.	

For the second term on RHS of \eqref{Ynew12}, a direct calculation shows that
\begin{align}\label{Eq4.26}
\Big|\iint_{\mathbb{T}^3 \times \mathbb{R}^3} \{\mathbf{I}-\mathbf{P}\}\bar{\mathbf{f}}\cdot (v\cdot\nabla_x)\mathbf{\Psi}_{c}^t \mathrm{d} v\mathrm{d} x \Big|
	\leq \frac{m}{2} \|\nabla\Delta^{-1}\partial_t \tilde{c}(\tau)\|_2^2
	+\f{C}{2m}\|\{\mathbf{I}-\mathbf{P}\}\bar{\mathbf{f}}\|_\nu^2.
\end{align}

For the third term on RHS of \eqref{Ynew12}, according to Remark \ref{Rem4.4}, there is a  small constant $\kappa>0$ such that 
\begin{equation}\label{kappa}
\sup_{\substack{t \geq 1\\ y \in \mathbb{T}^3}}\left\{\int_{\mathbb{R}_{v}^3}|\mathbf{f}(t,x,v)|\mathrm{d}v\right\}\leq \kappa.
\end{equation}
which, together with \eqref{new5.4-1}, yields that
\begin{align}\label{Eq4.27}
	&\Big|e^{-\lambda t}\iint_{\mathbb{T}^3 \times \mathbb{R}^3}
	\mathbf{q}\sqrt{\mu}\bar{\mathbf{f}}\cdot (\nabla_x\bar{\phi}\cdot\nabla_v)(\frac{\mathbf{\Psi}_{c}^t}{\sqrt{\mu}} )\mathrm{d} v\mathrm{d} x \Big| \nonumber\\
	\lesssim & \sup\limits_{x\in\mathbb{T}^3} \Big\{\int_{\mathbb{R}_{v}^3}|\mathbf{f}(t,x,v)|\mathrm{d}v\Big\}\int_{\mathbb{T}^3}|\nabla_x\bar{\phi}||\nabla\Delta^{-1}\partial_t \tilde{c}| \mathrm{d}x
	\nonumber \\
	\lesssim &\kappa\left(\|\nabla\Delta^{-1}\partial_t \tilde{c}\|_{L^2}^2 +\|\nabla_x \bar{\phi}\|_{L^2}^2\right).
\end{align}

For the last term on RHS of \eqref{Ynew12}, using \eqref{new5.4-1}, one has
\begin{align}\label{Eq4.30}
\f{3\lambda}{2}	 \left|\int_{\mathbb{T}^3}\bar{c}(t)\partial_{t} \varphi_{c}\mathrm{d} x\right|
\lesssim  \lambda \left(\|\nabla\Delta^{-1}\partial_t \tilde{c}\|_{L^2}^2+ \|\bar{c}\|_{L^2}^2\right).
\end{align}
It follows from \eqref{new5c} and  \eqref{new5.3} that
\begin{align}\label{new5.7}
 \|\bar{c}(t)\|_{L^2}^2\lesssim  \|\tilde{c}(t)\|_{L^2}^2+\delta^2\|\nabla_x \bar{\phi}(t)\|_{L^2}^2,
\end{align}
which, together with \eqref{Eq4.30}, yields that 
\begin{align}\label{new5.8}
\frac{3\lambda}{2}\left|\int_{\mathbb{T}^3}\bar{c}(t)\partial_{t} \varphi_{c}\mathrm{d} x\right|
	\lesssim \lambda \left(\|\nabla\Delta^{-1}\partial_t \tilde{c}\|_{L^2}^2 +\|\tilde{c}(t)\|_{L^2}^2+\delta^2\|\nabla_x \bar{\phi}(t)\|_{L^2}^2\right).
\end{align}

Taking $m=\f{1}{2}$, and let $\delta$, $\kappa$ and $\lambda$ suitably small such that $2\delta+\kappa+\lambda <\f{1}{2}$, then substituting \eqref{new5.6}, \eqref{Eq4.25} $-$ \eqref{Eq4.27} and \eqref{new5.8} into \eqref{Ynew12}, we obtain that for $t\geq 1$,
\begin{align}\label{Eq4.31}
	\|\nabla_{x}\partial_t\varphi_{c}(t)\|_{L^2}^2=\|\nabla\Delta^{-1}\partial_t \tilde{c}(t)\|_{L^2}^2
	\leq& C\left(\|\mathbf{\bar{b}}(t)\|_{L^2}^2  +\|\{\mathbf{I}-\mathbf{P}\}\bar{\mathbf{f}}(t)\|_\nu^2\right) 
	+\lambda \|\tilde{c}(t)\|_{L^2}^2\nonumber\\
	&+\delta \|\bar{\mathbf{a}}(t)\|_{L^2}^2+(\kappa+\delta)\|\nabla_x \bar{\phi}(t)\|_{L^2}^2.
\end{align}

\noindent \textit{Step 3. Estimate on $\nabla_{x} \partial_{t} \varphi_{b_{i}}$ }. Choosing the test function $\displaystyle\mathbf{\Psi}=\mathbf{\Psi}_{b_{i}}^t:=\begin{bmatrix} 
	\frac{v_i}{\sqrt{2}}\sqrt{\mu}\chi_{b_{i}}(x)\\ \frac{v_i}{\sqrt{2}}\sqrt{\mu}\chi_{b_{i}}(x) \end{bmatrix}$  with $\chi_{b_{i}}(x)$ depending only on $x$, then we obtain from  \eqref{Eq4.6} that
 \begin{align*}
	\int_{\mathbb{T}^3} [\bar{b}_{i}(t+\varepsilon)-\bar{b}_{i}(t)] \chi_{b_{i}}(x)\mathrm{d}x&=
	\int_{t}^{t+\varepsilon} \int_{\mathbb{T}^3}\left(\frac{1}{\sqrt{2}} \bar{a}_+\partial_i\chi_{b_{i}}
	+\frac{1}{\sqrt{2}} \bar{a}_-\partial_i\chi_{b_{i}}+
	\bar{c}\partial_i\chi_{b_{i}}+\lambda\bar{b}_{i}\chi_{b_{i}}\right) \mathrm{d} x \mathrm{d} \tau \\
	&\quad+\int_{t}^{t+\varepsilon} \iint_{\mathbb{T}^3 \times \mathbb{R}^3} \{\mathbf{I}-\mathbf{P}\}\bar{\mathbf{f}}\cdot (v\cdot\nabla_x)\mathbf{\Psi}_{b_{i}}^t \mathrm{d} v\mathrm{d} x \mathrm{d} \tau \\
	&\quad+\int_{t}^{t+\varepsilon}e^{-\lambda\tau}\iint_{\mathbb{T}^3 \times \mathbb{R}^3}
	\mathbf{q}\sqrt{\mu}\bar{\mathbf{f}}\cdot (\nabla_x\bar{\phi}\cdot\nabla_v)\left(\frac{\mathbf{\Psi}_{b_{i}}^t}{\sqrt{\mu}}\right)\mathrm{d} v\mathrm{d} x \mathrm{d} \tau.
\end{align*}
Taking the difference quotient, we obtain 
 \begin{align} \label{Ynew13}
	\int_{\mathbb{T}^3} \partial_{t}\bar{b}_{i}(t) \chi_{b_{i}}(x)\mathrm{d}x&=
	 \int_{\mathbb{T}^3}\left(\frac{1}{\sqrt{2}} \bar{a}_+\partial_i\chi_{b_{i}}
	+\frac{1}{\sqrt{2}} \bar{a}_-\partial_i\chi_{b_{i}}+
	\bar{c}\partial_i\chi_{b_{i}}+\lambda\bar{b}_{i}\chi_{b_{i}}\right) \mathrm{d} x   \nonumber\\
	&\quad+\iint_{\mathbb{T}^3 \times \mathbb{R}^3} \{\mathbf{I}-\mathbf{P}\}\bar{\mathbf{f}}\cdot (v\cdot\nabla_x)\mathbf{\Psi}_{b_{i}}^t \mathrm{d} v\mathrm{d} x   \nonumber\\
	&\quad+e^{-\lambda t}\iint_{\mathbb{T}^3 \times \mathbb{R}^3}
	\mathbf{q}\sqrt{\mu}\bar{\mathbf{f}}\cdot (\nabla_x\bar{\phi}\cdot\nabla_v)\left(\frac{\mathbf{\Psi}_{b_{i}}^t}{\sqrt{\mu}}\right)\mathrm{d} v\mathrm{d} x,
\ 	\text{for} \ t \geq 1.
\end{align}

According to Remark \ref{Rem4.4}, it follows that
\begin{align}
\Big|e^{-\lambda t}\iint_{\mathbb{T}^3 \times \mathbb{R}^3}
\mathbf{q}\sqrt{\mu}\bar{\mathbf{f}}\cdot (\nabla_x\bar{\phi}\cdot\nabla_v)\left(\frac{\mathbf{\Psi}_{b_{i}}^t}{\sqrt{\mu}}\right)\mathrm{d} v\mathrm{d} x \Big|\lesssim~\kappa \|\nabla_x\bar{\phi}\|_{L^2}\cdot\|\chi_{b_{i}}\|_{H^1},
\end{align}
which, with \eqref{Ynew13},  yields that
\begin{align*}
	\|\partial_t\bar{b}_{i}(t)\|_{(H^1)^*}\lesssim& \|\bar{\mathbf{a}}(t)\|_{L^2}+\|\bar{c}(t)\|_{L^2}+\lambda\|\bar{b}_{i}(t)\|_{L^2}+\kappa\|\nabla_x\bar{\phi}\|_{L^2}+\|\{\mathbf{I}-\mathbf{P}\}\bar{\mathbf{f}}(t)\|_\nu \\
	\lesssim& \|\bar{\mathbf{a}}(t)\|_{L^2}+\|\tilde{c}(t)\|_{L^2}+\lambda\|\bar{b}_{i}(t)\|_{L^2}+(\kappa+\delta)\|\nabla_x\bar{\phi}\|_{L^2}+\|\{\mathbf{I}-\mathbf{P}\}\bar{\mathbf{f}}(t)\|_\nu.
\end{align*}

For fixed $t\geq 1$, we choose $\chi_{b_{i}}=\Phi_{b_{i}}$ with $-\Delta \Phi_{b_{i}}=\partial_{t} \bar{b}_{i}(t)$, $\int_{\mathbb{T}^3}\Phi_{b_{i}}dx=0$. It is clear that $\Phi_{b_{i}}=-\Delta^{-1} \partial_{t} \bar{b}_{i}=\partial_{t}\varphi_{b_{i}}$, where $\varphi_{b_{i}}$ is defined in \eqref{Ynew8}. By similar arguments as in \textit{Step 1}, we have
\begin{align} \label{Eq4.23}
	\left\|\nabla_{x} \partial_{t} \varphi_{b_{i}}(t)\right\|_{L^2}&\cong \left\|\Delta^{-1} \partial_{t} \bar{b}_i(t)\right\|_{H^{1}}=\left\|\Phi_{b_{i}}(t)\right\|_{H^{1}} \lesssim\left\|\partial_{t} \bar{b}_i(t)\right\|_{\left(H^{1}\right)^{*}}\nonumber\\ &\lesssim \|(\bar{\mathbf{a}},\tilde{c})(t)\|_{L^2}+\lambda\|\bar{b}_{i}(t)\|_{L^2}+(\kappa+\delta)\|\nabla_x\bar{\phi}(t)\|_{L^2}+\|\{\mathbf{I}-\mathbf{P}\}\bar{\mathbf{f}}(t)\|_\nu.
\end{align}

	\noindent{\it Step 4. Estimate on $\bar{\mathbf{a}}$, $\mathbf{\bar{b}}$, $\bar{c}$ }. 

	\noindent{\it Step 4.1. Estimate on $\tilde{c}$.} 
	Motivated by \cite{E-G-K-M13CMP}, we choose the test function
	$$\displaystyle\mathbf{\Psi}=\begin{bmatrix} 
		(|v|^2-\beta_{c})\sqrt{\mu}v\cdot\nabla_x\varphi_{c}\\ (|v|^2-\beta_{c})\sqrt{\mu}v\cdot\nabla_x\varphi_{c} \end{bmatrix},$$
	where $\varphi_{c}$ is the one defined in \eqref{Ynew9}, and $\beta_{c}=5$ so that $\int(|v|^2-\beta_{c})v_i^2\mu(v)dv=0$.  It is easy to verify that	$\uppercase\expandafter{\romannumeral5}$ in \eqref{Eq4.6} vanishes, due to the orthogonality.
	
	For $\uppercase\expandafter{\romannumeral1}$, due to oddness in $v$  and the choice of $\b_{c}$, 
	$$\iint_{\mathbb{T}^3\times\mathbb{R}^3}\mathbf P \bar{\mathbf{f}}\cdot \begin{bmatrix} 
		(|v|^2-\beta_{c})\sqrt{\mu}v\cdot\nabla_x\partial_{t}\varphi_{c}\\ (|v|^2-\beta_{c})\sqrt{\mu}v\cdot\nabla_x\partial_{t}\varphi_{c} \end{bmatrix}\mathrm{d}v \mathrm{d}x=0.$$
	Thus using \eqref{Eq4.31}, we have
	\begin{align}\label{Eq4.32}
	\uppercase\expandafter{\romannumeral1}
		\lesssim &~m_{c}\int_s^t\|\nabla\Delta^{-1}\partial_t \tilde{c}(\tau)\|_{L^2}^2\mathrm{d}\tau
		+\f{1}{m_c}\int_s^t\|\{\mathbf{I}-\mathbf{P}\}\bar{\mathbf{f}}(\tau)\|_\nu^2\mathrm{d}\tau	 \nonumber\\
	\lesssim &~m_{c} \int_s^t\left( \delta\|\bar{\mathbf{a}}(\tau)\|_{L^2}^2 +\lambda\|\tilde{c}(\tau)\|_{L^2}^2+(\kappa+\delta)\|\nabla_x \bar{\phi}(\tau)\|_{L^2}^2+\|\mathbf{\bar{b}}(\tau)\|_{L^2}^2\right)\mathrm{d}\tau  \nonumber\\
		&+\f{1}{m_c}\int_s^t\|\{\mathbf{I}-\mathbf{P}\}\bar{\mathbf{f}}(\tau)\|_\nu^2\mathrm{d}\tau.
	\end{align}

	For $\uppercase\expandafter{\romannumeral2}$, the $\bar{a}_{\pm}$, $\mathbf{\bar{b}}$ contributions vanish owing to oddness in $v$ and the choice of $\beta_{c}$, thus one has from \eqref{new5c} that
	\begin{align}\label{Eq4.33}
		-\uppercase\expandafter{\romannumeral2}
		&=-5\sqrt{2}\int_s^t\int_{\mathbb{T}^3}\bar{c}(\tau,x)\Delta\varphi_{c}\mathrm{d}x\mathrm{d}\tau   \nonumber\\
		&=-5\sqrt{2}\int_s^t\int_{\mathbb{T}^3}\tilde{c}(\tau,x)\Delta\varphi_{c}\mathrm{d}x\mathrm{d}\tau-\frac{5}{3}\int_s^t\int_{\mathbb{T}^3}e^{-\lambda t}|\nabla_{x}\bar{\phi}|^2\tilde{c}\mathrm{d}x\mathrm{d}\tau  \nonumber\\
		&\geq 5\sqrt{2}\int_s^t\|\tilde{c}\|_{L^2}^2\mathrm{d}\tau -\delta\int_s^t\|\nabla_x \bar{\phi}\|_{L^2}\|\tilde{c}\|_{L^2}\mathrm{d}\tau   \nonumber\\
		&\geq (5\sqrt{2}-\delta)\int_s^t\|\tilde{c}\|_{L^2}^2\mathrm{d}\tau -\delta\int_s^t\|\nabla_x \bar{\phi}\|_{L^2}^2\mathrm{d}\tau.
	\end{align}

For $\uppercase\expandafter{\romannumeral4}$, it is apparent from \eqref{kappa} that
	\begin{align}\label{Eq4.34}
		\uppercase\expandafter{\romannumeral4}
		\lesssim & \sup_{\substack{ t\geq 1 \\ y \in \mathbb{T}^3}} \left\{\int_{\mathbb{R}_{v}^3}|\mathbf{f}(t,x,v)|\mathrm{d}v\right\}\int_s^t\|\nabla_x\bar{\phi}\|_{L^2}\|\nabla_x\varphi_{c}\|_{L^2} \mathrm{d}\tau\nonumber\\
		\lesssim &~ \kappa\int_s^t ( \|\nabla_x \bar{\phi}\|_{L^2}^2+\|\tilde{c}\|_{L^2}^2)\mathrm{d}\tau.
	\end{align}
	
	 For $\uppercase\expandafter{\romannumeral6}$, due to oddness in $v$, one has
	\begin{align}\label{Eq4.35}
		\uppercase\expandafter{\romannumeral6}
		\lesssim \lambda \int_{s}^t \left(\|\tilde{c}\|_{L^2}^2
		+\|\{\mathbf{I}-\mathbf{P}\}\bar{\mathbf{f}}\|_\nu^2\right)\mathrm{d}\tau.
	\end{align}  
	
It remains to control $\uppercase\expandafter{\romannumeral3}$. Using \eqref{kappa} and Proposition \ref{prop4.3}, a routine computation yields  that
	\begin{align}\label{Eq4.37}
		\uppercase\expandafter{\romannumeral3}
		\lesssim&~m_c\int_{s}^t\|\tilde{c}\|_{L^2}^2\mathrm{d}\tau
		+\f{1}{m_c}\int_{s}^t\|\{\mathbf{I}-\mathbf{P}\}\bar{\mathbf{f}}\|_\nu^2 \mathrm{d}\tau  \nonumber\\
		&+\sup_{\substack{ t\geq 1\\ y \in \mathbb{T}^3}} \Big\{\int_{\mathbb{R}_{v}^3}|\mathbf{f}(t,x,v)|\mathrm{d}v\|w_\beta \mathbf{f}\|_{L^\infty}\Big\}^{\f{1}{2}}
		 \int_s^t\|\bar{\mathbf{f}}\|_{L^2}^2\mathrm{d}\tau \nonumber\\
		 \lesssim&~m_c\int_{s}^t\|\tilde{c}\|_{L^2}^2\mathrm{d}\tau
		 +\f{1}{m_c}\int_{s}^t\|\{\mathbf{I}-\mathbf{P}\}\bar{\mathbf{f}}\|_\nu^2 \mathrm{d}\tau  
		 +\kappa^{\f{1}{2}} \int_s^t\|\bar{\mathbf{f}}\|_{L^2}^2\mathrm{d}\tau.
	\end{align}

Plugging \eqref{Eq4.32} $-$ \eqref{Eq4.37}  into \eqref{Eq4.6}, and taking $\delta+\lambda+\kappa+m_c< 1$, one can get
\begin{align} \label{Eq4.38}
	\int_{s}^t\|\tilde{c}(\tau)\|_{L^2}^2\mathrm{d}\tau  
	\lesssim& \f{1}{m_c}\int_s^t\|\{\mathbf{I}-\mathbf{P}\}\bar{\mathbf{f}}(\tau)\|_\nu^2\mathrm{d}\tau
	+(\delta+\kappa)\int_s^t( \|\bar{\mathbf{a}}(\tau)\|_{L^2}^2+\|\nabla_x \bar{\phi}(\tau)\|_{L^2}^2)  \mathrm{d}\tau\nonumber\\
		&+m_c\int_{s}^t \|\mathbf{\bar{b}}(\tau)\|_{L^2}^2\mathrm{d}\tau
	+\kappa^{\f{1}{2}}\int_s^t \|\bar{\mathbf{f}}(\tau)\|_{L^2}^2\mathrm{d}\tau+G(s)-G(t),
\end{align}
where $G(s):=\iint_{\mathbb{T}^3\times\mathbb{R}^3}\bar{\mathbf{f}}(s,x,v)\cdot\mathbf{\Psi}(s,x,v)\mathrm{d}v \mathrm{d}x$, and $|G(s)|\leq C\|\bar{\mathbf{f}}(s)\|_{L^2}^2$.
	
		\noindent{\it Step 4.2. Estimate on $\bar{\mathbf{a}}$.} 
	  We choose a test function
	$$\displaystyle\mathbf{\Psi}=\mathbf{\Psi}_{a}:=\begin{bmatrix} 
		-(|v|^2-\beta_{a})\sqrt{\mu}v\cdot\nabla_x\varphi_{a}^{+}\\ -(|v|^2-\beta_{a})\sqrt{\mu}v\cdot\nabla_x\varphi_{a}^{-} \end{bmatrix},$$
	where
	$\varphi_{a}^{\pm}(t,x)$ are the ones defined in \eqref{Ynew7}, and  $\beta_{a}=10$  so that $\int(|v|^2-\beta_{a})\left(\frac{|v|^2-3}{2\sqrt{2}}\right)v_{i}^2\mu(v)dv=0$.

For $\uppercase\expandafter{\romannumeral2}$,  in view of oddness in $v$ and the choice of $\beta_{a}$, the $\mathbf{\bar{b}}$, $\bar{c}$ contributions vanish, one has
\begin{align}\label{Eq4.40}
	\uppercase\expandafter{\romannumeral2}
	=&5\int_s^t\int_{\mathbb{T}^3}\bar{a}_{+}\Delta\varphi_{a}^{+}\mathrm{d}x\mathrm{d}\tau
	+5\int_s^t\int_{\mathbb{T}^3}\bar{a}_{-}\Delta\varphi_{a}^{-}\mathrm{d}x\mathrm{d}\tau\nonumber\\
	=&-5\int_s^t\|\bar{\mathbf{a}}(\tau)\|_{L^2}^2\mathrm{d}\tau,
\end{align}
where we have used the fact $\int(|v|^2-10)v_i^2\mu(v) \mathrm{d}v=-5$.

For $\uppercase\expandafter{\romannumeral5}$, noting $\eqref{Eq4.2}_2$ and \eqref{Ynew7}, we have
 $\bar{\phi}=\varphi_{a}^{+}-\varphi_{a}^{-}$. Then one obtains that
\begin{align}\label{Eq4.40-1}
	\uppercase\expandafter{\romannumeral5}
	=&\sum_{i,j}\int_s^t\int\partial_i\varphi_{a}^{+}\partial_j\bar{\phi}\cdot\int(|v|^2-10)v_iv_j\mu \mathrm{d}v\nonumber\\
	&-\sum_{i,j}\int_s^t\int\partial_i\varphi_{a}^{-}\partial_j\bar{\phi}\cdot\int(|v|^2-10)v_iv_j\mu \mathrm{d}v\nonumber\\
	=&-5\int_s^t\int(\nabla_x\varphi_{a}^{+}-\nabla_x\varphi_{a}^{-})\cdot\nabla_x\bar{\phi} \mathrm{d}x\mathrm{d}\tau
	=-5\int_s^t\|\nabla_x\bar{\phi}\|_{L^2}^2\mathrm{d}\tau.
\end{align}

For $\uppercase\expandafter{\romannumeral1}$, due to the oddness in $v$,  it follows from \eqref{Eq4.3} and  \eqref{Eq4.15} that 
	\begin{align}\label{Eq4.39}
		\uppercase\expandafter{\romannumeral1}
		=&\int_s^t\langle \mathbf{P}\bar{\mathbf{f}},\partial_t\mathbf{\Psi}_{a} \rangle\mathrm{d}\tau+\int_s^t\langle \{\mathbf{I}-\mathbf{P}\}\bar{\mathbf{f}},\partial_t\mathbf{\Psi}_{a} \rangle\mathrm{d}\tau\nonumber\\
		=&\frac{5}{\sqrt{2}}\int_s^t\int (\mathbf{\bar{b}}\cdot\nabla_x\partial_t\varphi_{a}^{+}
		+ \mathbf{\bar{b}}\cdot\nabla_x\partial_t\varphi_{a}^{-} )\mathrm{d}x\mathrm{d}\tau +\int_s^t\langle \{\mathbf{I}-\mathbf{P}\}\bar{\mathbf{f}},\partial_t\mathbf{\Psi}_{a}\rangle\mathrm{d}\tau\nonumber\\
		\lesssim&
		\int_s^t\left(\|\nabla\Delta^{-1}\partial_t \bar{\mathbf{a}}\|_{L^2}^2
		+\|\mathbf{\bar{b}}\|_{L^2}^2+\|\{\mathbf{I}-\mathbf{P}\}\bar{\mathbf{f}}\|_\nu^2\right)\mathrm{d}\tau	\nonumber \\
		\lesssim&\lambda \int_s^t\|\bar{\mathbf{a}}\|_{L^2}^2\mathrm{d}\tau
		+\int_s^t\left(\|\mathbf{\bar{b}}\|_{L^2}^2+\|\{\mathbf{I}-\mathbf{P}\}\bar{\mathbf{f}}\|_\nu^2\right)\mathrm{d}\tau.
\end{align}

Using the same procedure as in \textit{Step 4.1} for $\uppercase\expandafter{\romannumeral3}$, $\uppercase\expandafter{\romannumeral4}$ and $\uppercase\expandafter{\romannumeral6}$, one can get
\begin{align}\label{Eq4.41}
	\uppercase\expandafter{\romannumeral3}+\uppercase\expandafter{\romannumeral4}+\uppercase\expandafter{\romannumeral6}
	\leq& (\kappa+m_a+\lambda)\int_s^t  \left(\|\nabla_x\bar{\phi}\|_{L^2}^2+\|\bar{\mathbf{a}}\|_{L^2}^2\right)\mathrm{d}\tau 
	+C\int_s^t\|\mathbf{\bar{b}}\|_{L^2}^2\mathrm{d}\tau \nonumber\\
	&+\f{C}{m_a}\int_{s}^t\| \{\mathbf{I}-\mathbf{P}\}\bar{\mathbf{f}}\|_\nu^2\mathrm{d}\tau     
	+\kappa^{\f{1}{2}}
	\int_s^t\|\bar{\mathbf{f}}\|_{L^2}^2\mathrm{d}\tau.
\end{align}

Substituting \eqref{Eq4.40} $-$ \eqref{Eq4.41} into \eqref{Eq4.6}, and noting the smallness of $m_a, \kappa, \lambda$, we obtain that
	\begin{align}\label{Eq4.45}
		\int_s^t\left(\|\bar{\mathbf{a}}\|_{L^2}^2+\|\nabla_x\bar{\phi}\|_{L^2}^2\right)\mathrm{d}\tau  
		\leq& C\int_{s}^{t}\left(\| \{\mathbf{I}-\mathbf{P}\}\bar{\mathbf{f}}\|_\nu^2+\|\bar{\mathbf{b}}\|_{L^2}^2\right)\mathrm{d}\tau 
		+\kappa^{\f{1}{2}}\int_s^t\|\bar{\mathbf{f}}\|_{L^2}^2\mathrm{d}\tau \nonumber\\
		&+G(s)-G(t).
	\end{align}
	
\noindent{\it Step 4.3. Estimate on $\mathbf{\bar{b}}$.} To close the estimate, we need to divide the proof into two cases. \\
	\text{Case 1:} 
	For fixed $i,j$, we choose the test function in \eqref{Eq4.6}
	$$\displaystyle\mathbf{\Psi}=\mathbf{\Psi}_{b,1}^{i,j}:=\begin{bmatrix} 
		( v_i^2-\beta_{b})\sqrt{\mu} \partial_{j}\varphi_{b_{j}}\\ ( v_i^2-\beta_{b})\sqrt{\mu} \partial_{j}\varphi_{b_{j}} \end{bmatrix},\ i, j=1,2,3,$$ 
	where
	$\varphi_{b_{j}}$ is the one defined in \eqref{Ynew8} and $\beta_{b}=1$ so that $\int(v_i^2-\beta_{b})\mu(v)dv=0$.  It is clear that	$\uppercase\expandafter{\romannumeral5}$ in \eqref{Eq4.6} is zero, due to the orthogonality.
	
		For $\uppercase\expandafter{\romannumeral2}$,  the $\bar{a}$, $\bar{c}$ contributions vanish due to oddness in $v$, one has
	\begin{align}\label{Eq4.47}
	\uppercase\expandafter{\romannumeral2}
		=2\sqrt{2}\int_s^t \int_{\mathbb{T}^3} \bar{b}_i\cdot\partial_i\partial_j\varphi_{b_j}\mathrm{d}x\mathrm{d}\tau
		=-2\sqrt{2} \int_s^t\int_{\mathbb{T}^3} \bar{b}_i\cdot(\partial_i\partial_j\Delta^{-1}\bar{b}_j ) \mathrm{d}x\mathrm{d}\tau. 
	\end{align}
	
	For $\uppercase\expandafter{\romannumeral1}$,  due to oddness in $v$ and the choice of $\beta_{b}$,  it follows from \eqref{Eq4.3},  \eqref{Eq4.23} and \eqref{new5.7} that 
	\begin{align}\label{Eq4.46}
		\uppercase\expandafter{\romannumeral1}
		=&\sqrt{2}\int_s^t\int \bar{c}\cdot\partial_j\partial_t\varphi_{b_{j}}\mathrm{d}x\mathrm{d}\tau+\int_s^t \left\langle \{\mathbf{I}-\mathbf{P}\}\bar{\mathbf{f}},\partial_t\mathbf{\Psi}_{b,1}^{i,j} \right\rangle\mathrm{d}\tau\nonumber\\
		\lesssim&~m_b\int_s^t\|\nabla_{x} \partial_{t} \varphi_{b_{j}}(\tau)\|_{L^2}^2\mathrm{d}\tau
		+\f{1}{m_b}\int_s^t\left(\|\bar{c}(\tau)\|_{L^2}^2
		+\|\{\mathbf{I}-\mathbf{P}\}\bar{\mathbf{f}}\|_\nu^2 \right) \mathrm{d}\tau\nonumber\\
		\lesssim&~\lambda\int_{s}^{t}\|\mathbf{\bar{b}}\|_{L^2}^2\mathrm{d}\tau 
		+(m_b+\f{\delta}{m_b} )\int_{s}^{t}\left(\|\bar{\mathbf{a}}\|_{L^2}^2
		+\|\nabla_x\bar{\phi}\|_{L^2}^2\right)\mathrm{d}\tau  
		\nonumber\\
		&+\f{1}{m_b}\int_s^t\left(\|\tilde{c}\|_{L^2}^2+ \|\{\mathbf{I}-\mathbf{P}\}\bar{\mathbf{f}}\|_\nu^2 \right)\mathrm{d}\tau,
	\end{align}
	where $m_b$ is a small constant chosen later.
	
	For $\uppercase\expandafter{\romannumeral6}$, it is obvious from \eqref{new5.7} that
	\begin{align}\label{Eq4.49}
		\uppercase\expandafter{\romannumeral6}
		\leq \lambda \int_{s}^t \|\mathbf{\bar{b}}\|_{L^2}^2\mathrm{d}\tau+\delta\int_{s}^t \|\nabla_x\bar{\phi}\|_{L^2}^2\mathrm{d}\tau
		+\lambda \int_s^t(\|\tilde{c}\|_{L^2}^2+\|\{\mathbf{I}-\mathbf{P}\}\bar{\mathbf{f}}\|_\nu^2)\mathrm{d}\tau,
	\end{align} 
	where we have used \eqref{Ynew9-1}.   
	
For $\uppercase\expandafter{\romannumeral3}, \uppercase\expandafter{\romannumeral4}$, by similar arguments as in \textit{Step 4.1}, one can obtain
\begin{align}\label{Eq4.48}
	\uppercase\expandafter{\romannumeral3}
	+\uppercase\expandafter{\romannumeral4}
	\leq& (m_b+\kappa)\int_s^t ( \|\nabla_x\bar{\phi}\|_{L^2}^2+\|\mathbf{\bar{b}}\|_{L^2}^2)\mathrm{d}\tau 
	+\f{C}{m_b}\int_{s}^t\| \{\mathbf{I}-\mathbf{P}\}\bar{\mathbf{f}}\|_\nu^2 \mathrm{d}\tau  +\kappa^{\f{1}{2}}\int_s^t\|\bar{\mathbf{f}}\|_{L^2}^2\mathrm{d}\tau.
\end{align}

	
	
Plugging \eqref{Eq4.47} $-$ \eqref{Eq4.48}  into \eqref{Eq4.6}, we obtain
	\begin{align}\label{B-1}
		&\Big|\int_s^t\int \bar{b}_i\cdot(\partial_i\partial_j\Delta^{-1}\bar{b}_j ) \mathrm{d}x\mathrm{d}\tau\Big| \nonumber\\
		\leq&~(\lambda+\kappa+m_b)\int_{s}^{t}\|\mathbf{\bar{b}}\|_{L^2}^2\mathrm{d}\tau
			+(m_b+\f{\delta}{m_b} )\int_{s}^{t}\left(\|\bar{\mathbf{a}}\|_{L^2}^2
		+\|\nabla_x\bar{\phi}\|_{L^2}^2\right)\mathrm{d}\tau  
		\nonumber\\
		&+\f{C}{m_b}\int_s^t\left(\|\tilde{c}\|_{L^2}^2+ \|\{\mathbf{I}-\mathbf{P}\}\bar{\mathbf{f}}\|_\nu^2 \right)\mathrm{d}\tau
		+ \kappa^{\f{1}{2}}\int_{s}^{t}\|\bar{\mathbf{f}}\|_{L^2}\mathrm{d}\tau+G(s)-G(t).
	\end{align}
	
		\text{Case 2:} 
We choose the test function in \eqref{Eq4.6}
	$$\displaystyle\mathbf{\Psi}=\mathbf{\Psi}_{b,2}^{i,j}:=\begin{bmatrix} 
		|v|^2v_iv_j\sqrt{\mu} \partial_{j}\varphi_{b_{i}}\\ |v|^2v_iv_j\sqrt{\mu} \partial_{j}\varphi_{b_{i}} \end{bmatrix}, \ i\neq j.$$ Due to the oddness in $v$, it is clear $\uppercase\expandafter{\romannumeral5}=0$.
	
		For $\uppercase\expandafter{\romannumeral2}$, the $\bar{a}$, $\bar{c}$ contributions vanish due to oddness in $v$, one has
	\begin{align}\label{Eq4.52}
		\uppercase\expandafter{\romannumeral2}
		=-7\sqrt{2} \int_s^t\int_{\mathbb{T}^3} \bar{b}_i\cdot(\partial_j\partial_j\Delta^{-1}\bar{b}_i ) +\bar{b}_j\cdot(\partial_i\partial_j \Delta^{-1}\bar{b}_i ) \mathrm{d}x\mathrm{d}\tau .
	\end{align}

	For $\uppercase\expandafter{\romannumeral1}$, due to oddness in $v$,  it follows from  \eqref{Eq4.23} that 
	\begin{align}\label{Eq4.51}
		\uppercase\expandafter{\romannumeral1}
		=&\int_s^t\left\langle  \{\mathbf{I}-\mathbf{P}\}\bar{\mathbf{f}}, \partial_{t}\mathbf{\Psi}_{\mathbf{b},2}^{i,j}\right\rangle \mathrm{d}\tau\nonumber\\
	\lesssim& ~m_b\int_s^t\|\nabla_{x} \partial_{t} \varphi_{b_{i}}(\tau)\|_{L^2}^2\mathrm{d}\tau
		+\f{1}{m_b}\int_s^t\|\{\mathbf{I}-\mathbf{P}\}\bar{\mathbf{f}}\|_\nu^2 \mathrm{d}\tau  \nonumber\\
	\lesssim& m_b\int_s^t \left(\|\bar{\mathbf{a}}\|_{L^2}^2+\|\nabla_x\bar{\phi}\|_{L^2}^2+\|\tilde{c}\|_{L^2}^2 \right)\mathrm{d}\tau
		+\lambda \int_s^t\|\mathbf{\bar{b}}\|_{L^2}^2 \mathrm{d}\tau 
		+\f{1}{m_b}\int_s^t\|\{\mathbf{I}-\mathbf{P}\}\bar{\mathbf{f}}\|_\nu^2 \mathrm{d}\tau .
	\end{align}

Moreover, it is not hard to show that $\uppercase\expandafter{\romannumeral3}$,  $\uppercase\expandafter{\romannumeral4}$ and $\uppercase\expandafter{\romannumeral6}$ are bounded by
\begin{align}\label{Eq4.53}
	m_b\int_s^t \|\mathbf{\bar{b}}\|_{L^2}^2\mathrm{d}\tau
	+\kappa\int_s^t \|\nabla_x\bar{\phi}\|_{L^2}^2\mathrm{d}\tau
	+\f{C}{m_b}\int_{s}^t\| \{\mathbf{I}-\mathbf{P}\}\bar{\mathbf{f}}\|_\nu^2 \mathrm{d}\tau   
	+\kappa^{\f{1}{2}}\int_s^t\|\bar{\mathbf{f}}\|_{L^2}^2\mathrm{d}\tau.
\end{align}

Combining \eqref{Eq4.52} $-$ \eqref{Eq4.53}, one has
	\begin{align}\label{B-2}
	&\Big| \int_s^t\int_{\mathbb{T}^3} \bar{b}_i\cdot(\partial_j\partial_j\Delta^{-1}\bar{b}_i ) +\bar{b}_j\cdot(\partial_i\partial_j \Delta^{-1}b_i ) \mathrm{d}x\mathrm{d}\tau \Big| \nonumber\\
	\leq& G(s)-G(t) +C(m_b+\lambda)\int_{s}^{t}\|\mathbf{\bar{b}}\|_{L^2}^2\mathrm{d}\tau +m_b\int_{s}^{t}(\|\bar{\mathbf{a}}\|_{L^2}^2+\|\nabla_x\bar{\phi}\|_{L^2}^2+\|\tilde{c}\|_{L^2}^2)\mathrm{d}\tau\nonumber\\ 
	& +\f{C}{m_b}\int_{s}^{t}\|\{\mathbf{I}-\mathbf{P}\}\bar{\mathbf{f}}\|_{\nu}^2\mathrm{d}\tau+ \kappa^{\f{1}{2}}\int_{s}^{t}\|\bar{\mathbf{f}}\|_{L^2}\mathrm{d}\tau.
\end{align}

	Then it is follows \eqref{B-1} and \eqref{B-2}  that
\begin{align}\label{Eq4.57}
	\int_s^t\|\mathbf{\bar{b}}(\tau)\|_{L^2}^2\mathrm{d}\tau
&\leq  G(s)-G(t)
+(m_b+\f{\delta}{m_b} )\int_{s}^{t}\left(\|\bar{\mathbf{a}}\|_{L^2}^2
+\|\nabla_x\bar{\phi}\|_{L^2}^2\right)\mathrm{d}\tau  
\nonumber\\
&+\f{C}{m_b}\int_s^t\left(\|\tilde{c}\|_{L^2}^2+ \|\{\mathbf{I}-\mathbf{P}\}\bar{\mathbf{f}}\|_\nu^2 \right)\mathrm{d}\tau
+ \kappa^{\f{1}{2}}\int_{s}^{t}\|\bar{\mathbf{f}}\|_{L^2}\mathrm{d}\tau.
\end{align}

Using \eqref{Eq4.38}, \eqref{Eq4.45} and \eqref{Eq4.57}, taking $m_b=\sqrt{m_c}$ suitably small, and then choose $\delta$ and $\kappa$ sufficiently small, we have
	\begin{equation}\label{newB-3}
	\int_s^t\|\left(\bar{\mathbf{a}},\bar{\mathbf{b}},\tilde{c}\right)(\tau)\|_{L^2}^2+\|\nabla_x\bar{\phi}(\tau)\|_{L^2}^2\mathrm{d}\tau
	\leq G(s)-G(t)+C\int_s^t\|\{\mathbf{I}-\mathbf{P}\}\bar{\mathbf{f}}\|_\nu^2\mathrm{d}\tau,
\end{equation}
which, together with \eqref{new5.7}, yields that
	\begin{equation}\label{newB-4}
	\int_s^t\|\bar{c}(\tau)\|_{L^2}^2\mathrm{d}\tau
	\leq G(s)-G(t)+C\int_s^t\|\{\mathbf{I}-\mathbf{P}\}\bar{\mathbf{f}}\|_\nu^2\mathrm{d}\tau.
\end{equation}
The Lemma \ref{Lem4.1} is now a direct consequence of \eqref{newB-3} $-$ \eqref{newB-4}.
The proof is completed.

\end{proof}

\subsection{Exponential decay of $L^2$ }
Next, we establish the $L^2$ decay estimate for $\mathbf{f}$, which is crucial step to get the exponential decay in $L^{\infty}$ norm.
\begin{proposition}\label{prop5.2}
	Assume \eqref{Eq2.21}, there exists $0<\lambda \ll 1$ such that
	\begin{align*}
		\|\mathbf{f}(t)\|_{L^2}\leq Ce^{-\frac{\lambda}{2} t}.
	\end{align*}
	where the positive constant $C\geq1$ depends only on $M_0$.
\end{proposition}

\begin{proof}
Multiplying \eqref{Eq4.2} by $\bar{\mathbf{f}}$ and integrating over $\mathbb{T}^3\times \mathbb{R}^3$, one obtains that
\begin{align}\label{Eq4.58}
	&\frac{1}{2}\frac{d}{dt}\|\bar{\mathbf{f}}(t)\|_{L^2}^2 
	+\iint \nabla_x\bar{\phi}\cdot v\sqrt{\mu}(\bar{f}_+-\bar{f}_- ) \mathrm{d}v\mathrm{d}x
	+\lambda_0\|\{\mathbf{I}-\mathbf{P}\}\bar{\mathbf{f}}\|_\nu^2\nonumber\\
	\leq &e^{-\lambda t}\iint\nabla_x\bar{\phi}\cdot \frac{v}{2}(|\bar{f}_-|^2- |\bar{f}_+|^2)  \mathrm{d}v\mathrm{d}x
	+e^{-\lambda t}\iint \mathbf{\Gamma}(\bar{\mathbf{f}},\bar{\mathbf{f}})\cdot \bar{\mathbf{f}}\mathrm{d}v\mathrm{d}x\nonumber\\
	&+\lambda\iint (|\bar{f}_+|^2+|\bar{f}_-|^2 ) \mathrm{d}v\mathrm{d}x,
\end{align}
where $	\|\bar{\mathbf{f}}\|_{L^2}^2:=\|(\bar{f}_{+},\bar{f}_{-})\|_{L^2}^2$.
Recalling \eqref{Eq1.4} and \eqref{Eq4.2}, it follows from integration by parts that
\begin{align}\label{Eq4.59}
	\iint_{\mathbb{T}^3\times \mathbb{R}^3}\nabla_x\bar{\phi}\cdot v\sqrt{\mu}( \bar{f}_+-\bar{f}_-)\mathrm{d}v\mathrm{d}x
	=\frac{d}{dt}\Big(\frac{1}{2}\int|\nabla_x\bar{\phi}|^2dx \Big) -\lambda\int|\nabla_x\bar{\phi}|^2dx.
\end{align}

It is easy to show that the first term on RHS of \eqref{Eq4.58} can be bounded by
\begin{align}\label{Eq4.61}
	 \sup\limits_{ x\in\mathbb{T}^3}\Big\{\int_{\mathbb{R}_{v}^3}|\mathbf{f}(t,x,v)|\mathrm{d}v\cdot\|w_\beta \mathbf{f}\|_{L^\infty}\Big\}^{\f{1}{2}}\cdot
	\|\bar{\mathbf{f}}\|_{L^2}^2.
\end{align}
A simple manipulation shows that
\begin{align}\label{Eq4.62}
	&\Big|e^{-\lambda t}\iint_{\mathbb{T}^3\times \mathbb{R}^3}\mathbf{\Gamma}(\bar{\mathbf{f}},\bar{\mathbf{f}})\cdot \bar{\mathbf{f}} \mathrm{d}v\mathrm{d}x\Big|  \nonumber\\
	\leq&\sup\limits_{x\in\mathbb{T}^3} \Big\{\int_{\mathbb{R}_{v}^3}|\mathbf{f}(t,x,v)|\mathrm{d}v\cdot\|w_\beta \mathbf{f}\|_{L^\infty}\Big\}^{\f{1}{2}}
	\|\bar{\mathbf{f}}\|_{L^2} \|\{\mathbf{I}-\mathbf{P}\}\bar{\mathbf{f}}\|_\nu    \nonumber\\
	\leq& \frac{\lambda_0}{2} \|\{\mathbf{I}-\mathbf{P}\}\bar{\mathbf{f}}\|_{\nu }^{2}
	+\frac{1}{2\lambda_0}\sup\limits_{x\in\mathbb{T}^3} \Big\{\int_{\mathbb{R}_{v}^3}|\mathbf{f}(t,x,v)|\mathrm{d}v\cdot\|w_\beta \mathbf{f}\|_{L^\infty}\Big\}\cdot
	\|\bar{\mathbf{f}}\|_{L^2}^2.
\end{align}
Substituting \eqref{Eq4.59}-\eqref{Eq4.62}  into \eqref{Eq4.58}, one has that
\begin{align}\label{Eq4.63}
	&\frac{1}{2}\frac{d}{dt}\left(\|\bar{\mathbf{f}}(t)\|_{L^2}^2+\|\nabla_x\bar{\phi}(t)\|_{L^2}^2 \right) 
	-\lambda\left(\|\bar{\mathbf{f}}(t)\|_{L^2}^2+\|\nabla_x\bar{\phi}(t)\|_{L^2}^2 \right)+\frac{\lambda_0}{2}\|\{\mathbf{I}-\mathbf{P}\}\bar{\mathbf{f}}\|_\nu^2\nonumber\\
	\leq& \Big(\frac{1}{2\lambda_0}\sup\limits_{x\in\mathbb{T}^3} \Big\{\int_{\mathbb{R}_{v}^3}|\mathbf{f}(t,x,v)|\mathrm{d}v\cdot\|w_\beta \mathbf{f}\|_{L^\infty}\Big\}+\sup\limits_{ x\in\mathbb{T}^3} \Big\{\int_{\mathbb{R}_{v}^3}|\mathbf{f}(t,x,v)|\mathrm{d}v\|w_\beta \mathbf{f}\|_{L^\infty}\Big\}^{\f{1}{2}}\Big)
	\cdot\|\bar{\mathbf{f}}\|_{L^2}^2.
\end{align}
For $t \leq 1$, it follows from \eqref{Eq4.63} and Proposition \ref{prop4.3} that
\begin{align}\label{Eq4.64}
	&\left(\|\bar{\mathbf{f}}(t)\|^2_{L^2}+\|\nabla_x\bar{\phi}\|^2_{L^2}\right)+\frac{\lambda_0}{2}\int_0^t\|(\mathbf{I}-\mathbf{P})\bar{\mathbf{f}}(\tau)\|^2_{\nu}\mathrm{d}\tau\nonumber\\
	\le &\left(\|\mathbf{f}_0\|^2_{L^2}+\|\nabla_x\phi_0\|^2_{L^2}\right)\exp\Big\{\f{1}{\lambda_0}\left(\|w_{\beta}\mathbf{f}\|_{L^{\infty}}^2+1\right)\Big\}\nonumber\\
	\le & \left(\|\mathbf{f}_0\|^2_{L^2}+\|\nabla_x\phi_0\|^2_{L^2}\right)\exp\Big\{C(1+M_0)^4\Big\}, \quad 0\le t\le 1.
\end{align}
For $t \geq 1$, integrating \eqref{Eq4.63} over time on $(1, t)$, and noting Remark \ref{Rem4.4}, one has
\begin{align}\label{Eq4.65}
	&\frac{1}{2}\left(\|\bar{\mathbf{f}}(t)\|_{L^2}^2+\|\nabla_x\bar{\phi}(t)\|_{L^2}^2\right)
	+\frac{\lambda_0}{2}\int_1^t\|\{\mathbf{I}-\mathbf{P}\}\bar{\mathbf{f}}\|_\nu^2\nonumber\\
	\leq&\frac{1}{2}\left(\|\bar{\mathbf{f}}(1)\|_{L^2}^2+\|\nabla_x\bar{\phi}(1)\|_{L^2}^2\right)
	+(\kappa+\lambda) \int_1^t\left(\|\bar{\mathbf{f}}(\tau)\|_{L^2}^2+\|\nabla_x\bar{\phi}(\tau)\|_{L^2}^2 \right)\mathrm{d}\tau.
\end{align}

Substituting $\eqref{Eq4.5}$ into \eqref{Eq4.65}
\begin{align}
	&\Big(\frac{1}{2}-\kappa-\lambda\Big)\left(\|\bar{\mathbf{f}}(t)\|_{L^2}^2
	+\|\nabla_x\bar{\phi}(t)\|_{L^2}^2\right)
	+\Big(\f{\lambda_0}{2}-(\kappa+\lambda) (C_7+1)\Big)\int_1^t\|\{\mathbf{I}-\mathbf{P}\}\bar{\mathbf{f}}\|_\nu^2     \nonumber\\
	\leq& C\left(\|\mathbf{f}(1)\|_{L^2}^2+\|\nabla_x\bar{\phi}(1)\|_{L^2}^2 \right). 
\end{align}
Taking $\lambda\leq \f{\lambda_0}{4(C_7+1)}$, and noting $\kappa$ is sufficiently small,  one has
\begin{align*}
	\|\bar{\mathbf{f}}(t)\|_{L^2}^2+\|\nabla_x\bar{\phi}(t)\|_{L^2}^2
	\leq C\left(\|\bar{\mathbf{f}}(1)\|_{L^2}^2+\|\nabla_x\bar{\phi}(1)\|_{L^2}^2\right),
\end{align*}
which yields that
\begin{equation*}
	\|\mathbf{f}(t)\|_{L^2}^2\leq Ce^{-\lambda t},
\end{equation*}
where $C$ depend on $M_0$ and $\lambda_0$.  Therefore the proof of Proposition \ref{prop5.2} is completed.
\end{proof}

We now present the main conclusions of this section.
\begin{proposition}\label{Decay}
	Assume	\eqref{Eq2.21}, there is $0<\lambda_1\leq \min\{\frac{\lambda}{2}, \f{\tilde{\nu}_0}{2}\}$ such that
	\begin{equation}\label{Eq4.67}
		\|\mathbf{h}(t)\|_{L^\infty}\leq Ce^{-\lambda_1 t},
	\end{equation}
	where the positive constant $C \geq 1$ depends only on $M_0, \beta_1, \beta$.
\end{proposition}
\begin{proof}
According to the standard $L^2-L^\infty$ method of the Boltzmann equation, one can easily infer \eqref{Eq4.67}. Here we omit the details for brevity of presentations.  The reader can refer to  \cite{Guo10ARMA, G-J10CMP, Duan17ARMA} for more details.  Therefore the proof of Proposition \ref{Decay} is completed. 
\end{proof}

\section{Uniform $W^{1,\infty}$-Bound}  \label{section 6}
In order to close the free streaming condition \eqref{Eq2.21}, the remaining task is to perform  the $W^{1, \infty}$ estimate of $\mathbf{f}$.
Recall $\tilde{h}_{\pm}$ in \eqref{Eq1.23}, which satisfy \eqref{Eq3.31} $-$ \eqref{Eq3.33}, one can  establish the following two key lemmas.
\begin{lemma}\label{Lem5.1}
	Assume \eqref{Eq2.21}, let $4\leq\b_{1} <\b -4$,  there exits constant $C$, depending on $\b$, $\b_1$ and $M_0$,  such that  
	\begin{align*}
		\|\partial_x \mathbf{\tilde{h}}(t)\|_{L^{\infty}}
		\le C(1+\|\partial_{x,v}\mathbf{\tilde{h}}_{0}\|_{L^\infty})^2  
		+\delta \sup_{0 \leq s\leq t}\|\partial_v\mathbf{\tilde{h}}(s)\|_{L^\infty}.
	\end{align*} 
\end{lemma}

\begin{lemma}\label{Lem5.2}
	Assume  \eqref{Eq2.21}, let $4\leq\b_{1} <\b -4$,  there exists constant $C$, depending on $\b$, $\b_1$ and $M_0$,  such that 
	\begin{align*}
		\|\partial_v\mathbf{\tilde{h}}(t)\|_{L^\infty}\leq C(1+\|\partial_{x,v}\mathbf{\tilde{h}}_0\|_{L^\infty})^2 +\sup_{0 \leq s\leq t}	\|\partial_x\mathbf{\tilde{h}}(s)\|_{L^\infty}.
	\end{align*}
\end{lemma}

With the help of the two lemmas we can establish the following proposition:
\begin{proposition}\label{Prop5.3}
	Assume \eqref{Eq2.21}, let $4\leq\b_{1} <\b -4$,  there exists constant $C$, depending on $\b$, $\b_1$ and $M_0$,  such that 
	\begin{equation*}
		\|\partial_{x,v}\mathbf{\tilde{h}}(t)\|_{L^\infty} \leq C(1+\|\partial_{x,v}\mathbf{\tilde{h}}_0\|_{L^\infty})^2.
	\end{equation*}
\end{proposition}
\begin{proof}
Proposition \ref{Prop5.3} follows from Lemmas \ref{Lem5.1} and \ref{Lem5.2}. Therefore the proof of Proposition \ref{Prop5.3} is completed.
\end{proof}

In the following two subsections, we prove the above two lemmas.

\subsection{Proof of Lemma \ref{Lem5.1}} \label{section 6.1}
Recall $\partial_i=\frac{\partial}{\partial x_i}$, applying $\partial_i$ to  \eqref{Eq3.31}, we get 
\begin{align}\label{Eq5.3}
	&(\partial_t+v \cdot \nabla _x\mp\nabla_x \phi \cdot \nabla_v)\partial_i\tilde{h}_{\pm}+\tilde{\nu}_{\pm,1}\partial_i\tilde{h}_{\pm} \nonumber\\
	=&\pm\nabla_x \partial_i\phi \cdot \nabla_v\tilde{h}_{\pm} -\partial_i\tilde{\nu}_{\pm,1}\tilde{h}_{\pm} 
	\mp\nabla_x \partial_i\phi \cdot v w_{\b_1}\sqrt{\mu} + \partial_i(w_{\b_1}\mathrm{\Gamma}^{\pm}(\mathbf{f},\mathbf{f}))+\partial_i(w_{\b_1}\mathrm{K}^{\pm}\mathbf{f}).
\end{align}
Integrating \eqref{Eq5.3} along the characteristic lines, one obtains
\begin{align}\label{Eq5.4}
	\partial_i\tilde{h}_{\pm}(t,x,v)
	&=\partial_i\tilde{h}_{\pm0}(X_{\pm}(0),V_{\pm}(0))e^{-\int_0^t\tilde{\nu}_{\pm,1}(\tau, X_{\pm}(\tau), V_{\pm}(\tau))\mathrm{d}\tau}\nonumber\\
	&\quad \pm\int_0^t e^{-\int_s^t\tilde{\nu}_{\pm,1}(\tau, X_{\pm}(\tau), V_{\pm}(\tau))\mathrm{d}\tau}\left(\nabla_x \partial_i\phi \cdot \nabla_v\tilde{h}_{\pm}\right)(s,X_{\pm}(s),V_{\pm}(s))\mathrm{d}s\nonumber\\
	&\quad -\int_0^te^{-\int_s^t\tilde{\nu}_{\pm,1}(\tau, X_{\pm}(\tau), V_{\pm}(\tau))\mathrm{d}\tau}\left(\partial_i\tilde{\nu}_{\pm,1}\tilde{h}_{\pm}\right)(s,X_{\pm}(s),V_{\pm}(s))\mathrm{d}s\nonumber\\
	&\quad \mp\int_0^te^{-\int_s^t\tilde{\nu}_{\pm,1}(\tau, X_{\pm}(\tau), V_{\pm}(\tau))\mathrm{d}\tau}\left(\nabla_x\partial_i\phi\cdot vw_{\b_1}\sqrt{\mu}\right)(s,X_{\pm}(s),V_{\pm}(s))\mathrm{d}s\nonumber\\
	&\quad +\int_0^te^{-\int_s^t\tilde{\nu}_{\pm,1}(\tau, X_{\pm}(\tau), V_{\pm}(\tau))\mathrm{d}\tau}\left(\partial_i(w_{\b_1}\mathrm{\Gamma}^{\pm}(\mathbf{f},\mathbf{f}))\right)(s,X_{\pm}(s),V_{\pm}(s))\mathrm{d}s\nonumber\\
	&\quad +\int_0^te^{-\int_s^t\tilde{\nu}_{\pm,1}(\tau, X_{\pm}(\tau), V_{\pm}(\tau))\mathrm{d}\tau}\left(\partial_i(w_{\b_1}\mathrm{K}^{\pm}\mathbf{f})\right)(s,X_{\pm}(s),V_{\pm}(s))\mathrm{d}s
	=\sum_{i=0}^{5}H_i.
\end{align}
For the nonlinear collision term $H_4$, it follows from \eqref{new6.11}, \eqref{E2.3}, \eqref{Eq4.67}, Proposition \ref{prop4.3} and Remark \ref{Rem4.4} that 
\begin{align}\label{new6.12}
	&\int_0^te^{-\int_s^t\tilde{\nu}_{\pm,1}(\tau, X_{\pm}(\tau), V_{\pm}(\tau))\mathrm{d}\tau}\left|\left(\partial_i(w_{\b_1}\mathrm{\Gamma}^{\pm}(\mathbf{f},\mathbf{f}))\right)(s,X_{\pm}(s),V_{\pm}(s))\right| \mathrm{d}s   \nonumber \\
	\leq&\Big\{\int_{0}^{t_1}+\int_{t_1}^{t} \Big\}e^{-\f{\nu(v)}{2}(t-s)}\nu(v) \|\partial_x\mathbf{\tilde{h}}(s)\|_{L^\infty}\|\mathbf{h}(s)\|_{L^\infty}^{\f{1}{2}} \cdot\sup\limits_{x\in\mathbb{T}^3} \Big\{\int_{\mathbb{R}_{v}^3} |\mathbf{\tilde{h}}(s,x,v)|\mathrm{d}v\Big\}^{\f{1}{2}}    \mathrm{d}s  \nonumber\\
	&+\int_0^te^{-\tilde{\nu_0}(t-s)} \|\mathbf{h}(s)\|_{L^\infty}\|\partial_x\mathbf{\tilde{h}}(s)\|_{L^\infty}\mathrm{d}s   \nonumber\\
	\leq& C(1+\|\partial_{x, v}\mathbf{\tilde{h}}_{0})\|_{L^\infty})^2 +\int_0^te^{-\tilde{\nu_0}(t-s)} e^{-\lambda_1s}\|\partial_x\mathbf{\tilde{h}}(s)\|_{L^\infty}\mathrm{d}s \nonumber\\
	&+\int_{t_1}^{t}e^{-\f{\nu(v)}{2}(t-s)}\nu(v) \|\partial_x\mathbf{\tilde{h}}(s)\|_{L^\infty}\|\mathbf{h}(s)\|_{L^\infty}^{\f{1}{2}}\cdot\sup\limits_{x\in\mathbb{T}^3} \Big\{\int_{\mathbb{R}_{v}^3} |\mathbf{\tilde{h}}(s,x,v)|\mathrm{d}v\Big\}^{\f{1}{2}}    \mathrm{d}s  \nonumber\\
	\leq& C(1+\|\partial_{x, v}\mathbf{\tilde{h}}_{0})\|_{L^\infty})^2 +\kappa ^{\f{1}{2}}\sup_{0 \leq s\leq t}\|\partial_x\mathbf{\tilde{h}}(s)\|_{L^\infty}+\int_0^t e^{-\lambda_1 s}\|\partial_x\mathbf{\tilde{h}}(s)\|_{L^\infty}\mathrm{d}s.
\end{align}

Next we estimate  the remaining terms on RHS of \eqref{Eq5.4} as follows.

It follows from \eqref{Eq3.32} and \eqref{Eq2.21} that
\begin{align}\label{Eq5.7}
	|\partial_i\tilde{\nu}_{\pm,1}(t, x, v)|\leq C\delta(1+t)^{-\frac{5}{2}}(1+|v|).
\end{align}
Then plugging \eqref{Eq2.21}, \eqref{Eq2.47}, \eqref{Eq3.33}, \eqref{new6.12} $-$  \eqref{Eq5.7} and Proposition \ref{prop4.3} into \eqref{Eq5.4}, one has
\begin{align}\label{Eq5.11}
	|\partial_x \tilde{h}_{\pm}(t,x,v)|
	& \leq C(1+\|\partial_{x, v}\mathbf{\tilde{h}}_{0})\|_{L^\infty})^2 +\delta \sup_{0 \leq s\leq t}\|\partial_v\mathbf{\tilde{h}}(s)\|_{L^\infty} \nonumber\\
	  &\quad+\kappa ^{\f{1}{2}}\sup_{0 \leq s\leq t}\|\partial_x\mathbf{\tilde{h}}(s)\|_{L^\infty} 
	  +\int_0^t e^{-\lambda_1 s}\|\partial_x\mathbf{\tilde{h}}(s)\|_{L^\infty}\mathrm{d}s \nonumber\\
	&\quad  +\Big|\int_0^te^{-\int_s^t\tilde{\nu}_{\pm,1}(\tau)\mathrm{d}\tau}
	\int_{\mathbb{R}^3}\mathtt{k}_{w_{\b_1}}^{(2)}(V_{\pm}(s),u)\partial_x \tilde{h}_{\pm}(s,X_{\pm}(s),u)\mathrm{d}u\Big|    \nonumber\\
	&\quad  +\Big|\int_0^te^{-\int_s^t\tilde{\nu}_{\pm,1}(\tau)\mathrm{d}\tau}
	\int_{\mathbb{R}^3}\mathtt{k}_{w_{\b_1}}^{(1)}(V_{\pm}(s),u)
	\partial_x \tilde{h}_{\mp}(s,X_{\pm}(s),u)du \Big|.
\end{align}

We only deal with the case for $|\partial_i\tilde{h}_{+}(t,x,v)|$, because $|\partial_i\tilde{h}_{-}(t,x,v)|$  can be controlled in the same way. We denote $\hat{\nu}_{\pm,1}(\tau_1):=\tilde{\nu}_{\pm,1}(\tau_1, \hat{X}(\tau_1),  \hat{V}(\tau_1))$ for simplicity of presentation, where $ \hat{X}(\tau_1),  \hat{V}(\tau_1)$ has been defined in \eqref{Hat1}.  Indeed, using \eqref{Eq5.11} again for $\partial_x \tilde{h}_{\pm}(s,X_{\pm}(s),u)$, one has
\begin{align}\label{Eq5.12}
	|\partial_x \tilde{h}_{+}(t,x,v)|
	&\le C(1+\|\partial_{x, v}\mathbf{\tilde{h}}_{0})\|_{L^\infty})^2 +\delta \sup_{0 \leq s\leq t}\|\partial_v\mathbf{\tilde{h}}(s)\|_{L^\infty}
	+\kappa ^{\f{1}{2}}\sup_{0 \leq s\leq t}\|\partial_x\mathbf{\tilde{h}}(s)\|_{L^\infty} \nonumber\\
	&\quad  +\int_0^t e^{-\lambda_1 s}\|\partial_x\mathbf{\tilde{h}}(s)\|_{L^\infty}\mathrm{d}s +\sum_{i=1}^{4}I_i,
\end{align}
where,
\begin{align*}
	I_1:=&\int_0^te^{-\int_s^t\tilde{\nu}_{+,1}(\tau)\mathrm{d}\tau}\mathrm{d}s\int_0^se^{-\int_{s_1}^{s}\hat{\nu}_{+,1}(\tau_1)\mathrm{d}\tau_1}\mathrm{d}s_1  \\
	&\times\int_{\mathbb{R}^3}\int_{\mathbb{R}^3}\mathtt{k}_{w_{\b_1}}^{(2)}(V_{+}(s),u)\cdot\mathtt{k}_{w_{\b_1}}^{(2)}(\hat{V}_{+}(s_1),u_1)\cdot \partial_x \tilde{h}_{+}(s_1,\hat{X}_{+}(s_1),u_1)du_1\mathrm{d}u,   \\
	I_2:=&\int_0^te^{-\int_s^t\tilde{\nu}_{+,1}(\tau)\mathrm{d}\tau}\mathrm{d}s\int_0^se^{-\int_{s_1}^{s}\hat{\nu}_{+,1}(\tau_1)\mathrm{d}\tau_1}\mathrm{d}s_1  \\
	&\times\int_{\mathbb{R}^3}\int_{\mathbb{R}^3}\mathtt{k}_{w_{\b_1}}^{(2)}(V_{+}(s),u)\cdot\mathtt{k}_{w_{\b_1}}^{(1)}(\hat{V}_{+}(s_1),u_1)\cdot \partial_x \tilde{h}_{-}(s_1,\hat{X}_{+}(s_1),u_1)du_1\mathrm{d}u,   \\
	I_3:=&\int_0^te^{-\int_s^t\tilde{\nu}_{+,1}(\tau)\mathrm{d}\tau}\mathrm{d}s\int_0^se^{-\int_{s_1}^{s}\hat{\nu}_{-,1}(\tau_1)\mathrm{d}\tau_1}\mathrm{d}s_1  \\
	&\times\int_{\mathbb{R}^3}\int_{\mathbb{R}^3}\mathtt{k}_{w_{\b_1}}^{(1)}(V_{+}(s),u)\cdot\mathtt{k}_{w_{\b_1}}^{(2)}(\hat{V}_{-}(s_1),u_1)\cdot \partial_x \tilde{h}_{-}(s_1,\hat{X}_{-}(s_1),u_1)du_1\mathrm{d}u,  \\
	I_4:=&\int_0^te^{-\int_s^t\tilde{\nu}_{+,1}(\tau)\mathrm{d}\tau}\mathrm{d}s\int_0^se^{-\int_{s_1}^{s}\hat{\nu}_{-,1}(\tau_1)\mathrm{d}\tau_1}\mathrm{d}s_1   \\ &\times\int_{\mathbb{R}^3}\int_{\mathbb{R}^3}\mathtt{k}_{w_{\b_1}}^{(1)}(V_{+}(s),u)\cdot
	\mathtt{k}_{w_{\b_1}}^{(1)}(\hat{V}_{-}(s_1),u_1)\cdot \partial_x \tilde{h}_{+}(s_1,\hat{X}_{-}(s_1),u_1)du_1\mathrm{d}u.
\end{align*}

We  discuss $I_1$ of \eqref{Eq5.12} and split it into four cases.\\

\text{Case 1:} $|v|\ge N$.  By similar argument as in \eqref{Eq3.19},
\begin{align}\label{Eq5.16}
	I_1\le \frac{C}{N}\sup_{0\le s\le t}\|\partial_x \mathbf{\tilde{h}}(s)\|_{L^{\infty}}.
\end{align}

\text{Case 2:} For either $|v|\le N, |u|\ge 2N$ or $|u|\le 2N, |u_1|\ge 3N$. By similar argument as in \eqref{Eq3.22},
\begin{align}\label{Eq5.17}
	I_1 \le Ce^{-\frac{N^2}{64}}\sup_{0\le s\le t}\|\partial_x \mathbf{\tilde{h}}(s)\|_{L^{\infty}}.
\end{align}

\text{Case 3:} $|v|\le N, |u|\le 2N, |u_1|\le 3N, s-\f{1}{N}\le s_1\le s$.  By similar argument as in \eqref{Eq3.23},
\begin{align}\label{Eq5.18}
	I_1\le \f{C}{N}\sup_{0\le s\le t}\|\partial_x \mathbf{\tilde{h}}(s)\|_{L^{\infty}}.
\end{align}

\text{Case 4:} $|v|\le N, |u|\le 2N, |u_1|\le 3N, 0\le s_1 \le s-\f{1}{N}$.
\begin{align}\label{Eq5.19}
	I_1&\cong \int_0^t \int_0^{s-\f{1}{N}} e^{-\int_s^t\tilde{\nu}_{+,1}(\tau)\mathrm{d}\tau}e^{-\int_{s_1}^{s}\hat{\nu}_{+,1}(\tau_1)\mathrm{d}\tau_1}\mathrm{d}s_1ds   \nonumber\\
	&\quad\times \iint_{B}\mathtt{k}_{w_{\b_1}}^{(2)}(V_{+}(s),u)\cdot\mathtt{k}_{w_{\b_1}}^{(2)}(\hat{V}_{+}(s_1),u_1)\cdot \partial_x \tilde{h}_{+}(s_1,\hat{X}_{+}(s_1),u_1)du_1du,
\end{align}
where  $B=\{(u,u_1):~|u|\leq 2N,$ $|u_1|\leq 3N\}.$ From \eqref{newK3} $-$ \eqref{newK5}, $\mathtt{k}_{w_{\b_1}}^{(2)}(v,u)$ has integrable singularity of $\frac{%
	1}{|v-u|},$ we can choose $\mathtt{k}_{N}^{(2)}(v,u)$ smooth with
compact support such that
\begin{equation*}
	\sup_{|v|\leq 3N}\int_{|u|\leq 3N}|\mathtt{k}_{N}^{(2)}(v,u)-\mathtt{k}_{w_{\b_1}}^{(2)}(v, u)|\mathrm{d}u\leq \frac{1}{N}, \quad  |\mathtt{k}_{N}^{(2)}(v,u)|+|\partial_{v,u}\mathtt{k}_{N}^{(2)}(v,u)|\leq C_N.
\end{equation*}%
Noting that
\begin{align*}
	&\mathtt{k}_{w_{\b_1}}^{(2)}(V_{+}(s),u)\cdot\mathtt{k}_{w_{\b_1}}^{(2)}(\hat{V}_{+}(s_1),u_1)  
	=\left(\mathtt{k}_{w_{\b_1}}^{(2)}(V_{+}(s),u)-\mathtt{k}_{N}^{(2)}(V_{+}(s),u)\right)\cdot\mathtt{k}_{w_{\b_1}}^{(2)}(\hat{V}_{+}(s_1),u_1)    \\
	&+\left(\mathtt{k}_{w_{\b_1}}^{(2)}(\hat{V}_{+}(s_1),u_1)-\mathtt{k}_{N}^{(2)}(\hat{V}_{+}(s_1),u_1)\right)\cdot\mathtt{k}_{N}^{(2)}(V_{+}(s),u)
	+\mathtt{k}_{N}^{(2)}(V_{+}(s),u)\cdot\mathtt{k}_{N}^{(2)}(\hat{V}_{+}(s_1),u_1),
\end{align*}
we can estimate $I_1$ by
\begin{align}\label{Eq5.20}
	I_1&\leq \frac{C}{N}\sup_{0\le s\le t}\|\partial_x \mathbf{\tilde{h}}(s)\|_{L^{\infty}}    
	+ \int_0^t \int_0^{s-\f{1}{N}} e^{-\int_s^t\tilde{\nu}_{+,1}(\tau)\mathrm{d}\tau}e^{-\int_{s_1}^{s}\hat{\nu}_{+,1}(\tau_1)\mathrm{d}\tau_1}\mathrm{d}s_1ds   \nonumber\\
	&\quad\times \iint_{B}\mathtt{k}_{N}^{(2)}(V_{+}(s),u)\mathtt{k}_{N}^{(2)}(\hat{V}_{+}(s_1),u_1)\cdot \partial_x \tilde{h}_{+}(s_1,\hat{X}_{+}(s_1),u_1)du_1du =:I_{1}^{*}.
\end{align}
Consider a change of variable:
\begin{equation}\label{Eq5.21}
	y=\hat{X}_{+}(s_{1})=X_{+}(s_{1};s,X_{+}(s;t,x,v),u)  
\end{equation}
such that
\begin{align}\label{Eq5.22}
	I_{1}^{*}&\cong \int_0^t \int_0^{s-\f{1}{N}} e^{-\int_s^t\tilde{\nu}_{+,1}(\tau)\mathrm{d}\tau}e^{-\int_{s_1}^{s}\hat{\nu}_{+,1}(\tau_1)\mathrm{d}\tau_1}\mathrm{d}s_1ds  \nonumber\\
	&\quad\times\iint_{\hat{B}}\mathtt{k}_{N}^{(2)}(V_{+}(s),u)\mathtt{k}_{N}^{(2)}(\hat{V}_{+}(s_1),u_1)\cdot \partial_y \tilde{h}_{+}(s_1,y,u_1)|det(\frac{\partial u}{\partial y})|du_1dy,
\end{align}
where $\hat{B}$ is the image of $B$ under the transformation $(u, u_1) \rightarrow (y, u_1)$. In addition, it is clear that $\hat{B}\subseteq\{(y, u_1):~|y-X_{+}(s)|\leq C(s-s_{1})N,\;|u_1|\leq 3N\}$. After integrating by parts, one has
\begin{align}\label{Eq5.23}
	I_{1}^{*}&= \int_0^t \int_0^{s-\f{1}{N}} e^{-\int_s^t\tilde{\nu}_{+,1}(\tau)\mathrm{d}\tau} \partial_y\{e^{-\int_{s_1}^{s}\hat{\nu}_{+,1}(\tau_1)\mathrm{d}\tau_1}\}\mathrm{d}s_1ds   \nonumber\\
	&\quad\times \iint_{\hat{B}}\mathtt{k}_{N}^{(2)}(V_{+}(s),u)\mathtt{k}_{N}^{(2)}(\hat{V}_{+}(s_1),u_1)\cdot  \tilde{h}_{+}(s_1,y,u_1)|det(\frac{\partial u}{\partial y})|du_1dy \nonumber\\
	&\quad +\int_0^t \int_0^{s-\f{1}{N}} e^{-\int_s^t\tilde{\nu}_{+,1}(\tau)\mathrm{d}\tau}e^{-\int_{s_1}^{s}\hat{\nu}_{+,1}(\tau_1)\mathrm{d}\tau_1}\mathrm{d}s_1ds   \nonumber\\
	&\quad\times \iint_{\hat{B}} \partial_y\{\mathtt{k}_{N}^{(2)}(V_{+}(s),u)\mathtt{k}_{N}^{(2)}(\hat{V}_{+}(s_1),u_1)\}\cdot  \tilde{h}_{+}(s_1,y,u_1)|det(\frac{\partial u}{\partial y})|du_1dy    \nonumber\\
	&\quad +\int_0^t \int_0^{s-\f{1}{N}} e^{-\int_s^t\tilde{\nu}_{+,1}(\tau)\mathrm{d}\tau}e^{-\int_{s_1}^{s}\hat{\nu}_{+,1}(\tau_1)\mathrm{d}\tau_1}\mathrm{d}s_1ds    \nonumber\\
	&\quad\times \iint_{\hat{B}}\mathtt{k}_{N}^{(2)}(V_{+}(s),u)\mathtt{k}_{N}^{(2)}(\hat{V}_{+}(s_1),u_1)\cdot  \tilde{h}_{+}(s_1,y,u_1)\partial_y\{|det(\frac{\partial u}{\partial y})|\}du_1dy   \nonumber\\
	&\quad +\int_0^t \int_0^{s-\f{1}{N}} e^{-\int_s^t\tilde{\nu}_{+,1}(\tau)\mathrm{d}\tau}e^{-\int_{s_1}^{s}\hat{\nu}_{+,1}(\tau_1)\mathrm{d}\tau_1}\mathrm{d}s_1ds    \nonumber\\
	&\quad\times \iint_{\partial_{y}\hat{B}}\mathtt{k}_{N}^{(2)}(V_{+}(s),u)\mathtt{k}_{N}^{(2)}(\hat{V}_{+}(s_1),u_1)\cdot  \tilde{h}_{+}(s_1,y,u_1)|det(\frac{\partial u}{\partial y})|du_1dS_y   \nonumber\\
	&=: I_{11}^{*}+I_{12}^{*}+I_{13}^{*}+I_{B}.
\end{align}
For the boundary term $I_{B}$, it follows from Corollary \ref{Coro2.8}  and Proposition \ref{prop4.3} that
\begin{equation}\label{newB}
	I_B \leq C_N\sup_{0 \leq s\leq t}\|h(s)\|_{L^\infty} \leq C_N.
\end{equation}

To estimate $I_{11}^{*}$, we notice that
\begin{align}\label{Eq5.24}
	\partial_y\left(e^{-\int_{s_1}^{s}\hat{\nu}_{+,1}(\tau_1)\mathrm{d}\tau_1}\right)=-e^{-\int_{s_1}^{s}\hat{\nu}_{+,1}(\tau_1)\mathrm{d}\tau_1}\cdot \int_{s_1}^{s}\partial_y\left(\hat{\nu}_{+,1}(\tau_1)\right)\mathrm{d}\tau_1,
\end{align}
and
\begin{align}\label{Eq5.25}
	\partial_y\left(\tilde{\nu}_{+,1}(\tau_1, \hat{X}_{+}(\tau_1), \hat{V}_{+}(\tau_1))\right)=\partial_x \hat{\nu}_{+,1}(\tau_1)\cdot\frac{\partial \hat{X}_{+}(\tau_1)}{\partial u}\cdot\frac{\partial u}{\partial y}+\partial_v \hat{\nu}_{+,1}(\tau_1)\cdot\frac{\partial \hat{V}_{+}(\tau_1)}{\partial u}\cdot\frac{\partial u}{\partial y}.
\end{align}
It is clear that
\begin{equation}\label{Eq5.25-1}
	\left|\partial_x \hat{\nu}_{+,1}(\tau_1)\right|+\left|\partial_v \hat{\nu}_{+,1}(\tau_1)\right|\leq C(1+|\hat{V}_{+}(\tau_1)|)\leq C(1+|u|)\leq C_{N}.
\end{equation}

It follows from \eqref{Eq2.23}  that
\begin{align}\label{Eq5.26}
	\Big|\frac{\partial \hat{X}_{+}(\tau_1)}{\partial u}\Big|\leq C(s-s_1), \quad \Big|\frac{\partial \hat{V}_{+}(\tau_1)}{\partial u}\Big|\leq C\delta (s-s_1)+1, \quad  \Big|\frac{\partial u}{\partial y}\Big|\leq C(s-s_1)^{-1}.
\end{align}
Combining \eqref{Eq5.25} $-$ \eqref{Eq5.26}, one gets
\begin{align}\label{Eq5.27}
	\left|\partial_y\left(\tilde{\nu}_{+,1}(\tau_1, \hat{X}_{+}(\tau_1), \hat{V}_{+}(\tau_1))\right)\right|
	\leq C_{N}\left(1+(s-s_1)^{-1}\right),
\end{align}
Which, together with \eqref{Eq5.24}, yields
\begin{align}\label{Eq5.28}
	\left|\partial_y\left(e^{-\int_{s_1}^{s}\hat{\nu}_{+,1}(\tau_1)\mathrm{d}\tau_1}\right)\right|\leq C_{N}e^{-\int_{s_1}^{s}\hat{\nu}_{+,1}(\tau_1)\mathrm{d}\tau_1}\left(1+(s-s_1)\right).
\end{align}
Substituting \eqref{Eq5.28} into \eqref{Eq5.23}, following the same argument in $L^{\infty }$ bound, one can deduce that 
\begin{align}\label{Eq5.29}
	I_{11}^{*} &\leq C_N\int_0^t\int_{0}^{s-\f{1}{N}}e^{-\tilde{\nu}_0(t-s)}e^{-\tilde{\nu}_0(s-s_1)}(1+(s-s_1))\mathrm{d}s_1\mathrm{d}s\Big(\iint_{\hat{B}}|f_+(s_1, y, u_1)|^2 dydu_1\Big)^{\frac{1}{2}}  \nonumber\\
	&\leq C_N\Big(\|h_{+}(\tau)\|^{\frac{1}{2}}_{L^{\infty}}\sqrt{\mathcal{E}(\mathbf{F}_0)}+\sqrt{\mathcal{E}(\mathbf{F}_0)}\Big).
\end{align}

For $I_{12}^{*}$, a direct calculation shows that 
\begin{align}\label{Eq5.30}
	\partial_y\left(\mathtt{k}_{N}^{(2)}(V_{+}(s),u)\cdot\mathtt{k}_{N}^{(2)}(\hat{V}_{+}(s_1),u_1)\right)=&(\partial_u \mathtt{k}_{N}^{(2)})(V_{+}(s),u)\cdot\frac{\partial u}{\partial y}\cdot \mathtt{k}_{N}^{(2)}(\hat{V}_{+}(s_1),u_1)  
	+\mathtt{k}_{N}^{(2)}(V_{+}(s),u)\nonumber\\
	&\times(\partial_v \mathtt{k}_{N}^{(2)})(\hat{V}_{+}(s_1),u_1)\cdot\frac{\partial \hat{V}_{+}(s_1)}{\partial u}\cdot\frac{\partial u}{\partial y},
\end{align}
which, together with \eqref{Eq2.23}, yields that
\begin{align}\label{Eq5.31}
	\left|\partial_y\left(\mathtt{k}_{N}^{(2)}(V_{+}(s),u)\mathtt{k}_{N}^{(2)}(\hat{V}_{+}(s_1),u_1)\right)\right| \leq C_{N}\left(1+(s-s_1)^{-1}\right).
\end{align}
Combining \eqref{Eq5.31} and \eqref{Eq5.23}, by similar arguments as in \eqref{Eq5.29}, one can obtain that
\begin{align}\label{Eq5.32}
	I_{12}^{*} \leq C_N\Big(\|h_{+}(\tau)\|^{\frac{1}{2}}_{L^{\infty}}\sqrt{\mathcal{E}(\mathbf{F}_0)}+\sqrt{\mathcal{E}(\mathbf{F}_0)}\Big).
\end{align}

For $I_{13}^{*}$, we denote $z=X_{+}(s; t,x,v)$. Recalling $y=\hat{X}(s_{1})=X(s_{1};s,X_{+}(s;t,x,v),u)= X(s_{1};s,z,u)$, then 
\begin{align}\label{Eq5.33}
	I_{13}^{*} &\leq C_N\int_0^t\int_{0}^{s-\f{1}{N}}e^{-\tilde{\nu}_0(t-s)}e^{-\tilde{\nu}_0(s-s_1)}\mathrm{d}s_1ds   \nonumber\\
	&\quad \times\Big(\iint_{\hat{B}}|f_+(s_1, y, u_1)|^2 dydu_1\Big)^{\frac{1}{2}}\cdot\Big(\iint_{\hat{B}}|\partial_y\Big(|\det(\frac{\partial u}{\partial y})|\Big)|^2dydu_1\Big)^{\frac{1}{2}}. 
\end{align}
A direct calculation shows that
\begin{equation}\label{Eq5.34}
	\partial_y\Big(\det \Big( \frac{\partial u}{\partial y}\Big) \Big)=\partial_y\Big(\frac{1}{\det
		\Big( \frac{\partial y}{\partial u}\Big) }\Big)=-\frac{1}{\Big(\det \Big( \frac{\partial y}{%
			\partial u}\Big) \Big)^{2}}\partial_u\Big(\det \Big( \frac{\partial y}{\partial u}\Big)
	\Big)\cdot \frac{\partial u}{\partial y},
\end{equation}%
and
\begin{equation}\label{Eq5.35}
	\left|\partial_u\left(\det \left( \frac{\partial y}{\partial u}\right)\right)\right|\leq \left|\frac{\partial \hat{X}_+(s_1)}{\partial u}\right|^2
	\left|\partial^2_{uu} \hat{X}_+(s_1)\right|\leq C(s-s_1)^2\left|\partial^2_{uu} X(s_1; s,z,u)\right|.
\end{equation}%
Then it follows from \eqref{Eq2.23},  \eqref{Eq5.34} $-$ \eqref{Eq5.35}  that 
\begin{equation}\label{Eq5.36}
	\Big|\partial_y\Big(\det \Big( \frac{\partial u}{\partial y}\Big) \Big)\Big| \leq C(s-s_1)\frac{1}{\Big(\det \Big( \frac{\partial y}{%
			\partial u}\Big) \Big)^{2}}\Big|\partial^2_{uu} X(s_1; s,z,u)\Big|.
\end{equation}%
Substituting \eqref{Eq5.36}  into \eqref{Eq5.33}, one has
\begin{align}\label{Eq5.36-1}
	&\Big(\iint_{\hat{B}}|\partial_y\Big(|det(\frac{\partial u}{\partial y})|\Big)|^2\mathrm{d}y\mathrm{d}u_1\Big)^{\frac{1}{2}}   \nonumber \\
	 \leq& C(s-s_1)\Big(\iint_{\hat{B}}\frac{1}{\Big(\det \Big( \frac{\partial y}{%
	 		\partial u}\Big) \Big)^{4}}|\partial^2_{uu} X(s_1; s,z,u)|^2\mathrm{d}y\mathrm{d}u_1\Big)^{\frac{1}{2}}   \nonumber \\
	\leq& C(s-s_1)\Big(\iint_{B}\frac{1}{\left(\det \left( \frac{\partial y}{%
			\partial u}\right)\right)^{4}}|\partial^2_{uu} X(s_1; s,z,u)|^2|\det(\frac{\partial y}{\partial u})|\mathrm{d}u\mathrm{d}u_1\Big)^{\frac{1}{2}}     \nonumber \\
   \leq& C_N(s-s_1)^{-\f{7}{2}}\Big(\int_{|u|\leq 2N}|\partial^2_{uu} X(s_1; s,z,u)|^2\mathrm{d}u\Big)^{\frac{1}{2}},   
\end{align}
where we have used  Corollary \ref{Coro2.8}.
Noting Lemma \ref{Lem2.9}, we have from \eqref{Eq5.36-1} that
\begin{align}\label{Eq5.37}
	\Big(\iint_{\hat{B}}|\partial_y\Big(|det(\frac{\partial u}{\partial y})|\Big)|^2\mathrm{d}y\mathrm{d}u_1\Big)^{\frac{1}{2}}   \leq C_N\sup_{0\le s\le t}\|\partial_x \mathbf{\tilde{h}}(s)\|_{L^{\infty}}.
\end{align}
Following the same argument as in \eqref{Eq3.26}, together with \eqref{Eq5.37} and \eqref{Eq5.33}, one has
\begin{align}\label{Eq5.38}
	I_{13}^{*}&\leq C_{N}\Big(\sup_{0 \leq s\leq t}\|h(s)\|^{\frac{1}{2}}_{L^{\infty}}\sqrt{\mathcal{E}(\mathbf{F}_0)}+\sqrt{\mathcal{E}(\mathbf{F}_0)}\Big)\sup_{0\le s\le t}\|\partial_x \mathbf{\tilde{h}}(s)\|_{L^{\infty}} \nonumber \\
	&\leq \frac{1}{N}\sup_{0\le s\le t}\|\partial_x \mathbf{\tilde{h}}(s)\|_{L^{\infty}},
\end{align}
where $\mathcal{E}(\mathbf{F}_0)\leq \v_1$ with $\v_1$ sufficiently small.  Then it follows from \eqref{newB}, \eqref{Eq5.29}, \eqref{Eq5.32} and \eqref{Eq5.38}  that 
\begin{equation}\label{Eq5.39}
	I_1^{*} \leq \frac{C}{N}\sup_{0\le s\le t}\|\partial_x \mathbf{\tilde{h}}(s)\|_{L^{\infty}}+C_N.
\end{equation}%

Combining \eqref{Eq5.16} $-$ \eqref{Eq5.18}, \eqref{Eq5.20} and \eqref{Eq5.39}, we have
\begin{align}\label{Eq5.40}
	I_1 \leq\frac{C}{N}\sup_{0\le s\le t}\|\partial_x \mathbf{\tilde{h}}(s)\|_{L^{\infty}}+C_N.
\end{align}

Following analogous reasoning to \eqref{Eq5.16} $-$ \eqref{Eq5.40}, we deduce that for $i=2,3,4,5$
\begin{align}\label{Eq5.40-1}
	I_i \leq\frac{C}{N}\sup_{0\le s\le t}\|\partial_x \mathbf{\tilde{h}}(s)\|_{L^{\infty}}+C_N.
\end{align}

Plugging \eqref{Eq5.40} and \eqref{Eq5.40-1} into  \eqref{Eq5.12},  one obtains that
\begin{align}\label{Eq5.41}
	\|\partial_x \tilde{h}_{+}(t)\|_{L^{\infty}}
	&\le C(1+\|\partial_{x, v}\mathbf{\tilde{h}}_{0})\|_{L^\infty})^2  +C_N +(\frac{C}{N}+\kappa ^{\f{1}{2}})\sup_{0 \leq s\leq t}\|\partial_x\mathbf{\tilde{h}}(s)\|_{L^\infty} \nonumber\\
	&\quad  +\delta \sup_{0 \leq s\leq t}\|\partial_v\mathbf{\tilde{h}}(s)\|_{L^\infty}+\int_0^t e^{-\lambda_1 s}\|\partial_x\mathbf{\tilde{h}}(s)\|_{L^\infty}\mathrm{d}s. 
\end{align}
Similarly, $h_{-}(t,x,v)$ has the same estimate as \eqref{Eq5.41}. Thus,
\begin{align}\label{Eq5.42}
	\|\partial_x \mathbf{\tilde{h}}(t)\|_{L^{\infty}}
	&\le C(1+\|\partial_{x, v}\mathbf{\tilde{h}}_{0})\|_{L^\infty})^2  +C_N +(\frac{C}{N}+\kappa ^{\f{1}{2}})\sup_{0 \leq s\leq t}\|\partial_x\mathbf{\tilde{h}}(s)\|_{L^\infty} \nonumber\\
	&\quad  +\delta \sup_{0 \leq s\leq t}\|\partial_v\mathbf{\tilde{h}}(s)\|_{L^\infty}+\int_0^t e^{-\lambda_1 s}\|\partial_x\mathbf{\tilde{h}}(s)\|_{L^\infty}\mathrm{d}s. 
\end{align}
Then Lemma \ref{Lem5.1} follows from Gronwall's inequality and $N\gg 1$, $\kappa\ll1$. Therefore the proof of Lemma \ref{Lem5.1} is completed.

$\hfill\Box$


\subsection{Proof of Lemma \ref{Lem5.2}}

Recall $\partial^j=\frac{\partial}{\partial v_j}$, applying $\partial^j$ to  \eqref{Eq3.31}, we obtain 
\begin{align}\label{Eq5.44}
	&(\partial_t+v \cdot \nabla _x \mp\nabla_x \phi \cdot \nabla_v)\partial^j \tilde{h}_{\pm}+\tilde{\nu}_{\pm,1}\partial^j \tilde{h}_{\pm}   \nonumber\\
	=&-\partial_j \tilde{h}_{\pm} -\partial^j \tilde{\nu}_{\pm,1}\tilde{h}_{\pm} \mp\nabla_x \phi \cdot \partial^j (v w_{\b_1}\sqrt{\mu})   + \partial^j(w_{\b_1}\mathrm{\Gamma}^{\pm}(\mathbf{f},\mathbf{f}))+\partial^j (w_{\b_1}K_{\pm}f).
\end{align}
Integrating  \eqref{Eq5.44} along the characteristic, one has 
\begin{align}\label{Eq5.45}
	\partial^j \tilde{h}_{\pm}(t,x,v)
	&=\partial^j \tilde{h}_{\pm0}(X_{\pm}(0),V_{\pm}(0))e^{-\int_0^t\tilde{\nu}_{\pm,1}(\tau, X_{\pm}(\tau), V_{\pm}(\tau))\mathrm{d}\tau}\nonumber\\
	&\quad -\int_0^te^{-\int_s^t\tilde{\nu}_{\pm,1}(\tau, X_{\pm}(\tau), V_{\pm}(\tau))\mathrm{d}\tau}\left(\partial_j \tilde{h}_{\pm}\right)(s,X_{\pm}(s),V_{\pm}(s))\mathrm{d}s\nonumber\\
	&\quad -\int_0^te^{-\int_s^t\tilde{\nu}_{\pm,1}(\tau, X_{\pm}(\tau), V_{\pm}(\tau))\mathrm{d}\tau}\left(\partial^j \tilde{\nu}_{\pm,1}\tilde{h}_{\pm}\right)(s,X_{\pm}(s),V_{\pm}(s))\mathrm{d}s\nonumber\\
	&\quad \mp\int_0^te^{-\int_s^t\tilde{\nu}_{\pm,1}(\tau, X_{\pm}(\tau), V_{\pm}(\tau))\mathrm{d}\tau}\left(\nabla_x \phi \cdot \partial^j (v w_{\b_1}\sqrt{\mu})\right)(s,X_{\pm}(s),V_{\pm}(s))\mathrm{d}s\nonumber\\
	&\quad +\int_0^te^{-\int_s^t\tilde{\nu}_{\pm,1}(\tau, X_{\pm}(\tau), V_{\pm}(\tau))\mathrm{d}\tau}\left(\partial^j(w_{\b_1}\mathrm{\Gamma}^{\pm}(\mathbf{f},\mathbf{f}))\right)(s,X_{\pm}(s),V_{\pm}(s))\mathrm{d}s\nonumber\\
	&\quad +\int_0^te^{-\int_s^t\tilde{\nu}_{\pm,1}(\tau, X_{\pm}(\tau), V_{\pm}(\tau))\mathrm{d}\tau}\left(\partial^j (w_{\b_1}\mathrm{K}^{\pm}\mathbf{f})\right)(s,X_{\pm}(s),V_{\pm}(s))\mathrm{d}s
	=\sum_{i=0}^{5}J_i.
\end{align}

For the nonlinear collision term $J_4$,	using \eqref{new6.18} and similar arguments as in \eqref{new6.12}, one gets that
\begin{align}\label{new6.19}
	&\int_0^te^{-\int_s^t\tilde{\nu}_{\pm,1}(\tau, X_{\pm}(\tau), V_{\pm}(\tau))\mathrm{d}\tau}\left|\left(\partial^{j}(w_{\b_1}\mathrm{\Gamma}^{\pm}(\mathbf{f},\mathbf{f}))\right)(s,X_{\pm}(s),V_{\pm}(s))\right| \mathrm{d}s   \nonumber \\
	\leq& C(1+\|\partial_{x, v}\mathbf{\tilde{h}}_{0})\|_{L^\infty})^2 +\kappa ^{\f{1}{2}}\sup_{0 \leq s\leq t}\|\partial_v\mathbf{\tilde{h}}(s)\|_{L^\infty}+\int_0^t e^{-\lambda_1 s}\|\partial_v\mathbf{\tilde{h}}(s)\|_{L^\infty}\mathrm{d}s.
\end{align}

Next, we deal with the remaining terms on RHS of \eqref{Eq5.45}.
It is easy to see that
\begin{equation}\label{Eq5.46}
	|\partial_{v_j}\tilde{\nu}_{\pm,1}(t,x,v)|\leq C(1+|v| ),
\end{equation}
which, together with   \eqref{Eq4.67} and \eqref{new6.19}, yields  that
\begin{align}\label{Eq5.49}
	\sum_{i=0}^{4}J_i&\leq C(\|\partial_v\mathbf{\tilde{h}}_0\|_{L^\infty}+1)^2 +\sup_{s}\|\partial_x\mathbf{\tilde{h}}(s)\|_{L^\infty}
	+\kappa ^{\f{1}{2}}\sup_{0 \leq s\leq t}\|\partial_v\mathbf{\tilde{h}}(s)\|_{L^\infty}   \nonumber\\
	&\quad+\int_0^t e^{-\lambda_1 s}\|\partial_v\mathbf{\tilde{h}}(s)\|_{L^\infty}\mathrm{d}s.
\end{align}

For $J_5$,  we only consider $\mathrm{K}^{+}\mathbf{f}$, because $\mathrm{K}^{-}\mathbf{f}$ is similar. 
Using \eqref{Eq2.50}, it is easy to have
\begin{align}\label{Eq5.51}
	J_5
	=&\int_0^te^{-\int_s^t\tilde{\nu}_{+,1}(\tau)\mathrm{d}\tau}\left(\partial_{v_j}w_{\b_1}\cdot\mathrm{K}^{+} \mathbf{f}\right)(s,X_+(s),V_+(s)) \mathrm{d}s\nonumber  \\
	&+\int_0^te^{-\int_s^t\tilde{\nu}_{+,1}(\tau)\mathrm{d}\tau}\int\tilde{\mathtt{k}}_{ w_{\b_1}}^{(2)}(V_+(s),u)\tilde{h}_{+}(s,X_+(s),u)du \mathrm{d}s\nonumber\\
	&+\int_0^te^{-\int_s^t\tilde{\nu}_{+,1}(\tau)\mathrm{d}\tau}\int\tilde{\mathtt{k}}_{ w_{\b_1}}^{(1)}(V_+(s),u)\tilde{h}_{-}(s,X_+(s),u)du \mathrm{d}s\nonumber\\
	&+\int_0^te^{-\int_s^t\tilde{\nu}_{+,1}(\tau)\mathrm{d}\tau}\int\mathtt{k}_{ w_{\b_1}}^{(2)}(V_+(s),u)\cdot(w_{\b_1}\partial_{v_j}f_{+})(s,X_+(s),u)du \mathrm{d}s\nonumber\\
	&+\int_0^te^{-\int_s^t\tilde{\nu}_{+,1}(\tau)\mathrm{d}\tau}\int\mathtt{k}_{ w_{\b_1}}^{(1)}(V_+(s),u)\cdot(w_{\b_1}\partial_{v_j}f_{-})(s,X_+(s),u)du \mathrm{d}s \nonumber\\
	=:&J_{51}+J_{52}^{+}+J_{52}^{-}+J_{53}^{+}+J_{53}^{-}.
\end{align}
It follows from \eqref{Eq2.41} that
\begin{equation}\label{Eq5.52}
	J_{51}+J_{52}^{+}+J_{52}^{-}\leq\int_0^te^{-\nu_0(t-s)}\|\mathbf{h}(s)\|_{L^\infty} \mathrm{d}s\leq C.	
\end{equation}
We only need to deal with $J_{53}^{+}$ since $J_{53}^{-}$ can be dealt in the same way. It is clear that 
\begin{align}\label{newV}
	J_{53}^{+} \leq& \int_0^te^{-\tilde{\nu_0}(t-s)} \int\mathtt{k}_{ w_{\b_1}}^{(2)}(V_+(s),u)\cdot\left((\partial_{v_j}w_{\b_1})f_{+}\right)(s,X_+(s),u)du \mathrm{d}s\nonumber\\
	&+\int_0^te^{-\int_s^t\tilde{\nu}_{+,1}(\tau)\mathrm{d}\tau} \int\mathtt{k}_{ w_{\b_1}}^{(2)}(V_+(s),u)\cdot\partial_{v_j}\tilde{h}_{+}(s,X_+(s),u) \mathrm{d}u \mathrm{d}s  \nonumber \\
	\leq& ~C+ \bar{J}_{53}^{+}.
\end{align}
To estimate $\bar{J}_{53}^{+}$,  performing the same procedure as in \eqref{Eq5.11} $-$ \eqref{Eq5.20}, one has
\begin{align}\label{Eq5.47}
	\bar{J}_{53}^{+}\leq& C(\|\partial_{x,v}\mathbf{\tilde{h}}_0\|_{L^\infty}+1)^2 +\sup_{s}\|\partial_x\mathbf{\tilde{h}}(s)\|_{L^\infty}
	+\Big(\f{C}{N}+\kappa ^{\f{1}{2}}\Big)\sup_{0 \leq s\leq t}\|\partial_v\mathbf{\tilde{h}}(s)\|_{L^\infty}   \nonumber\\
	&\quad+\int_0^t e^{-\lambda_1 s}\|\partial_v\mathbf{\tilde{h}}(s)\|_{L^\infty}\mathrm{d}s   
	+(**),
\end{align}
where
\begin{align*}
	(**)\cong&\int_0^t\int_0^{s-\f{1}{N}}e^{-\int_s^t\tilde{\nu}_{\pm,1}(\tau)\mathrm{d}\tau}e^{-\int_{s_1}^s\hat{\nu}_{\pm,1}(\tau_1))\mathrm{d}\tau_1}\mathrm{d}s_1\mathrm{d}s    \\
	&\times\iint_{B} \mathtt{k}_{N}^{(2)}(V_{+}(s),u)\cdot\mathtt{k}_{N}^{(2)}(\hat{V}_{+}(s_1),u_1)\partial_{v_j}\tilde{h}_{+}(s_1, \hat{X}_{+}(s_1), u_1)\mathrm{d}u_1\mathrm{d}u\nonumber\\
	=&\int_0^t\int_0^{s-\f{1}{N}}e^{-\int_s^t\tilde{\nu}_{\pm,1}(\tau)\mathrm{d}\tau}e^{-\int_{s_1}^s\hat{\nu}_{\pm,1}(\tau_1))\mathrm{d}\tau_1}\mathrm{d}s_1\mathrm{d}s    \\
	&\times\iint_{B} \mathtt{k}_{N}^{(2)}(V_{+}(s),u)\cdot\mathtt{k}_{N}^{(2)}(\hat{V}_{+}(s_1),u_1)\partial_{u_{1,j}}\tilde{h}_{+}(s_1, \hat{X}_{+}(s_1), u_1)\mathrm{d}u_1\mathrm{d}u,
\end{align*}
here $u_{1,j}$  is the $j$-th component of $u_1$.
Denote $y=X_+(s_1,s,X_+(s,t,x,v),u)$, after integrating by parts, and noting the boundary term contribution can be controlled by $C_N\sup\limits_{0 \leq s\leq t}\|\mathbf{h}(s)\|_{L^\infty}$, we have
\begin{align}\label{Eq5.54}
	(**)\leq&\int_0^t\int_0^{s-\f{1}{N}}\int_{\hat{B}}e^{-\int_s^t\tilde{\nu}_{\pm,1}(\tau)\mathrm{d}\tau}e^{-\int_{s_1}^{s}\hat{\nu}_{\pm,1}(\tau_1)\mathrm{d}\tau_1}\mathtt{k}_{N}^{(2)}(V_+(s),u)\partial_{u_{1,j}}\mathtt{k}_{N}^{(2)}(\hat{V}_+(s_1),u_1)   \nonumber\\
	&\times\tilde{h}_{+}(s_1,y,u_1)\left|\det\left(\frac{\partial u}{\partial y} \right)  \right| dydu_1
	+C_N\sup_{0 \leq s\leq t}\|\mathbf{h}(s)\|_{L^\infty}\nonumber\\
	\leq&C_N\Big( \sup_{0 \leq s\leq t}\|\mathbf{h}(s)\|_{L^\infty}^{1/2}\sqrt{\mathcal{E}(\mathbf{F}_0)}+\sqrt{\mathcal{E}(\mathbf{F}_0)}\Big)+C_N \leq C_N.
\end{align}
It is a consequence of \eqref{Eq5.45}, \eqref{Eq5.49} and \eqref{Eq5.51}--\eqref{Eq5.54} that
\begin{align*}
	\|\partial_v\mathbf{\tilde{h}}(t)\|_{L^\infty}\leq& C(1+\|\partial_v\mathbf{\tilde{h}}_0\|_{L^\infty} )^2 
	+C_N
	+\sup_{0 \leq s\leq t}	\|\partial_x\mathbf{\tilde{h}}(s)\|_{L^\infty}+\Big(\frac{1}{N}+\kappa^{\f{1}{2}}\Big) \sup_{0 \leq s\leq t}\|\partial_v\mathbf{\tilde{h}}(s)\|_{L^\infty}\nonumber\\
	&+\int_0^t e^{-\lambda_1 s}\|\partial_v\mathbf{\tilde{h}}(s)\|_{L^\infty}\mathrm{d}s.
\end{align*}
Taking $N\gg1$, and noting $\kappa\ll 1$, Lemma \ref{Lem5.2} follows from Gronwall's inequality. Therefore the proof of Lemma \ref{Lem5.2} is completed.
$\hfill\Box$

\section{Proof of Theorem \ref{Thm1.1}}  \label{section 7}

\medskip
\noindent{ Step 1.} 
Noting Propositions \ref{Decay} and \ref{Prop5.3}, we only need to close \eqref{Eq1.27} $-$ \eqref{Eq1.25} and the {\it a priori} assumption \eqref{Eq2.21}.

It follows from \eqref{L44} $-$ \eqref{L49} that for any $t\geq 0$,
\begin{align}\label{Eq6.4-0}
	\|(f_{+}-f_{-})(t)\|_{L^{\infty}} \leq &
	(1+\sup_{0 \leq s\leq t}\|\mathbf{h}(s)\|_{L^\infty}+\sup_{0 \leq s\leq t}\|\partial_v\mathbf{\tilde{h}}(s)\|_{L^\infty})
	\int_0^t \|(f_{+}-f_{-})(s)\|_{L^{\infty}}\mathrm{d}s \nonumber\\
	&+\|(f_{+}-f_{-})(0)\|_{L^{\infty}}.
\end{align}
 Then it is a consequence of  Propositions \ref{prop4.3} and \ref{Prop5.3}  that for  $t\geq 0$
\begin{align}\label{Eq6.4}
		\|(f_{+}-f_{-})(t)\|_{L^{\infty}}\leq \v_0 
		+C(1+\|\partial_{x,v}\mathbf{\tilde{h}}_0\|_{L^\infty})^2 \int_0^t\|(f_{+}-f_{-})(s)\|_{L^{\infty}}\mathrm{d}s,
\end{align}
which, together with Gronwall's inequality, yields that 
\begin{align}\label{Eq6.5}
	\|(f_{+}-f_{-})(t)\|_{L^{\infty}}\leq \v_0 \exp\left\{C(1+\|\partial_{x,v}\mathbf{\tilde{h}}_0\|_{L^\infty})^2t\right\},
\end{align}
where the positive constant $C>0$ depends on $\|\mathbf{h}_{0}\|_{L^\infty}$. Combining \eqref{Eq2.19} and \eqref{Eq6.5}, one has
\begin{align}\label{Eq6.6}
	\|\nabla_x \phi(t)\|_{L^{\infty}}\leq C	\|(f_{+}-f_{-})(t)\|_{L^{\infty}}\leq C\v_0 \exp\left\{C(1+\|\partial_{x,v}\mathbf{\tilde{h}}_0\|_{L^\infty})^2t\right\}. 
\end{align}
On the other hand, it is follows from \eqref{Eq4.67} that
\begin{align}\label{Eq6.7}
	\|\nabla_x \phi(t)\|_{L^{\infty}}\leq C\|\mathbf{h}(t)\|_{L^\infty}\leq Ce^{-\lambda_1 t}.  
\end{align}
Thus we have from \eqref{Eq6.6} $-$ \eqref{Eq6.7} that
\begin{align}\label{Eq6.9}
	\|\nabla_x \phi(t)\|_{L^{\infty}}\lesssim \min\left\{\v_0 \exp\{C(1+\|\partial_{x,v}\mathbf{\tilde{h}}_0\|_{L^\infty})^2t\}, e^{-\lambda_1 t}\right\}.  
\end{align}

Taking $d= \delta^2e^{-\lambda_1t}(1+\|\partial_{x, v}\mathbf{\tilde{h}}_{0}\|_{L^\infty})^{-2}$ and $R=\frac{1}{4}$ in  Lemma \ref{Lem2.5}, and using Proposition \ref{Prop5.3}, one has
\begin{align}\label{Eq6.10-0}
	\|\nabla^2_x\phi(t)\|_{L^{\infty}}&\le C\|(f_{+}-f_{-})(t)\|_{L^{\infty}}\Big(1+\ln \frac{R}{d}+R^{-3}\Big)+\|\partial_{x}\mathbf{\tilde{h}}(t)\|_{L^\infty}d  \nonumber\\
	&\lesssim \delta^2 e^{-\lambda_1t}+\|(f_{+}-f_{-})(t)\|_{L^{\infty}}\Big(1+2\ln \Big(\f{1}{\delta}\Big)+\lambda_1 t+\ln (1+\|\partial_{x,v}\mathbf{\tilde{h}}_0\|_{L^\infty})\Big),
\end{align}
which, together with \eqref{Eq6.5}, yields that
\begin{align}\label{Eq6.10}
	\|\nabla^2_x\phi(t)\|_{L^{\infty}}&\lesssim \delta^2 e^{-\lambda_1t} + \v_0 \mathfrak{H}(t),
\end{align}
where $\mathfrak{H}(t):=\left(1+2\ln \Big(\f{1}{\delta}\Big)+\lambda_1 t+\ln (1+\|\partial_{x,v}\mathbf{\tilde{h}}_0\|_{L^\infty})\right)\exp\left\{C(1+\|\partial_{x,v}\mathbf{\tilde{h}}_0\|_{L^\infty})^2t\right\}$.

On the other side, it follows from \eqref{Eq6.10-0} and Proposition \ref{Decay} that 
\begin{align}\label{Eq6.11}
	\|\nabla^2_x\phi(t)\|_{L^{\infty}}
	\lesssim \delta^2 e^{-\lambda_1t} + e^{-\lambda_2t} \Big(1+2\ln \Big(\f{1}{\delta}\Big)+\ln (1+\|\partial_{x,v}\mathbf{\tilde{h}}_0\|_{L^\infty})\Big), \ \text{where} \ 0<\lambda_2\leq\f{\lambda_1}{2}.
\end{align}

 Noting \eqref{Eq6.7} $-$ \eqref{Eq6.11},  we choose $\v_0$ sufficiently small such that 
\begin{align}\label{Eq6.12}
	\min\left\{\v_0 \exp\{C(1+\|\partial_{x,v}\mathbf{\tilde{h}}_0\|_{L^\infty})^2t\}, e^{-\lambda_1 t}\right\}\leq \delta^2(1+t)^{-2},
\end{align}
and
\begin{align}\label{Eq6.13}
	\min\Big\{\v_0 \mathfrak{H}(t), e^{-\lambda_2t}\Big(1+2\ln \Big(\f{1}{\delta}\Big)+\ln (1+\|\partial_{x,v}\mathbf{\tilde{h}}_0\|_{L^\infty})\Big)\Big\}\leq \delta^2(1+t)^{-\f{5}{2}}.
\end{align}

Thus we conclude the {\it a priori} assumption \eqref{Eq2.21} from \eqref{Eq6.9} $-$ \eqref{Eq6.13}.

\medskip
\noindent{ Step 2.} 
Recall  \eqref{Eq5.4}, Proposition \ref{Prop5.3} and \eqref{Eq6.11}, it is easy to have
\begin{align}\label{Eq6.15}
	H_1+H_3\leq& C\int_0^t e^{-\tilde{\nu}_0 (t-s)}\|\nabla^2_x\phi(s)\|_{L^{\infty}}(1+
	\|\partial_{v}\mathbf{\tilde{h}}(s)\|_{L^\infty})\mathrm{d}s   \nonumber\\
	\leq& C(1+\|\partial_{x,v}\mathbf{\tilde{h}}_0\|_{L^\infty})^2 \Big(\ln \Big(\f{1}{\delta}\Big)+\ln (1+\|\partial_{x,v}\mathbf{\tilde{h}}_0\|_{L^\infty})\Big)\int_0^t e^{-\tilde{\nu}_0 (t-s)}e^{-\lambda_2s} \mathrm{d}s   \nonumber\\
	\leq& C(1+\|\partial_{x,v}\mathbf{\tilde{h}}_0\|_{L^\infty})^2 \Big(\ln \Big(\f{1}{\delta}\Big)+\ln (1+\|\partial_{x,v}\mathbf{\tilde{h}}_0\|_{L^\infty})\Big)e^{-\lambda_2t}.
\end{align}

In addition, it is not difficult to verify 
\begin{align}\label{Eq6.16}
 H_2+H_4 \leq C(1+\|\partial_{x,v}\mathbf{\tilde{h}}_0\|_{L^\infty})^2e^{-\lambda_1t}.
\end{align}

It remains to estimate $H_5$. We shall adopt the same procedure as in \eqref{Eq5.11} $-$ \eqref{Eq5.20}. Indeed, in view of $\int_0^t e^{-\tilde{\nu}_0 (t-s)}e^{-\lambda_2s} \mathrm{d}s \leq Ce^{-\lambda_2t}$, all of the above cases in Section \ref{section 6.1} except the last one can be bounded by
\begin{align}\label{Eq6.19}
\frac{C}{N}e^{-\lambda_2t}\sup_{0\le s\le t}\{e^{\lambda_2s}\|\partial_x \mathbf{\tilde{h}}(s)\|_{L^{\infty}}\}.
\end{align}

 For the last case, analogous to the treatment of \eqref{Eq5.19}--\eqref{Eq5.38},  it follows from integration by part that 
\begin{align}\label{Eq6.23}
	 I_{1}^{*} \leq& \frac{C}{N}e^{-\lambda_2t}\sup_{0\le s\le t}\{e^{\lambda_2s}\|\partial_x \mathbf{\tilde{h}}(s)\|_{L^{\infty}}\}  +C_Ne^{-\lambda_1 t}\nonumber\\
	 &+C_N\int_0^t\int_{0}^{s-\f{1}{N}}
	 e^{-\tilde{\nu}_0(t-s)}e^{-\tilde{\nu}_0(s-s_1)}\|\mathbf{f}(s_1)\|_{L^2}^{\f{1}{2}} \mathrm{d}s_1\mathrm{d}s\cdot\Big(\iint_{\hat{B}}|f_+(s_1, y, u_1)|^2 dydu_1\Big)^{\frac{1}{4}}  \nonumber\\
	 &+C_N\int_0^t\int_{0}^{s-\f{1}{N}}
	 e^{-\tilde{\nu}_0(t-s)}e^{-\tilde{\nu}_0(s-s_1)}\|\mathbf{f}(s_1)\|_{L^2}^{\f{1}{2}}
	 \mathrm{d}s_1ds   
	\cdot\Big(\iint_{\hat{B}}|f_+(s_1, y, u_1)|^2 dydu_1\Big)^{\frac{1}{4}} \nonumber\\
	 &\quad\times\Big(\iint_{\hat{B}}|\partial_y\Big(|\det(\frac{\partial u}{\partial y})|\Big)|^2dydu_1\Big)^{\frac{1}{2}} \nonumber\\
	 \leq&  \frac{C}{N}e^{-\lambda_2t}\sup_{0\le s\le t}\{e^{\lambda_2s}\|\partial_x \mathbf{\tilde{h}}(s)\|_{L^{\infty}}\} +C_Ne^{-\lambda_1 t} +C_{N}\mathcal{E}^{\f{1}{4}}(\mathbf{F}_0)
	 (1+\|\partial_{x,v}\mathbf{\tilde{h}}_0\|_{L^\infty})^2e^{-\f{\lambda}{4}t}.
\end{align}
Noting $\lambda_2 \leq \f{\lambda_1}{2} \leq \f{\lambda}{4}$, using \eqref{Eq6.15} $-$ \eqref{Eq6.23}, one can choose $N\geq1$ large enough, then let  $\v_1$ suitably small, thus we obtain
\begin{align}\label{Eq6.26}
		\|\partial_{x}\mathbf{\tilde{h}}(t)\|_{L^\infty}  \leq C(1+\|\partial_{x,v}\mathbf{\tilde{h}}_0\|_{L^\infty})^2 \Big(\ln \Big(\f{1}{\delta}\Big)+\ln (1+\|\partial_{x,v}\mathbf{\tilde{h}}_0\|_{L^\infty})\Big)e^{-\lambda_2t}.
\end{align}
Similarly,
\begin{align}\label{Eq6.27}
\|\partial_{v}\mathbf{\tilde{h}}(t)\|_{L^\infty}  \leq C(1+\|\partial_{x,v}\mathbf{\tilde{h}}_0\|_{L^\infty})^2 \Big(\ln \Big(\f{1}{\delta}\Big)+\ln (1+\|\partial_{x,v}\mathbf{\tilde{h}}_0\|_{L^\infty})\Big)e^{-\lambda_2t}.
\end{align}
Combining \eqref{Eq6.26}, \eqref{Eq6.27} and Proposition \ref{Prop5.3},  we conclude \eqref{Eq1.27}.
Therefore the proof of Theorem\ref{Thm1.1} is completed. 
$\hfill\Box$

\bigskip

\noindent {\bf Acknowledgments:}  Zaihong Jiang's research is partially supported by the National Natural Science Foundation of China, grants No. 12071439. Yong Wang's research is partially supported by the National Natural Science Foundation of China, grants No. 12022114, No. 12421001 and No. 12288201, and CAS Project for Young Scientists in Basic Research, grant No. YSBR-031.  Hang Xiong's research  is partially supported by China Postdoctoral Science Foundation 2021TQ0351 and 2021M693336. We thank Changguo Xiao for his constructive comments to improve the presentation of the paper.


\end{document}